\documentclass[11pt,reqno,a4paper]{amsart}
\usepackage{amsfonts,amsmath,times}
\usepackage{tikz}
\usepackage{lmodern}

\usefont{T1}{cmr}{m}{n}

\definecolor{blau}{rgb}{0,0,0.75} 
\usepackage[pdftex]{hyperref}
\hypersetup{colorlinks,linkcolor=blau,citecolor=blue,urlcolor=blau}

\allowdisplaybreaks

\newtheorem{theorem}{Theorem}
\newtheorem{lemma}{Lemma}
\newtheorem{defi}{Definition}
\newtheorem{prop}{Proposition}
\newtheorem{coroll}{Corollary}
\theoremstyle{definition}
\newtheorem{remark}{Remark}
\newtheorem{example}{Example}
\newtheorem{urn}{Urn}

\newcommand{\fallfak}[2]{\ensuremath{#1^{\underline{#2}}}}

\newcommand{\auffak}[2]{\ensuremath{#1^{\overline{#2}}}}
\newcommand{\Stir}[2]{\genfrac{ \{ }{ \} }{0pt}{}{#1}{#2}}
\newcommand{\stir}[2]{\genfrac{ [ }{ ] }{0pt}{}{#1}{#2}}

\newcommand{\ith}[1]{\ensuremath{ {#1}\textsuperscript{th}}}

\newcommand{\rec}{\textsc{Rect}}
\newcommand{\gport}{\textsc{Gport}}
\newcommand{\dinc}{\ensuremath{d}\textsc{-Inct}}

\newcommand{\Cc}{\ensuremath{\frac{c_2}{c_1}}}

\newcommand\gf{\varphi}

\newcommand{\N}{\ensuremath{\mathbb{N}}}
\newcommand{\R}{\ensuremath{\mathbb{R}}}

\DeclareMathOperator{\MPo}{\text{MPo}}
\DeclareMathOperator{\Po}{\text{Po}}

\newcommand{\law}{\ensuremath{\stackrel{\mathcal{L}}=}}
\newcommand{\claw}{\ensuremath{\xrightarrow{\mathcal{L}}}}

\def\P{{\mathbb {P}}}
\def\E{{\mathbb {E}}}

\newcommand\X{\ensuremath{X_{n,\ell}}}

\author[M.~Kuba]{Markus Kuba}
\address{Markus Kuba\\
Institute of Applied Mathematics and Natural Sciences\\
University of Applied Sciences - Technikum Wien\\
H\"ochst\"adtplatz 5, 1200 Wien} %
\email{kuba@technikum-wien.at}

\author[A.~Panholzer]{Alois Panholzer}
\address{Alois Panholzer\\
Institut f{\"u}r Diskrete Mathematik und Geometrie\\
Technische Universit\"at Wien\\
Wiedner Hauptstr. 8-10/104\\
1040 Wien, Austria} \email{Alois.Panholzer@tuwien.ac.at}

\thanks{The second author was supported by the Austrian Science Foundation FWF, grant P25337-N23.}

\title[Moment sequences and mixed Poisson distributions]{On moment sequences and mixed Poisson distributions}

\begin{document}

\begin{abstract}
In this article we survey properties of mixed Poisson distributions and probabilistic aspects of the Stirling transform: given a non-negative random variable $X$ with moment sequence $(\mu_s)_{s\in\N}$ we determine
a discrete random variable $Y$, whose moment sequence is given by the Stirling transform of the sequence $(\mu_s)_{s\in\N}$, 
and identify the distribution as a mixed Poisson distribution. We discuss properties of this family of distributions and present a simple limit theorem based on expansions of factorial moments instead of power moments.
Moreover, we present several examples of mixed Poisson distributions in the analysis of random discrete structures, unifying and extending earlier results. 
We also add several entirely new results: we analyse triangular urn models, where the initial configuration or the dimension of the urn is not fixed, but may depend on the discrete time $n$. We discuss the branching structure of plane recursive trees and its relation to
table sizes in the Chinese restaurant process. Furthermore, we discuss root isolation procedures in Cayley-trees, a parameter in parking functions, zero contacts in lattice paths consisting of bridges, and a parameter related to cyclic points and trees in graphs of random mappings, all leading to mixed Poisson-Rayleigh distributions. Finally, we indicate how mixed Poisson distributions naturally arise in the critical composition scheme of Analytic Combinatorics.
\end{abstract}

\date{\today}
\maketitle

\section{Introduction}
In combinatorics the \emph{Stirling transform} 
of a given sequence $(a_s)_{s\in\N}$, see~\cite{Sloane,Weisstein}, is the sequence 
$(b_s)_{s\in\N}$, with elements given by
\begin{equation}
\label{MOMSEQ-eqn1}
b_s=\sum_{k=1}^{s}\Stir{s}{k}a_k,\quad s\ge 1.
\end{equation}
The inverse Stirling transform of the sequence $(b_s)_{s\in\N}$
is obtained as follows:
\begin{equation}
\label{MOMSEQ-eqn1b}
a_s=\sum_{k=1}^{s}(-1)^{s-k}\stir{s}{k}b_k,\quad s\ge 1.
\end{equation}
Here $\Stir{s}{k}$ denote the Stirling numbers of the second kind, counting the number of ways to partition a set
of $s$ objects into $k$ non-empty subsets, see~\cite{Stanley} or~\cite{Concrete}, and $\stir{s}{k}$ denotes the unsigned Stirling numbers of the first kind, counting
the number of permutations of $s$ elements with $k$ cycles~\cite{Concrete}. These numbers appear as coefficients in the expansions
\begin{equation}
\label{MOMSEQ-eqn1c}
x^s=\sum_{k=0}^{s}\Stir{s}k \fallfak{x}k,\qquad \fallfak{x}s=\sum_{k=0}^{s}(-1)^{s-k}\stir{s}k x^k,\qquad 
\end{equation}
relating ordinary powers $x^s$ to the so-called falling factorials $\fallfak{x}{s}=x(x-1)\dots(x-(s-1))$, $s\in\N_0$. On the level of exponential generating functions $A(z)=\sum_{s\ge 1}a_s z^s/s!$ and $B(z)=\sum_{s\ge 1}b_s z^s/s!$, the Stirling transform and the relations~\eqref{MOMSEQ-eqn1} and~\eqref{MOMSEQ-eqn1b} turn into 
\begin{equation}
\label{MOMSEQ-eqn1d}
B(z)=A\big(e^{z}-1\big),\quad A(z)=B\big(\log(1+z)\big).
\end{equation}
This definition is readily generalized: given a sequence $(a_s)_{s\in\N}$ the generalized Stirling transform with parameter $\rho>0$ is the sequence 
$(b_s)_{s\in\N}$ with
\begin{equation}
\label{MOMSEQ-eqnNEU1}
b_s=\sum_{k=1}^{s}\rho^k\Stir{s}{k}a_k,\qquad\text{such that } a_s=\frac{1}{\rho^s}\sum_{k=1}^{s}(-1)^{s-k}\stir{s}{k}b_k,\quad s\ge 1.
\end{equation}
On the level of exponential generating functions: $B(z)=A\big(\rho(e^{z}-1)\big)$ and $A(z)=B\big(\log(1+\frac{z}{\rho})\big)$.
The aim of this work is to discuss several probabilistic aspects of a generalized Stirling transform with parameter $\rho>0$ in connection with moment sequences and \emph{mixed Poisson distributions}, pointing out applications in the analysis of random discrete structures. Given a non-negative random variable $X$ with power moments $\E(X^s)=\mu_s\in\R^+$, $s\ge 1$, we study the properties of another random variable $Y$, given its sequence of \emph{factorial moments}
$\E(\fallfak{Y}{s})=\E\big(Y(Y-1)\dots(Y-(s-1))\big)$, which are determined by the moments of $X$,
\begin{equation}
\label{MOMSEQ-eqn2}
\E(\fallfak{Y}{s})=\rho^s\E(X^s)=\rho^s\mu_s,\quad s\ge 1,
\end{equation}
where $\rho>0$ denotes an auxiliary scale parameter. Moreover, we discuss relations between the moment generating functions 
$\psi(z)=\E(e^{zX})$ and $\varphi(z)=\E(e^{zY})$ of $X$ and $Y$, respectively. 

\subsection{Motivation}
Our main motivation to study random variables with a given sequence of factorial moments~\eqref{MOMSEQ-eqn2} stems from the analysis of combinatorial structures. In many cases, amongst others the analysis of inversions in labelled tree families~\cite{PanSeitz2011}, stopping times in urn models~\cite{PanKuCPC,PanKuAdvances}, 
node degrees in increasing trees~\cite{KubPan2007}, block sizes in $k$-Stirling permutations~\cite{PanKuCPC}, descendants in increasing trees~\cite{Desc-KubPan2005}, 
ancestors and descendants in evolving k-tree models~\cite{PanSeitz2012}, pairs of random variables $X$ and $Y$ arise as limiting distributions for certain parameters of interest associated to the combinatorial structures. The random variable $X$ can usually be determined via its (power) moment sequence $(\mu_s)_{s\in\N}$, and the random variable $Y$ in terms of the sequence of factorial moments satisfying relation~\eqref{MOMSEQ-eqn2}. An open problem was to understand in more detail the nature of the random variable $Y$. In~\cite{PanKuAdvances,PanSeitz2011} a few results
in this direction were obtained. The goal of this work is twofold: first, to survey the properties of mixed Poisson distributions, and second to discuss their appearances in combinatorics and the analysis of random discrete structures, complementing existing results; in fact, we will add several entirely new results. It will turn out that the identification of the distribution of $Y$ can be directly solved using mixed Poisson distributions, which are widely used in applied probability theory, see for example~\cite{Ferrari,Karlis,Masse,Neyman,Willmot}. In the analysis of random discrete structures mixed Poisson distributions have been used mainly in the context of Poisson approximation, see e.g.~\cite{Gruebi2005}. In this work we point out the appearance of mixed Poisson distributions as a genuine limiting distribution, and also present closely related phase transitions. In particular, we discuss natural occurrences of mixed Poisson distributions in urn models of a \emph{non-standard nature} - either the size of the urn, or the initial conditions are allowed to depend on the discrete time.

\subsection{Notation and terminology}
We denote with $\R^+$ the non-negative real numbers. Here and throughout this work we use the notation $\fallfak{x}{s}=x(x-1)\dots(x-(s-1))$ for the falling factorials, and $\auffak{x}s=x(x+1)\dots(x+s-1)$ for the rising factorials.\footnote{The notation $\fallfak{x}{s}$ and $\auffak{x}s$ was introduced and popularized by Knuth; alternative notations for the falling factorials 
include the Pochhammer symbol $(x)_s$, which is unfortunately sometimes also used for the rising factorials.}
Moreover, we denote with $\Stir{s}{k}$ the Stirling numbers of the second kind. We use the notation $U \law V$ for the equality in distribution
of random variables $U$ and $V$, and $U_{n} \claw V$ denotes the converge in distribution of a sequence of random variables $U_{n}$ to $V$. 
The indicator variable of the event $A$ is denoted by $\boldsymbol{1}_{A}$.
Throughout this work the term ``convergence of all moments'' of a sequence of random variables refers exclusively to the convergence of all non-negative integer moments. 
Given a formal power series $f(z)=\sum_{n\ge 0} a_n z^n$ we denote with $[z^n]$ the extraction of coefficients operator: $[z^n]f(z)=a_n$. 
Furthermore, we denote with $E_{v}$ the evaluation operator of the variable $v$ to the value $v=1$, and with $D_v$ the differential operator with respect to $v$. 

\subsection{Plan of the paper}
In the next section we state the definition of mixed Poisson distributions and discuss its properties. In Section~\ref{secApplications} we collect several examples from the literature,
unifying and extending earlier results. 
Furthermore, in Section~\ref{MOMSEQtriangular} we present a novel approach to balanced triangular urn models and its relation to mixed Poisson distributions. Section~\ref{RayPoi} is devoted to new results concerning mixed Poisson distributions with Rayleigh mixing distribution;
in particular, we discuss node isolation in Cayley-trees, zero contacts in directed lattice paths, and also cyclic points in random mappings. Finally, in Section~\ref{SecMMPD} we discuss multivariate mixed Poisson distributions. 

\section{Moment sequences and mixed Poisson distributions\label{secMixedPoisson}}
\subsection{Discrete distributions and factorial moments}
In order to obtain a random variable $Y$ with a prescribed sequence of factorial moments, given according to Equation~\eqref{MOMSEQ-eqn2} by $\E(\fallfak{Y}{s})=\rho^s\mu_s$, a first \emph{ansatz} would be the following. 
Let $Y$ denote a discrete random variable with support the non-negative integers, 
and $p(v)$ its probability generating function, $$p(v)=\E(v^{Y})=\sum_{\ell\ge 0}\P\{Y=\ell\}v^{\ell}.$$ 
The factorial moments of $Y$ can be obtained from the probability generating function by repeated differentiation, 
\begin{equation}
\label{MOMSEQFactMomDiff}
\E(\fallfak{Y}{s})=\sum_{\ell\ge 0}\fallfak{\ell}s \, \P\{Y=\ell\}=E_{v}D_v^s p(v),\quad s\ge 0.
\end{equation}
Consequently, we can describe the probability mass function of the random variable $Y$ as follows:
\[
p(v)=\sum_{s\ge 0}\E(\fallfak{Y}{s})\frac{(v-1)^s}{s!} = \sum_{s\ge 0}(v-1)^s \frac{\mu_s\rho^s}{s!} = \sum_{\ell \ge 0} v^{\ell} \sum_{s\ge \ell}\binom{s}{\ell}(-1)^{s-\ell}\frac{\mu_s\rho^s}{s!}.
\]
This implies that
\begin{equation}
\label{MOMSEQmomentexpansion}
\P\{Y=\ell\}=[v^\ell]p(v)=\sum_{s\ge \ell}\binom{s}{\ell}(-1)^{s-\ell}\frac{\mu_s\rho^s}{s!},\quad \ell\ge 0.
\end{equation}
Up to now the calculations have been purely symbolic, no convergence issues have been addressed. In order to put the calculations above on solid grounds, and to identify the distribution, we discuss mixed Poisson distributions 
and their properties in the next subsection.

\subsection{Properties of mixed Poisson distributions}
\begin{defi}
\label{MOMSEQdef1}
Let $X$ denote a non-negative random variable, with cumulative distribution function $\Lambda(.)$,
then the discrete random variable $Y$ with probability mass function given by 
\[
\P\{Y=\ell\}=\frac{1}{\ell!}\int_{\R^{+}} X^{\ell}e^{-X} d\Lambda, \quad \ell\ge 0,
\]
has a mixed Poisson distribution with mixing distribution $X$, in symbol $Y\law \MPo(X)$.
\end{defi}
The boundary case $X\law 0$ leads to a degenerate distribution with all mass concentrated at zero.  A more compact notation for the probability mass function of $Y$ is sometimes used instead of the one given above, namely
$\P\{Y=\ell\}=\frac{1}{\ell!}\E( X^{\ell}e^{-X})$. One often encounters a slightly different definition, which includes a scale parameter $\rho\ge 0$:
\[
\P\{Y=\ell\}=\frac{\rho^\ell}{\ell!}\int_{\R^{+}} X^{\ell}e^{-\rho X} d\Lambda, \quad \ell\ge 0,
\]
or $\P\{Y=\ell\} = \frac{\rho^{\ell}}{\ell!}\E( X^{\ell}e^{-\rho X})$.
This corresponds to a scaling of the mixing distribution, $Y\law \MPo(\rho X)$. Here and throughout this work 
we call $Y\law \MPo(\rho X)$ a mixed Poisson distributed random variable with mixing distribution $X$ and scale parameter $\rho$.

\begin{example}
The ordinary Poisson distribution $Y\law \text{Po}(\rho)$ with parameter $\rho>0$,
\[
\P\{Y=\ell\}=\frac{\rho^\ell}{\ell!}e^{-\rho},\quad \ell\ge 0,
\]
arises as a mixed Poisson distribution with degenerate mixing distribution $X\law 1$.
\end{example}

\begin{example}
The negative binomial distribution $Y\law \text{NegBin}(r,p)$ with parameters $p\in[0,1)$ and $r>0$,
\[
\P\{Y=\ell\}=\binom{\ell+r-1}\ell p^\ell(1-p)^r,\quad \ell\ge 0,
\]
arises as a mixed Poisson distribution with a Gamma mixing distribution $X\law \text{Gamma}(r, \theta)$ scaled by $\rho\ge 0$, such that the parameters $\theta$ and $\rho$ satisfy $\theta\cdot \rho = p/(1-p)$. 
In particular, for $\theta=1$ the parameter $p$ is given by $p=\rho/(1+\rho)$.
A special instance of this class of distributions is the geometric distribution $\text{Geom}(p)=\text{NegBin}(1,p)$. 

Note that a Gamma distributed r.v.\ $X\law \text{Gamma}(r, \theta)$ has the probability density function
\begin{equation*}
  f(x; r,\theta) = \frac{x^{r-1} e^{-\frac{x}{\theta}}}{\theta^{r} \Gamma(r)}, \quad x > 0.
\end{equation*}
\end{example}

\begin{example}
\label{ExampleRayleigh}
A Rayleigh distributed r.v.\ $X \law \text{Rayleigh}(\sigma)$ with parameter $\sigma$ has the probability density function
\begin{equation*}
    f(x;\sigma) = \frac{x}{\sigma^{2}} e^{-\frac{x^{2}}{2\sigma^{2}}}, \quad x \geq 0, 
\end{equation*}
and is fully characterized by its (power) moment sequence:
\begin{equation*}
  \mathbb{E}(X^{s}) = \sigma^{s} \, 2^{\frac{s}{2}} \,\Gamma\left(\frac{s}{2}+1\right).
\end{equation*}
A discrete random variable $Y$ with probability mass function
\[
\P\{Y=\ell\}=\frac{\rho^\ell}{\ell!}\int_0^\infty x^{\ell+1}e^{-\rho x -\frac{x^2}2}dx,\quad \ell\ge 0,
\]
arises as a mixed Poisson distribution  $Y\law \MPo(\rho X)$ with mixing distribution $X\law \text{Rayleigh}(1)$ 
and scale parameter $\rho$. We call $Y$ a Poisson-Rayleigh distribution with parameter $\rho$. 
Note that for $\rho<1$ we can expand $e^{-\rho x}$ and obtain a series representation of $\P\{Y=\ell\}$. 
Another representation valid for all $\rho>0$ can be stated in terms of the incomplete gamma function $\Gamma(s,x) = \int_x^{\infty} t^{s-1}\,e^{-t}\,{\rm d}t$:
\[
\P\{Y=\ell\}=\frac{\rho^\ell}{\ell!}e^{\frac{\rho^2}2}\sum_{i=0}^{\ell+1}\binom{\ell+1}i (-\rho)^{\ell+1-i} \, 2^{\frac{i-1}2} \, \Gamma\big(\frac{i+1}2,\frac{\rho^2}2\big).
\]
\end{example}

\begin{example}
The Neyman Type $A$ Distribution is a discrete probability distribution often used in biology and ecology~\cite{Masse,Neyman}.
It is a mixed Poisson distribution with mixing distribution $X\law \Po(\lambda)$ given by an (ordinary) Poisson distribution with parameter $\lambda$, scaled by $\rho$:
\[
\P\{Y=\ell\}=\frac{\rho^{\ell}}{\ell!}\sum_{m\ge 0}m^{\ell}e^{-\rho m} \Big(e^{-\lambda} \frac{\lambda^m}{m!}\Big)
= \frac{\rho^{\ell}}{\ell!} e^{-\lambda+\lambda e^{-\rho}}\sum_{j=0}^{\ell}\Stir{\ell}{j}\big(\lambda e^{-\rho }\big)^{j}.
\]
\end{example} 
For a very comprehensive list of examples of mixed Poisson distributions we refer the reader to the article of Willmot~\cite{Willmot}.
Since by~\eqref{MOMSEQ-eqn1c} the factorial moments $\E(\fallfak{Y}{s})$ are related to the ordinary moments in terms of the Stirling numbers of the second kind,
the moment sequence of $Y$ is the (scaled) Stirling transform of the moment sequence of $X$. Next we collect similar basic properties of mixed Poisson distributions.

\begin{prop}
\label{MOMSEQPropEigenschaften}
Let $Y\law \MPo(\rho X)$ denote a mixed Poisson distributed random variable with mixing distribution $X$ and scale parameter $\rho>0$. 
\begin{itemize}
	\item[(a)] The factorial moments of $Y$ are given by the scaled power moments of its mixing distribution,	$\E(\fallfak{Y}{s})=\rho^s\E(X^s)$, $s\ge 1$.
	\item[(b)] The power moments of $Y$ and $X$ are related by the generalized Stirling transform with parameter $\rho$, and its inverse, respectively:
\[
\E(Y^s)=\sum_{j=0}^{s}\Stir{s}{j}\rho^j\E(X^j),\quad \E(X^s)=\frac{1}{\rho^s}\sum_{j=0}^{s}(-1)^{s-j}\stir{s}{j}\E(Y^j).
\]
Similarly, the cumulants of $Y$ and $X$ are related by the generalized Stirling transform with parameter $\rho$, and its inverse, respectively.
	\item[(c)] The moment generating functions $\varphi(z)=\E(e^{zY})$ and $\psi(z)=\E(e^{zX})$ are related by the (generalized) Stirling transform of functions and its inverse, respectively:
\begin{equation}
\label{MOMSEQeqn2}
\varphi(z)=\psi\Big(\rho (e^{z}-1)\Big),\qquad \quad \psi(z)=\varphi\Big(\log\big(1+\frac{z}{\rho}\big)\Big).
\end{equation}
	\item[(d)] Let $Y_1\law \MPo(\rho_1X_1)$ and $Y_2\law \MPo(\rho_2X_2)$ denote two independent mixed Poisson distributed random variables. Then, the sum $Y=Y_1\oplus Y_2$ is again mixed Poisson distributed,
\[
 Y \law \MPo(\rho_1X_1\oplus \rho_2X_2).
\]
\end{itemize}
\end{prop}

\smallskip

\begin{proof}
(a) First we derive the factorial moments $\E(\fallfak{Y}{s})=\E\big(Y(Y-1)\dots(Y-s+1)\big)$ of $Y$ by a direct computation:
\begin{equation*}
\begin{split}
\E(\fallfak{Y}{s})&=\sum_{\ell\ge 0}\fallfak{\ell}{s} \, \P\{Y=\ell\}
= \sum_{\ell\ge s}\fallfak{\ell}{s} \, \frac{\rho^\ell}{\ell!}\int_{\R^+} X^{\ell}e^{-\rho X} d\Lambda\\
&=\rho^s\sum_{\ell\ge s}\frac{\rho^{\ell-s}}{(\ell-s)!}\E\big(X^{\ell}e^{-\rho X}\big)
=\rho^s \E\Big(X^s e^{-\rho X}\sum_{\ell\ge 0}\frac{\rho^{\ell}X^\ell}{\ell!}\Big)
=\rho^s\E(X^s).
\end{split}
\end{equation*}
(b) By converting $\fallfak{Y}{s}$ into ordinary powers~\eqref{MOMSEQ-eqn1c} the sequence of ordinary power moments $(\E(Y^s))_{s\in\N}$ of a mixed Poisson distributed random variable $Y$ is given by the Stirling transform of the moments of the mixing distribution in the following way:
\begin{equation}
\label{MOMSEQeqn1}
 \E(Y^s)=\E\Big(\sum_{j=0}^{s}\Stir{s}{j}\fallfak{Y}{j}\Big)=\sum_{j=0}^{s}\Stir{s}{j}\E(\fallfak{Y}{j})
 =\sum_{j=0}^{s}\Stir{s}{j}\rho^j\E(X^j),\quad s\ge 1.
\end{equation}

The result concerning the moment generating function in (c) can be shown similar to~\eqref{MOMSEQ-eqn1d} by directly computing $\E(e^{zY})$, interchanging integration and summation:
\[
\E(e^{zY})= \sum_{\ell\ge 0}e^{z\ell}\P\{Y=\ell\}=\int_{\R^{+}} \sum_{\ell\ge 0} \frac{\big(e^{z}\rho X\big)^{\ell}}{\ell!}e^{-\rho X} d\Lambda
=\int_{\R^{+}}e^{\rho(e^{z}-1) X} d\Lambda.
\]
By definition, the latter expression is exactly $\psi\Big(\rho (e^{z}-1)\Big)$, where $\psi(z)=\E(e^{zX})$ denotes the moment generating function of the mixing distribution $X$. 
If the cumulative distribution function of $X$ is not known, we can compute the moment generating function $\varphi(z)$ of $Y$ utilizing only the moment sequences:
\begin{equation*}
\begin{split}
\varphi(z)&=\E(e^{zY})= \sum_{s\ge 0}\E(Y^s)\frac{z^{s}}{s!}= \sum_{s\ge 0} \sum_{j=0}^{s}\Stir{s}{j}\rho^j \mu_{j}\frac{z^{s}}{s!}
=\sum_{j\ge 0 }\rho^j \mu_{j}\sum_{s\ge j}\Stir{s}{j}\frac{z^{s}}{s!}.
\end{split}
\end{equation*}
Using the bivariate generating function identity of the Stirling numbers of the second kind (see Wilf~\cite{Wilf})
\begin{equation}
\label{MOMSEQWilf}
\sum_{s\ge 0}\sum_{j\ge 0} \Stir{s}{j}\frac{z^{s}}{s!} u^{j} = e^{u(e^{z}-1)},
\end{equation}
we obtain further
\begin{equation*}
\varphi(z)= \sum_{j\ge 0}\mu_{j} \frac{\rho^j(e^{z}-1)^j}{j!}.
\end{equation*}
The latter expression is exactly the Stirling transform of $\psi(z)=\sum_{j\ge 0 }\mu_j\frac{z^j}{j!}$ - in other words, of the moment generating function of $X$ evaluated at $\rho(e^{z}-1)$.
The relation for the cumulants now follows readily from (c), since the cumulant generating functions $k_X(z)$ and $k_Y(z)$ of $X$ and $Y$ are given by
$k_X(z)=\log(\psi(z))$ and $k_Y(z)=\log(\varphi(z))$. For a proof of part (d) we refer the reader to Johnson, Kotz and Kemp~\cite{Johnson1992}.
\end{proof}

In the applied probability literature, see~\cite{Karlis,Willmot}, given $Y\law \MPo(\rho X )$ it is usually assumed that the cumulative distribution function of the mixing distribution of $X$ is known. However, in many cases in the analysis of random discrete structures the mixing distribution $X$ is solely determined by the sequence of moments $\E(X^s)=\mu_s\in\R^+$, $s\ge 1$. Hence, it is beneficial to express the probability mass function of a mixed Poisson distributed random variable solely in terms of the moments of $X$, justifying~\eqref{MOMSEQmomentexpansion}. Note that for specific mixed Poisson distributions different simpler formulas may exist (compare with Corollary~\ref{MOMSEQCorollDimUrns}).

\begin{prop}
\label{MOMSEQthe1}
Let $X$ denote a random variable with moment sequence given by $(\mu_s)_{s\in\N}$ such that $\psi(z)=\E(e^{zX})$ exists in a neighbourhood of zero, including the value $z=-\rho$.
A random variable $Y$ with factorial moments given by $\E(\fallfak{Y}s)=\rho^s\mu_s$ has a mixed Poisson distribution $Y\law \MPo(\rho X)$ with mixing distribution  $X$ and scale parameter $\rho>0$, and the sequence of power moments of $Y$ is the Stirling transform of the moment sequence $(\mu_s)_{s\in\N}$. The probability mass function of $Y$ is given by
\[
\P\{Y=\ell\}=\sum_{s\ge \ell}(-1)^{s-\ell}\binom{s}{\ell}\mu_{s}\frac{\rho^s}{s!}, \quad \ell\ge 0.
\]
\end{prop}

\begin{proof}
By our assumption on the existence of $\psi(z)$ in a neighbourhood of zero, it follows that $\varphi(z)$ is also analytic around $z=0$, and the random variable $Y$ is uniquely determined by its (factorial) moments. 
Consequently, $Y$ has a mixed Poisson distribution. Moreover, the probability mass function of $Y$ is obtained by
\begin{equation}
\begin{split}
\P\{Y=\ell\}&=\frac{\rho^\ell}{\ell!}\E( X^{\ell}e^{-\rho X})=\frac{\rho^\ell}{\ell!}\Big(D_z^{\ell}\psi(z)\Big)\Big|_{z=-\rho}
=\frac{\rho^\ell}{\ell!}\Big(D_z^{\ell}\sum_{s\ge 0 }z^s\frac{\mu_s}{s!}\Big)\Big|_{z=-\rho}\\
&=\frac{\rho^\ell}{\ell!}\sum_{s\ge \ell}(-\rho)^{s-\ell}\frac{\mu_s}{(s-\ell)!}=
\sum_{s\ge \ell}(-1)^{s-\ell}\binom{s}{\ell}\mu_{s}\frac{\rho^s}{s!}.
\end{split}
\end{equation}
Alternatively, the formula for the probability mass function can formally be obtained directly from the definition 
\[
\P\{Y=\ell\}=\frac{\rho^\ell}{\ell!}\int_0^{\infty} X^{\ell}e^{-\rho X} d\Lambda
=\frac{1}{\ell!}\int_0^{\infty}\sum_{s\ge 0}(-1)^{s}\frac{(\rho X)^{s+\ell}}{s!}d\Lambda.
\]
Interchanging summation and integration also leads to the stated result.
\end{proof}

\subsection{The method of moments and basic limit laws}
The method of moments is a classical way of deriving limit laws (see for example Hwang and Neininger~\cite{Hwang2002} and the references therein).
Given a sequence of random variables $(X_{n})_{n\in\N}$ one first derives asymptotic expansions of the power moments; 
assume that the moments satisfy the asymptotic expansion
\begin{equation}
\label{MOMSEQmmexpansion}
\E(X_n^s)=\lambda_n^s \cdot \mu_s\cdot (1 + o(1)),\quad s\ge 1,
\end{equation}
with $\lambda_n$ denoting non-negative scale parameters. Then, one considers the scaled random variables $\frac{X_n}{\lambda_n}$, and tries to prove convergence in distribution of $\frac{X_n}{\lambda_n}$  by using the Fr\'echet-Shohat moment convergence theorem~\cite{Loe1977}:
if the power moments of $\frac{X_n}{\lambda_n}$ converge to the moments $(\mu_s)_{s\in\N}$, and the moment sequence $(\mu_s)_{s\in\N}$ determines a unique non-degenerate distribution, then the random variable $\frac{X_n}{\lambda_n}$
converges in distribution to $X$. A well-known sufficient criterion for the uniqueness of the distribution of $X$ is Carleman's condition: 
the distribution of $X$ is uniquely determined if 
\begin{equation}
\label{MOMSEQmmCarleman}
\sum_{s\ge 1}(\mu_{2s})^{-\frac1{2s}}=\infty.
\end{equation}
Note that~\eqref{MOMSEQmmCarleman} is satisfied, whenever $\E(e^{zX})$ exists in a neighbourhood of zero. We obtain the following result concerning mixed Poisson distributions.
\begin{lemma}[Uniqueness of mixed Poisson distributions]
\label{MOMSEQLemmaUniqueness}
The moments of a mixed Poisson distributed random variable $Y\law \MPo(\rho X)$, with $\rho\ge 0$ and non-negative mixing distribution $X$, satisfy Carleman's criterion if and only if the moments of $X$ do so. 
Moreover, the moment generating function $\psi(z)=\E(e^{zX})$ exists in a neighbourhood of zero, if and only if $\varphi(z)=\E(e^{zY})$ 
exists in a neighbourhood of zero.
\end{lemma}
\begin{proof}
Note first that the second part follows directly from Proposition~\ref{MOMSEQPropEigenschaften} part (c).
Assume now that the moments of $Y$ satisfy Carleman's condition. We observe that the moments $(\mu_s)_{s\in\N}$ of $X$ are bounded by the scaled power moments of $Y$, $\mu_s \le \frac{1}{\rho^s}\E(Y^s)=\sum_{j=0}^{s}\Stir{s}{j}\rho^{s-j}\mu_j$. 
Consequently, the distribution of $X$ is also uniquely determined by its moment sequence:
\[
\sum_{s\ge 1}(\mu_{2s})^{-\frac1{2s}} \ge \rho \sum_{s\ge 1}(\E(Y^{2s}))^{-\frac1{2s}}=\infty.
\]
Conversely, assume that the moments of $X$ satisfy Carleman's condition:
\[
\sum_{s\ge 1}(\mu_{2s})^{-\frac1{2s}}=\infty.
\]
The \ith{s} power moment of $Y$ can be estimated using the \ith{s} factorial moment of $Y$ in the following way
\begin{equation*}
\begin{split}
\E(Y^s)&=\sum_{\ell\ge 0}\ell^s \, \P\{Y=\ell\}=\sum_{\ell= 0}^{2s-1}\ell^s \, \P\{Y=\ell\}+ \sum_{\ell\ge 2s}\ell^s \, \P\{Y=\ell\}\\
&\le (2s)^s\cdot 1 + \sum_{\ell\ge 2s}2^s\fallfak{\ell}s \, \P\{Y=\ell\}\le (2s)^s +2^s \E(\fallfak{Y}{s}).
\end{split}
\end{equation*}
This implies that
\[
\E(Y^s)\le 2^s(s^s + \E\big(\fallfak{Y}{s})\big)
\le 2^s(s^s + \rho^s\mu_s\big)\le  4^{s}(1+\rho)^s \cdot \max\{s^s,\mu_s\}.
\]
Consequently, 
\[
\big(\E(Y^{2s})\big)^{-\frac1{2s}}\ge \frac{1}{4(1+\rho)}  \cdot \min\{\frac{1}{2s},(\mu_{2s})^{-\frac1{2s}}\},
\]
such that
\[
\sum_{s\ge 1}\big(\E(Y^{2s})\big)^{-\frac1{2s}} \ge \frac{1}{4(1+\rho)} \sum_{s\ge 1}\min\{\frac{1}{2s},(\mu_{2s})^{-\frac1{2s}}\}.
\]
By H\"older's inequality, the moments of $X\ge 0$ satisfy for $0<r<s$ the inequality
\[
\E(X^r)\le \big(\E(X^s)\big)^{\frac{r}{s}}.
\]
Hence, for integers $0<r<s$ we have
\[
\mu_{2r}^{\frac1{2r}}\le \mu_{2s}^{\frac1{2s}},
\]
and sequence $(m_s)_{s\in\N}$, defined by $m_s:=\mu_{2s}^{-\frac1{2s}}$, is monotonically decreasing. 
It remains to show that
\begin{equation}
\label{MOMSEQseries1}
\sum_{s\ge 1}\min\{\frac{1}{s},m_s\}=\infty,
\end{equation}
which immediately implies the required result; note that we omitted the additional factor $\frac12$ for the sake of simplicity.
If $m_s$ is bounded away from zero this is immediately true. Hence, in the following we assume that $(m_s)_{s\in\N}$ is a null sequence.
Let $\N=I_1\cup I_2$, with $I_1\cap I_2=\emptyset$, such that for all $s\in I_1$ we have $\frac1{s}\le m_s$, 
and for $s\in I_2$ we have $\frac1{s}> m_s$.
We obtain 
\[
\sum_{s\ge 1}\min\{\frac{1}{s},m_s\}=\sum_{s\in I_1}\frac1s + \sum_{s\in I_2}m_s.
\]
By our initial assumption $\sum_{s\ge 1}m_s=\infty$, equation~\eqref{MOMSEQseries1} is directly satisfied if either $I_1$ or $I_2$ is finite.
Hence, we assume that both sets are infinite. Assume further that $\sum_{s\in I_1}\frac1s $ is finite. 
We can write the set $I_1$ as the disjoint union of 
infinitely many finite length intervals
\[
I_1=\bigcup_{\ell\ge 1}[a_\ell,b_\ell],
\]
with $[a_\ell,b_\ell]:=\{a_\ell,a_\ell+1,\dots,b_\ell\}$ and $a_\ell,b_\ell\in\N$ for all $\ell\in\N$.
If all but finitely many intervals are of length one, such that $a_\ell=b_\ell$, the values $s$ with $\min\{\frac{1}{s},m_s\}=\frac1s$ are essentially isolated.
In this case we note that $\ell\in I_1$ and $\ell-1\in I_2$ and use for $\ell\ge 2$ the inequality
\[
a_\ell\le a_{\ell-1}\le \frac{1}{\ell-1}\le\frac{2}{\ell}.
\]
This implies that also $\sum_{s\in I_1}m_s$ is finite too, such that $\sum_{s\in I_2}m_s$ is infinite.
Finally, we assume that infinitely many intervals are of length greater or equal two. 
By our earlier assumption $\sum_{s\in I_1}\frac1s $ is finite and satisfies
\[
\sum_{s\in I_1}\frac1s= \sum_{\ell\in\N}\sum_{s\in [a_\ell,b_\ell]}\frac1s > \sum_{\ell\in\N}\int_{a_\ell}^{b_\ell}\frac1x dx
=\sum_{\ell\in\N} \ln\Big(\frac{b_\ell}{a_\ell}\Big)>0.
\]
Furthermore, $\ln\Big(\frac{b_\ell}{a_\ell}\Big)<\epsilon$ for all sufficiently large $\ell$, such that $b_\ell <e^{\epsilon}a_\ell$. 
This implies that for $k\in [a_\ell,b_\ell]$ and sufficiently large $\ell$
\[
\frac{m_k}{\frac{1}k}\le \frac{m_{a_\ell}}{\frac{1}{b_\ell}}< \frac{m_{a_\ell}}{\frac{1}{e^{\epsilon}a_\ell}}
\le \frac{m_{a_\ell-1}}{\frac{1}{e^{\epsilon}a_\ell}}
\le e^{\epsilon}\frac{\frac1{a_{\ell-1}}}{\frac{1}{a_\ell}}\le 2e^{\epsilon}.
\]
Hence, $m_k\le \frac{2e^{\epsilon}}{k}$.
Combining this with our previous argument for the essentially isolated values we deduce that $\sum_{s\in I_1}m_s$ is finite too, such that $\sum_{s\in I_2}m_s=\infty$. 
\end{proof}

Concerning random discrete structures one usually encounters non-negative discrete random variables $X_n$. 
As mentioned before in \eqref{MOMSEQFactMomDiff} the factorial moments are readily obtained from the probability generating function $p(z)=\E\big(z^{X_n}\big)$ 
by repeated differentiation: 
\[
\E(\fallfak{X_n}{s})=E_z\,D_z^s\, p(z),\qquad s\ge 1.
\]
In contrast, the ordinary moments require the usage of the so-called theta differential operator $\Theta_z=zD_z$: 
$\E(X_n^{s})=E_z\,\Theta_z^s\, p(z)$. Mixed Poisson distributions and a related phase transition naturally occur if the factorial moments satisfy asymptotic expansions similar to~\eqref{MOMSEQmmexpansion} instead of the power moments. 
\begin{lemma}[Factorial moments and limit laws of mixed Poisson type]
\label{MOMSEQMainLemma}
Let $(X_n)_{n\in\N}$ denote a sequence of random variables, whose factorial moments are asymptotically of mixed Poisson type satisfying for $n$ tending to infinity the asymptotic expansion
\[
\E(\fallfak{X_n}s)=\lambda_n^s \cdot \mu_s\cdot (1 + o(1)),\quad s\ge 1,
\]
with $\mu_s\ge 0$, and $\lambda_n>0$. Furthermore assume that the moment sequence $(\mu_s)_{s\in\N}$ determines a unique distribution 
$X$ satisfying Carleman's condition. Then, the following limit distribution results hold:
\begin{itemize}
\item[(i)] if $\lambda_n\to\infty$, for $n\to\infty$, the random variable $\frac{X_n}{\lambda_n}$ converges in distribution, with convergence of all moments, to $X$. 

\item[(ii)] if $\lambda_n\to\rho \in (0,\infty)$, for $n\to\infty$, the random variable $X_n$ converges in distribution, with convergence of all moments, to a mixed Poisson distributed random variable $Y\law \MPo(\rho X)$.
\end{itemize}
Moreover, the random variable $Y\law \MPo(\rho X)$ converges for $\rho\to\infty$, after scaling, to its mixing distribution $X$: $\frac{Y}{\rho}\claw X$, 
with convergence of all moments.
\end{lemma}

\begin{remark}
It may be possible to unify cases (i) and (ii) to arbitrary sequences $\lambda_n$ by a suitable result 
for the distance between random variables $X_n$ and $Y_n=\MPo(\lambda_n X)$.
\end{remark}

\begin{remark}
The results above complement the standard case when the distribution of $X$ degenerates, $X = 1$. The random variables $X_n$ are then asymptotically Poisson distributed with parameter $\lambda_n$. 
Thus, the distribution of $\frac{X_n}{\lambda_n}$ degenerates for $\lambda_n\to\infty$, since we expect a central limit theorem for $(X_n-\lambda_n)/\sqrt{\lambda_n}$.
It might also be necessary for non-degenerate $X$ to consider centered random variables similar to $X^{\ast}_n=X_n-\lambda_n$, and its (factorial) moments, instead of $X_n$. 
\end{remark}

\begin{remark}
The result above can be strengthened to also include the degenerate case $\lambda_n\to 0$, such that $X_n\claw 0$. 
By Markov's inequality it suffices to prove that $\E(X_n)\to 0$. In order to obtain additional moment convergence one has to show $\E(\fallfak{X_n}s)\to 0$ for every $s\in\N$.
\end{remark}

\begin{remark}[Moment generating functions and limit laws of mixed Poisson type]
Let $\psi(z)=\E(e^{zX})$ denote the moment generating function of $X$. If the moment generating function $\varphi(z)=\E(e^{zX_n})$ satisfies for $n\to\infty$ the asymptotic expansion
\[
\varphi(z)=\psi\left(\lambda_n\big(e^z-1\big)\right)\cdot(1+o(1)),
\] 
then the conclusion of the lemma above - convergence in distribution - still holds, but a priori without moment convergence. On the other hand, 
if the moments of $(\mu_s)_{s\in\N}$ do not determine a unique distribution, one still obtains by the Lemma above convergence of integer moments, but one cannot deduce convergence in distribution. 
\end{remark}

\begin{remark}
\label{MOMSEQRemarkAdditionalPhase}
In the analysis of random discrete structures the random variables $X_n$ often depend on an additional parameter describing or measuring a certain local aspect of the combinatorial structure, such that $X_n=X_{n,j}$. Moreover, the expansion of the factorial moments often depend on this parameter in a crucial way. A quite common situation (see \cite{Desc-KubPan2005,KubPan2007,PanKuAdvances,PanSeitz2012} and also~\cite{FlaDumPuy2006,Janson2006})
is the following dichotomy for the asymptotic expansion of the factorial moments:
\[
\E(\fallfak{X_n}s)=
\begin{cases}
\lambda_{n,0}^s \cdot \mu_{s,j}\cdot (1 + o(1)),\quad s\ge 1, &\quad j\text{ fixed},\\
\lambda_{n,1}^s \cdot \mu_s\cdot (1 + o(1)),\quad s\ge 1, &\quad j\to\infty,
\end{cases}
\]
where $\lambda_{n,0}$ is independent of $j$, but $\lambda_{n,1}=\lambda_{n,1}(j)$ also depends on the growth of this additional parameter $j$ compared to $n$. 
Consequently, one encounters one additional family of limit laws when $j$ is fixed, determined by the moment sequence $(\mu_{s,j})_{s\in\N}$.
Note that in all presented examples the following additional property holds for $s\ge 1$:
\[
\Lambda_j^{s}\mu_{s,j}\to\mu_s,\quad j\to\infty,
\]
where $\Lambda_j$ denotes an additional scale parameter; compare with the Remarks~\ref{MOMSEQExampleDimurnsRemark}, \ref{MOMSEQExampleDescRemark}, and~\ref{MOMSEQExampleTriangularRemark}. 
\end{remark}

\begin{proof}[Proof of Lemma~\ref{MOMSEQMainLemma}]
By~\eqref{MOMSEQ-eqn1c} the power moments of $X_n$ satisfy the following asymptotic expansion
\begin{equation*}
\E(X_n^s)=\sum_{j=0}^{s}\Stir{s}{j}\E(\fallfak{X_n}{j})=\sum_{j=0}^{s}\Stir{s}{j}\lambda_n^j \, \mu_j(1 + o(1))
=\Big(\sum_{j=0}^{s}\Stir{s}{j}\lambda_n^j \, \mu_j\Big)(1 + o(1)).
\end{equation*}
If $\lambda_n\to\infty$ for $n\to\infty$, we obtain further the expansion
\begin{equation*}
\begin{split}
\E(X_n^s)&=\Big(\sum_{j=0}^{s}\Stir{s}{j}\lambda_n^j \, \mu_j\Big)(1 + o(1))=\Big(\Stir{s}{s}\lambda_n^s \, \mu_s + \mathcal{O}(\lambda_n^{s-1})\Big)(1 + o(1))\\
&=\lambda_n^s \, \mu_s + \mathcal{O}(\lambda_n^{s-1})+ o(\lambda_n^{s}).
\end{split}
\end{equation*}
Consequently, the moments of $\frac{X_n}{\lambda_n}$ converge to the moments $\mu_s$ of the mixing distribution.
Since the moments of $X$ satisfy Carleman's condition, this proves by the Fr\'echet-Shohat moment convergence theorem convergence in distribution.

Furthermore, for $\lambda_n\to\rho$ for $n\to\infty$, we directly obtain
\[
\E(X_n^s)=\Big(\sum_{j=0}^{s}\Stir{s}{j}\lambda_n^j \, \mu_j\Big)(1 + o(1))=\sum_{j=0}^{s}\Stir{s}{j}\rho^j \mu_j + o(1).
\]
Consequently, the moments of $X_n$ converge to the moments of a mixed Poisson distributed random variable $Y\law \MPo(\rho X)$, 
which is uniquely determined by its moment sequence, according to Lemma~\ref{MOMSEQLemmaUniqueness}, and our assumption on the moments of $X$.
Finally, an identical argument proves that a mixed Poisson distributed random variable $Y\law \MPo(\rho X)$ converges to its mixing distribution 
for $\rho\to\infty$. 
\end{proof}

\section{Examples and applications\label{secApplications}}
We present several appearances of mixed Poisson distributions in the analysis of random discrete structures, in particular various families of random trees, $k$-Stirling permutations, and urn models. 
We discuss several families of random trees where a mixed Poisson law arises as the limit law of a discrete random variable $X_{n,j}$. The parameter $n\in\N$ usually measures the size of the investigated trees, and $j$ denotes an additional parameter measuring or marking a certain aspect of the combinatorial structure, i.e., a node with a certain label $j$ of interest, often satisfying a natural constraint of the type $1\le j\le n$, see \cite{Desc-KubPan2005,KubPan2007,PanKuCPC,PanSeitz2011}. In the limit $n\to\infty$, with $j=j(n)$, phase transitions were observed according to the relative growth of $j$ with respect to $n$, e.g., $j=1,2,\dots$ being a constant independent of $n$, $j\to\infty$ but with $j=o(n)$, or $j\sim \rho\cdot n$, for fixed $\rho$. As mentioned in the introduction, in this section we will unify and simplify earlier arguments by starting from explicit formulas for the factorial moments occurring in various works. These explicit formulas directly lead to mixed Poisson laws, using Lemmas~\ref{MOMSEQLemmaUniqueness} and~\ref{MOMSEQMainLemma}, and Stirling's formula for the Gamma function
\begin{equation}
\label{MOMSEQStirling}
\Gamma(x)= x^{x - \frac{1}{2}} e^{-x} \sqrt{2\pi}
        \left(1 + \mathcal{O}\big(x^{-1}\big)\right), \qquad \text{for $x\to\infty$}.
\end{equation}
Besides that, whenever possible we give an interpretation of the random variables occurring in terms of urn models.

\subsection{Block sizes in \texorpdfstring{$k$}{k}-Stirling permutations\label{MOMSEQExampleBlocks}}
Stirling permutations were defined by Gessel and Stanley \cite{GessStan1978}. A Stirling permutation is a permutation of
the multiset $\{1, 1, 2, 2, \dots , n, n\}$ such that, for each $i$, $1\le i \le n$, the elements occurring between the two
occurrences of $i$ are larger than $i$. E.g., $1122$, $1221$ and $2211$ are Stirling permutations, whereas the permutations $1212$ and $2112$ of $\{1, 1, 2, 2\}$ aren't. The name of these combinatorial objects is due to relations with the Stirling numbers, see \cite{GessStan1978} for details and \cite{Koganov} for bijections with certain tree families.
A straightforward generalization of Stirling permutations is to consider permutations of a more general multiset $\{1^{k}, 2^{k}, \dots, n^{k}\}$, 
with $k\in\N$ (we use in this context $j^{\ell} := \underbrace{j, \dots, j}_{\ell}$, for $\ell \ge 1$), such that for each
$i$, $1\le i \le n$, the elements occurring between two occurrences of
$i$ are at least $i$. Such restricted permutations on the multiset $\{1^{k}, 2^{k}, \dots, n^{k}\}$ are called $k$-Stirling permutations of order $n$; they have already been considered by Brenti \cite{Brenti1989,Brenti1998}, Park~\cite{Park1994b,Park1994a,Park1994c}, and Janson et al.~\cite{Janson2008,Janson2011}. 
These $k$-Stirling permutations can be generated in a sequential manner: we start with $1^k=1\dots 1$ and insert the string $(n+1)^k$ at any position (anywhere, including first or last) in a given $k$-Stirling permutation of $\{1^{k}, 2^{k}, \dots, n^{k}\}$, $n\ge 1$. In the case $k=3$, we have for example one permutation of order $1$: $111$; four
permutations of order $2$: $111222$, $112221$, $122211$, $222111$; etc. 

\smallskip

A \emph{block} in a $k$-Stirling permutation $\sigma=\sigma_1\dotsm \sigma_s$ is a substring
$\sigma_a \dotsm \sigma_b$, with $\sigma_a=\sigma_b$, that is maximal, i.e., which is not contained in any larger
such substring. There is obviously at most one block for every $j \in \{1, 2, \dots, n\}$,
extending from the first occurrence of $j$ to the last one; we say that
$j$ forms a block if this substring is indeed a block, i.e., when it
is not contained in a string $j' \dotsm j'$, for some $j'<j$.
It can be shown easily by induction that any $k$-Stirling permutation has a unique
decomposition as a sequence of its blocks. For example, the 3-Stirling permutation $\sigma=112233321445554777666$, has 
block decomposition\newline
$[112233321][445554][777][666]$.

\smallskip

Of course, the size of a block in a $k$-Stirling permutation is always a multiple of $k$. 
The number of blocks of size $k\cdot \ell$ in a random $k$-Stirling permutation of order $n$ was studied in~\cite{PanKuCPC}.
There, a simple exact expression for the factorial moments was derived:
\begin{equation*}
   \E(\fallfak{\X}{s})= \frac{s!}{(k\ell)^s}\binom{\ell-1-\frac1k}{\ell-1}^{s} \cdot
   \frac{\binom{n-\ell s+\frac{s+1}{k}-1}{n-\ell s}}{\binom{n-1+\frac1k}{n}}.
\end{equation*}

Depending on the growth of $\ell=\ell(n)$ as $n \to \infty$, 
two random variables $X$ and $Y$ arose as limiting distributions of $X_{n,\ell}$. The random variable $X$ 
with moment sequence
\begin{equation}
\label{ExampleBlocks0}
\E(X^s)=(s+1)!\frac{\Gamma(1+\frac1k)}{\Gamma(1+\frac{s+1}k)}
\end{equation}
could be characterized using observations by Janson et al.~\cite{Janson2011}, and Janson~\cite{Janson2006}. 
It has a density function $f(x)$ that can be written as 
\begin{equation}
\label{ExampleBlocks1}
f(x)=\frac{\Gamma(\frac1k)}{\pi}\sum_{j=1}^\infty
 (-1)^{j-1} \frac{\Gamma(\frac jk+1)\sin\frac{j\pi}k}{j!}\,x^{j}, \qquad \text{for $x>0$}.
\end{equation}
However, the characterization of the random variable $Y$ was incomplete, only the (factorial) moments were known:
\begin{equation}
\label{ExampleBlocks2}
\E(\fallfak{Y}s)=(s+1)!\frac{\Gamma(1+\frac1k)}{\Gamma(1+\frac{s+1}k)}\rho^s, \quad s\ge 1.
\end{equation}
Using Lemma~\ref{MOMSEQMainLemma}, we can fill this gap, extending the results of~\cite{PanKuCPC}.

\begin{coroll}
\label{MOMSEQCorollExampleBlocks}
The factorial moments of the random variable $X_{n,\ell}$, counting the number of blocks of size $k\cdot \ell$ in a random k-Stirling permutation of order $n$, are for $n\to\infty$ asymptotically of mixed Poisson type,
with mixing distribution $X$, determined by its moments and density given by~\eqref{ExampleBlocks0} and~\eqref{ExampleBlocks1} and scale parameter $\lambda_{n,\ell}=\frac1{k\ell}\binom{\ell-1-\frac1k}{\ell-1}n^{\frac1k}$:
\[
\E(\fallfak{X_{n,\ell}}{s})=\lambda_{n,\ell}^s(s+1)!\frac{\Gamma(1+\frac1k)}{\Gamma(1+\frac{s+1}k)}(1+o(1)),
\]

\begin{itemize}
\item[(i)] for $\ell=\ell(n)$, such that $\lambda_{n,\ell}\to\infty$, the random variable $\frac{X_{n,\ell}}{\lambda_{n,\ell}}$ converges in distribution, with convergence of all moments, to $X$. 

\item[(ii)] for $\ell=\ell(n)$, such that $\lambda_{n,\ell}\to\rho\in(0,\infty)$, the random variable $X_{n,\ell}$ converges in distribution, with convergence of all moments, to a mixed Poisson distributed random variable $Y\law \MPo(\rho X)$. Its probability mass function is given by
\[
\P\{Y=i\}=\sum_{s\ge i}\binom{s}{i}(-1)^{s-i}\rho^{s}\frac{\Gamma(1+\frac1k)}{\Gamma(1+\frac{s+1}k)}, \quad i\ge 0.
\]
\end{itemize}
Moreover, for $\rho\to\infty$, the random variable $Y/\rho$ converges in distribution to $X$, with convergence of all moments.
\end{coroll}

\smallskip

The result above can also be interpreted in terms of a suitable urn model. First we recall the definition of P\'olya-Eggenberger urn models. We start with an urn containing $n$ white balls and $m$ black balls.
The evolution of the urn occurs in discrete time steps. At every step
a ball is drawn at random from the urn. The color of the ball is
inspected and then the ball is returned to the urn. According to
the observed color of the ball there are added/removed balls due to the
following rules. If a white ball has been drawn, we put into the urn
$\alpha$ white balls and $\beta$ black balls, but if a black ball has been drawn, we put into the urn $\gamma$ white balls and $\delta$ black balls.
The values $\alpha, \beta, \gamma, \delta \in \mathbb{Z}$ are fixed integer values and the
urn model is specified by the $2\times 2$ ball replacement matrix $M = \bigl(\begin{smallmatrix} \alpha & \beta \\ \gamma & \delta\end{smallmatrix}\bigr)$.
This definition readily extends to higher dimensions, leading to $r\times r$ ball replacement matrices, if balls of $r$ different colours are involved.
Note that we may consider $\alpha, \beta, \gamma, \delta \in \mathbb{R}$, defining the urn process as a Markov process; see Remark~1.11 of Janson~\cite{Janson2006}.
One usually assumes that the urns are tenable: the process of drawing and adding/removing balls can be continued ad infinitum, 
never having to remove balls which are not present in the urn. Starting with $W_0=w_0$ white balls and $B_0=b_0$ black balls, one is then interested in the composition $(W_n,B_n)$ of the urn after $n$ draws. 
For a few recent results we refer the reader to~\cite{bai2002,FlaDumPuy2006,FlaGabPek2005,Janson2004,Janson2006,Pou2008}. 

In order to describe the growth of the r.v.\ $X_{n,i}$, $1 \le i \le \ell$, by means of a P\'{o}lya-Eggenberger urn model, we consider the simple growth process of random $k$-Stirling permutations: in order to generate a random $k$-Stirling permutation of order $n+1$, we select uniformly at random a $k$-Stirling permutation of order $n$ and insert the substring $(n+1)^{k}$ uniformly at random at one of the $kn+1$ insertion positions. In the urn model description, each ball in the urn will correspond to an insertion position in the $k$-Stirling permutation. We will require $\ell+2$ different colours of balls.
Balls of colours $i$, with $1 \le i \le \ell$, correspond to insertion positions within blocks of size $k i$, balls of colour $\ell+1$ correspond to insertion positions within blocks of size $\ge k(\ell+1)$, and balls of colour $0$ correspond to insertion positions between two consecutive blocks, or before the first or after the last block. When inserting a new substring the following situations can occur, which then describe the evolution of the urn:
\begin{enumerate}
\item[$(i)$] When inserting the string $(n+1)^{k}$ into a block of size $k i$, $1 \le i \le \ell$, then this block changes to a block of size $k (i+1)$. In the urn model this means that if a ball of colour $i$ is drawn, $ki-1$ balls of colour $i$ have to be removed and $k(i+1)-1$ balls of colour $i+1$ have to be added.
\item[$(ii)$] When inserting the string $(n+1)^{k}$ into a block of size $\ge k (\ell+1)$, then it remains a block of size $\ge k (\ell+1)$, but its size increases by $k$. In the urn model this means that if a ball of colour $\ell+1$ is drawn, $k$ balls of colour $\ell+1$ have to be added.
\item[$(iii)$] When inserting the string $(n+1)^{k}$ between two consecutive blocks, or before the first or after the last block, then a new block of size $k$ appears; furthermore an additional insertion position between two consecutive blocks occurs. In the urn model this means that if a ball of colour $0$ is drawn, $k-1$ balls of colour $1$ and one ball of colour $0$ have to be added.
\end{enumerate}
The initial $k$-Stirling permutation $1^{k}$ contains one block of size $k \cdot 1$ with $k-1$ insertion positions; furthermore there are one insertion position before the first and one insertion position after the last block, which describes the initial configuration of the urn.
Thus the following urn model description follows immediately.
\begin{urn}
Consider a balanced urn (i.e., each row of the ball replacement matrix has the same row sum) with balls of $\ell+2$ colours and let the random vector $(U_{n,0},\dots,U_{n,\ell+1})$ count the number of balls of each colour at time $n$ with the $(\ell+2) \times (\ell+2)$ ball replacement matrix $M$ given by
\begin{equation*}
    M = \left(
    \begin{smallmatrix}
        1 & k-1 & 0 & \cdots & 0 & 0 & 0& 0 \\[-1ex]
        0 & -(k-1) & 2k-1 & \ddots & \ddots & \ddots & 0& 0 \\[-1ex]
        0 & 0 & -(2k-1) & 3k-1 & \ddots & \ddots & 0 & 0\\[-1ex]
        \vdots & \ddots & \ddots & \ddots & \ddots &
        \ddots & \vdots & 0\\[-1ex]
        \vdots & \ddots & \ddots & \ddots & \ddots &
        \ddots & \vdots & 0\\[-1ex]
                0 & \ddots & \ddots & \ddots & 0 & -((\ell-1)k-1) & \ell k-1& 0\\[-1ex]
        0 & \ddots & \ddots & \ddots & 0 & 0 & -(\ell k-1) & (\ell+1)k-1\\
        0 & 0 & 0 & \cdots & 0 & 0 & 0& k
    \end{smallmatrix}
    \right).
\end{equation*}
The initial configuration of the urn (it is here convenient to start at time $1$) is given by
$(U_{1,0},\dots,U_{1,\ell+1})=(2,k-1,0,\dots,0)$. It holds that the random variables $U_{n,i}$, with $1 \le i \le \ell$, described by the urn model are related to the random variables 
$X_{n,i}$, $1 \le i \le \ell$, which count the number of blocks of size $k i$ in a random $k$-Stirling permutation of order $n$, as follows:
\begin{equation*}
   U_{n,i} = (ki-1) X_{n,i}, \quad 1\le i\le \ell.
\end{equation*}
\end{urn}
By Theorem~\ref{MOMSEQCorollExampleBlocks} and the results of~\cite{PanKuCPC} this implies that the random variables $U_{n,i}$ occurring in the urn model undergo a phase transition according to the growth of $\ell$ with respect to $n$, from continuous to discrete, where the moments of the appearing random variables $X$ and $Y$ are related by the Stirling transform.

\subsection{Diminishing P\'olya-Eggenberger urn models\label{MOMSEQExampleDimurns}}
A classical example of a non-tenable urn model is the sampling without replace\-ment urn with ball replacement matrix given by $\left(\begin{smallmatrix} -1 & 0 \\ 0 & -1\end{smallmatrix}\right)$. 
The process of drawing and replacing balls ends after $n+m$ steps, starting with $n$ white and $m$ black balls. 
Here, one is interested in the number of white balls remaining in the urn, after all black balls have been drawn. Several urn models of a similar non-tenable nature have recently received some attention under the name diminishing urn models,
see~\cite{PanKuAdvances} and the references therein.

\begin{urn}
Consider a possibly unbalanced generalized sampling without replacement urn model with ball replacement matrix 
$$M = \left(\begin{matrix} -\alpha  & 0 \\ 0 & -\delta\end{matrix}\right),\quad a,\delta\in\N.$$ 
The initial configuration of the urn consists of $\alpha \cdot n$ white balls and $\delta\cdot m$ black balls.
The random variable $X_{\delta m,\alpha n}$ counts the number of white balls remaining in the urn, after all black balls have been drawn.
\end{urn}

It was shown in~\cite{PanKuAdvances} that the factorial moments of the random variable $\hat{X}_{\delta m,\alpha n}=X_{\delta m,\alpha n}/\alpha $ are given by
\begin{equation*}
\E\Big(\fallfak{\hat{X}_{\delta m,\alpha n}}{s}\Big)=\frac{\fallfak{n}{s}}{\binom{m+\frac{\alpha s}{\delta }}{m}},\quad s\ge 1.
\end{equation*}
Moreover, a random variable $Y$ arises in the limit, whose factorial moments are given by
\begin{equation}
\label{ExampleDimUrns1}
\E(\fallfak{Y}s)=\rho^{s} \, \Gamma\big(1+\frac{\alpha s}{\delta }\big),\quad s\ge 1.
\end{equation}
Using a special case of Theorem~\ref{MOMSEQthe1} it was shown that $Y$ has a discrete distribution. However, the result of~\cite{PanKuAdvances}
contains a small gap: the moments $(\Gamma(1+\frac{\alpha s}{\delta }))_{s\in\N}$ only determine a unique distribution for $\alpha /\delta\le 2$, see~\cite{Gut2002}. Hence, only in this case the (factorial) moments of $Y$ determine a unique distribution. Since a Weibull distributed random variable $X\law \text{W}_{\delta /\alpha ,1}$, with shape parameter $\frac{\delta}{\alpha }$, scale parameter $1$, and density $f(t)=\frac{\delta }{\alpha }t^{\frac{\delta }{\alpha }-1}e^{-t^{\frac{\delta }{\alpha }}}$, $t\ge 0$, has moments $\E(X^s)=\Gamma(1+\frac{\alpha s}{\delta })$, we obtain the following characterization of $Y$,
extending the result of~\cite{PanKuAdvances}.

\begin{coroll}
\label{MOMSEQCorollDimUrns}
The random variable $\hat{X}_{\delta m,\alpha n}$, counting the number of white balls remaining in the urn, after all black balls have been drawn
in a generalized sampling without replacement urn, starting with $\alpha \cdot n$ white balls and $\delta \cdot m$ black balls, 
has for $\min\{n,m\}\to\infty$ factorial moments of mixed Poisson type with a Weibull mixing distribution $X\law \text{W}_{\delta /\alpha ,1}$, 
and scale parameter $\lambda_{n,m}=\frac{n}{m^{\frac{\alpha }{\delta }}}$:
\[
\E\Big(\fallfak{\hat{X}_{\delta m,\alpha n}}{s}\Big)=\lambda_{n,m}^s \Gamma\big(1+\frac{\alpha s}{\delta }\big)(1+o(1)).
\]
Assume that $\alpha/\delta \le 2$:
\begin{itemize}
\item[(i)] for $\lambda_{n,m}\to\infty$, the random variable $\frac{\hat{X}_{\delta m,\alpha n}}{\lambda_{n,m}}$ converges in distribution, with convergence of all moments, to $X$. 

\item[(ii)] for $\lambda_{n,m}\to\rho\in(0,\infty)$, the random variable $\hat{X}_{\delta m,\alpha n}$ converges in distribution, with convergence of all moments, to a mixed Poisson distributed random variable $Y\law \MPo(\rho X)$. Its probability mass function is given as follows:
\[
\P\{Y=\ell\}=\begin{cases}
\displaystyle
{
\sum_{j\ge \ell}(-1)^{j-\ell}\binom{j}{\ell}\rho^j \frac{\Gamma(1+\frac{\alpha j}{\delta})}{j!}
}
,&\quad\text{for }\frac{\alpha }{\delta }<1,\\
\displaystyle
{
\frac{1}{1+\rho}\Big(\frac{\rho}{1+\rho}\Big)^\ell
}
,&\quad\text{for }\frac{\alpha }{\delta }=1,\\

\displaystyle
{
\frac{\delta }{\alpha }\sum_{j\ge 0}(-1)^{j}\binom{j+\ell}{\ell}\rho^{-\frac{\delta}{\alpha }(j+1)} \frac{\Gamma(\frac{\delta(j+1)}{\alpha }+\ell)}{(j+\ell)!}
}
,&\quad\text{for }\frac{\alpha }{\delta }>1.
\end{cases}
\]
\end{itemize}
\end{coroll}

\begin{remark}
\label{MOMSEQExampleDimurnsRemark}
As shown in~\cite{PanKuAdvances}, for fixed $m$ the random variable $\hat{X}_{\delta m,\alpha n}/n$ converges to the power $B^{\frac{1}{\alpha}}$ of a beta-distributed random variable $B$, with moments $\E(B^s)=1/\binom{m+\frac{\alpha s}{\delta}}{m}$.
The Weibull mixing distribution $X\law \text{W}_{\delta /\alpha ,1}$ can be recovered by considering the limit $m\to\infty$ of $B=B_m$:
\[
m^{\frac{\alpha}{\delta}}B_m\to X, \quad \text{for }m\to\infty,
\]
with convergence of all moments. Note that the results above can be extended to all $\alpha ,\delta \in\N$; however, for $\alpha/\delta\ge 2$ the method of moments cannot be used anymore. Instead, has to directly analyse the probability generating function $h_{n,m}(v)$, which can be derived using stochastic processes~\cite{PanKuAofA2012}.
\end{remark}

\begin{proof}[Proof of Corollary~\ref{MOMSEQCorollDimUrns}]
According to the definition of a mixed Poisson distributed random variable $Y\law\MPo(\rho X)$, it has factorial moments given by~\eqref{ExampleDimUrns1}.
In order to derive the integral-free series representation we proceed as follows. In the first case $\alpha /\delta <1$ we can directly use Theorem~\ref{MOMSEQthe1}, since the moment generating function of the mixing Weibull distribution $X$ exists at $-\rho$. In the remaining cases $\alpha /\delta \ge 1$ we use the definition and the density function of the Weibull distribution to get first
\[
\P\{Y=\ell\}=\frac{\rho^\ell}{\ell!}\int_0^{\infty}\frac{\delta }{\alpha }t^{\frac{\delta }{\alpha }+\ell-1}\exp\left(-t^{\frac{\delta }{\alpha }}-\rho t\right)dt.
\]
The case $\alpha /\delta =1$ readily leads to the stated geometric distribution after using the obvious simplification
\[
\frac{\delta}{\alpha}t^{\frac{\delta }{\alpha }+\ell-1}\exp\left(-t^{\frac{\delta }{\alpha }}-\rho t\right)=t^{\ell}e^{-t(\rho+1)}.
\]
Finally, for $\alpha /\delta >1$ we expand $e^{-t^{\frac{\delta }{\alpha }}}=\sum_{j\ge 0 }(-1)^j\frac{t^{\frac{j\delta }{\alpha }}}{j!}$ and obtain
\[
\P\{Y=\ell\}=\frac{\delta }{\alpha \cdot \ell!}\sum_{j\ge 0 }\frac{(-1)^j}{j!}\int_0^{\infty}t^{\frac{(j+1)\delta }{\alpha }+\ell-1}e^{-\rho t}dt.
\] 
The Gamma-function type integrals are readily evaluated and the stated result follows.
\end{proof}

\subsection{Descendants in increasing trees.\label{MOMSEQExampleDesc}}
Increasing trees are labelled trees, where the nodes of a tree of size $n$ are labelled by distinct integers
of the set $\{1,\dots, n\}$ in such a way that each sequence of labels along any branch starting at the root
is increasing. They have been introduce by Bergeron et al.~\cite{BerFlaSal1992}, and can be described combinatorially as follows.
Given a so-called degree-weight sequence $(\varphi_{k})_{k \ge 0}$, the corresponding degree-weight generating function $\varphi(t)$ is defined by $\varphi(t) := \sum_{k \ge 0} \varphi_{k} t^{k}$.
The simple family of increasing trees $\mathcal{T}$ associated with a degree-weight generating function $\varphi(t)$, can be described
by the formal recursive equation 
\begin{equation}
   \label{eqnz0}
   \mathcal{T} = \bigcirc\hspace*{-0.75em}\text{\small{$1$}}\hspace*{0.4em}
   \times \Big(\varphi_{0} \cdot \{\epsilon\} \; \dot{\cup} \;
   \varphi_{1} \cdot \mathcal{T} \; \dot{\cup} \; \varphi_{2} \cdot
   \mathcal{T} \ast \mathcal{T} \; \dot{\cup} \; \varphi_{3} \cdot
   \mathcal{T} \ast \mathcal{T} \ast \mathcal{T} \; \dot{\cup} \; \cdots \Big)
   = \bigcirc\hspace*{-0.75em}\text{\small{$1$}}\hspace*{0.3em} \times \varphi(\mathcal{T}),
\end{equation}
where $\bigcirc\hspace*{-0.75em}\text{\small{$1$}}\hspace*{0.4em}$ denotes the node
labelled by $1$, $\times$ the Cartesian product, $\dot{\cup}$ the disjoint union, $\ast$ the partition product for labelled
objects, and $\varphi(\mathcal{T})$ the substituted structure (see, e.g., the books \cite{FlaSed2009,VitFla1990}). 
Note that the elements of $\mathcal{T}$ are increasing plane trees, and that such a tree of size $n$, whose nodes have outdegrees $d_1,\dots,d_n$, has the weight $\prod_{i=1}^{n}\gf_{d_i}$. When speaking about a random tree of size $n$ from the family $\mathcal{T}$, we always use the random tree model for weighted trees, i.e., we assume that each tree of size $n$ from $\mathcal{T}$ can be chosen with a probability proportional to its weight.

\smallskip

Let $T_n$ be the total weight of all trees from $\mathcal{T}$ of size $n$.
It follows from \eqref{eqnz0} that the exponential generating function
$T(z) := \sum_{n \ge 1} T_{n} \frac{z^{n}}{n!}$ of the total weights
satisfies the autonomous first order differential equation
\begin{equation}
   \label{eqnz1}
   T'(z) = \varphi\big(T(z)\big), \quad T(0)=0.
\end{equation}

From now on we consider tree families $\mathcal{T}$ having degree-weights of one of the
following three forms, as studied by~\cite{PanPro2005+},
where we use the abbreviations \rec{} for recursive trees, 
\gport{} for generalized plane recursive trees (also called generalized plane-oriented recursive trees), and \dinc{} for $d$-ary increasing trees.

\begin{equation}\label{pp}
\varphi(t) = \begin{cases}
				e^{c_1 t}, & \text{for} \enspace c_1>0,
\quad\enspace\text{ for }\text{\rec,}\\
		\frac{\varphi_{0}}
      {(1 + \frac{c_{2} t}{\varphi_{0}})^{-\frac{c_{1}}{c_{2}}-1}}, 
			& \text{for} \enspace \varphi_{0} > 0, \; 0 < -c_{2} <
       c_{1},
\quad\enspace\text{ for }\text{\gport,}\\
       \varphi_{0}
      \Big(1 + \frac{c_{2} t}{\varphi_{0}}\Big)^{d},
			& \text{for} \enspace \varphi_{0},c_2 > 0,    \; d :=
	  \frac{c_{1}}{c_{2}}+1 \in \N\setminus\{1\},
\enspace \text{for \dinc}.
      \end{cases}
\end{equation}
Consequently, by solving \eqref{eqnz1}, we obtain the exponential generating function $T(z)$,
\begin{equation}\label{t}
T(z)=
     \begin{cases}
    \log\Big(\frac1{1-c_1 z}\Big), & \text{for \rec},\\
   \frac{\varphi_{0}}{c_{2}} \Big(\frac{1}{(1-c_{1} z)^{\frac{c_{2}}{c_{1}}}} -1 \Big), & \text{for \gport},\\
    \frac{\varphi_{0}}{c_{2}} \Big(\frac{1}{(1-(d-1) c_{2} z)^{\frac{1}{d-1}}} - 1 \Big), & \text{for \dinc},
    \end{cases}
\end{equation}
and the total weights $T_n$,
\begin{equation}\label{tn}
T_{n} = \varphi_{0} c_{1}^{n-1} (n-1)! \binom{n-1+\frac{c_{2}}{c_{1}}}{n-1}.
\end{equation}
Note that changing $\gf_k$ to $ab^k\gf_k$ for some positive constants
$a$ and $b$ will affect the weights of all trees of a given size $n$
by the same factor $a^n b^{n-1}$, which does not affect the
distribution of a random tree from the family. Hence, when considering
random trees from these three classes, $\gf_0$ is irrelevant and $c_1$
and $c_2$ are relevant only through the ratio $c_1/c_2$. (We may thus,
if we like, normalize $\gf_0=1$ and either $c_1$ or $|c_2|$, but not both.)
It is convenient to set $c_1=1$ for (random) recursive trees, to use the parameter $\alpha :=-1-\frac{c_{1}}{c_{2}} > 0$ 
for (random) generalized plane recursive trees, and
$d :=\frac{c_{1}}{c_{2}}+1 \in{2,3,\dots}$ for (random) $d$-ary
increasing trees, i.e., it suffices to consider the degree-weight generating functions
\begin{equation}\label{pp_simplified}
\varphi(t) = \begin{cases}
				e^{t}, & \enspace \text{for \rec},\\
		\frac{1}{(1-t)^{\alpha}}, & \enspace \text{with $\alpha > 0$}, \enspace \text{for \gport},\\
      (1+t)^{d}, & \enspace \text{with $d = 2, 3, 4, \dots$}, \enspace \text{for \dinc}.
      \end{cases}
\end{equation}
\begin{figure}[!htb]
\centering
\includegraphics[scale=0.7]{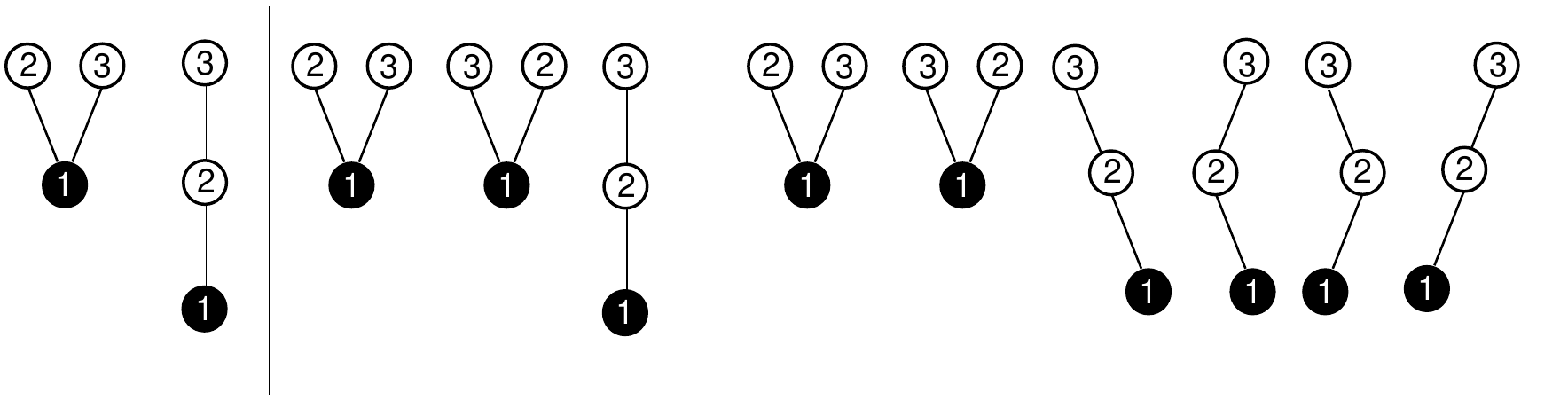}
\caption{Two recursive trees of size three - $\varphi(t)=e^t$ (no left-to-right order); three plane recursive trees of size three - $\varphi(t)=\frac1{1-t}$; six binary increasing trees - $\varphi(t)=(1+t)^2$.}
\label{fig:IncTrees}
\end{figure}
\smallskip

As shown by Panholzer and Prodinger~\cite{PanPro2005+}, random trees
in the three classes of families given in~\eqref{pp_simplified}
can be generated by an evolution process in the following way. 
The process, evolving in discrete time, starts with the root labelled by $1$.
At step $i+1$ the node with label $i+1$ is attached to any previous
node $v$ (with outdegree $d(v)$) of the already grown tree of size
$i$ with probabilities $p(v)$ given by
\begin{equation}\label{pv}
   p(v)=
   \begin{cases}
   \frac{1}{i}, & \text{for \rec},\\
   \frac{d(v)+\alpha}{(\alpha+1)i-1}, & \text{for \gport},\\
   \frac{d-d(v)}{(d-1)i+1}, & \text{for \dinc}.
   \end{cases}
\end{equation}
Moreover, it has been shown in \cite{PanPro2005+} that only random trees from simple families of trees $\mathcal{T}$ given in \eqref{pp} can be generated by such a tree evolution process, i.e., only for these tree families exist suitable attachment probabilities $p(v)$.

\bigskip

Let $D_{n,j}$ denote the random variable, counting the number of descendants, i.e., the size of the subtree rooted at node $j$, of a specific node $j$, with $1 \le j \le n$, in a tree of size $n$. 
In~\cite{Desc-KubPan2005} this random variable has been studied for the three tree families mentioned beforehand using a generating functions approach.  
In the following we collect and somewhat simplify these earlier results.
One obtains a simple exact formula for the factorial moments of $\hat{D}_{n,j}=D_{n,j}-1$ directly from the results of~\cite{Desc-KubPan2005}:
\begin{equation*}
   \E(\fallfak{\hat{D}_{n,j}}s) = s! \frac{\binom{n-j}{s} \binom{s+\frac{c_{2}}{c_{1}}}{s}}{\binom{j-1+\frac{c_{2}}{c_{1}}+s}{s}},
\end{equation*}
with $c_1,c_2$ as given in~\eqref{pp}. Hence, by using Stirling's formula for the Gamma function~\eqref{MOMSEQStirling}, for $n\to\infty$ and $j=j(n)\to \infty$, the factorial moments of $\hat{D}_{n,j}$ are of mixed Poisson type and Lemma~\ref{MOMSEQMainLemma} can be applied.
\begin{coroll}
\label{MOMSEQCorollDesc}
The random variable $\hat{D}_{n,j}$, counting the number of descendants minus one of node $j$ in a random increasing tree of size $n$, has for $n\to\infty$ and $j=j(n)\to \infty$, 
factorial moments of mixed Poisson type with a Gamma mixing distribution $X\law\gamma(1,1+\Cc)$, 
and scale parameter $\lambda_{n,j}=\frac{n-j}{j}$:
\[
\E(\fallfak{\hat{D}_{n,j}}s)= \lambda_{n,j}^s\frac{\Gamma(s+1+\Cc)}{\Gamma(1+\Cc)}(1+o(1)).
\]
\begin{itemize}
\item[(i)] for $\lambda_{n,j}\to\infty$, the random variable $\frac{\hat{D}_{n,j}}{\lambda_{n,j}}$ converges in distribution, with convergence of all moments, to $X$. 

\item[(ii)] for $\lambda_{n,j}\to\rho\in(0,\infty)$, the random variable $\hat{D}_{n,j}$ converges in distribution, with convergence of all moments, to a mixed Poisson distributed random variable $Y\law \MPo(\rho X)$, 
which has a negative binomial distribution. 
\end{itemize}
\end{coroll}

\begin{remark}
\label{MOMSEQExampleDescRemark}
Note that for fixed $n$ the random variable $\hat{D}_{n,j}/n\claw B_j$, where $B_j\law \beta(1+\Cc,j-1)$, is asymptotically beta-distributed (see~\cite{Desc-KubPan2005}). 
One readily recovers the mixing distribution $X$ from $B_j$ by taking the limit $j\to\infty$, using a well known result for beta-distributed random variables:
\[
j B_j\to X, \quad \text{for }j\to\infty,
\]
with convergence of all moments.
\end{remark}

\begin{remark}
Panholzer and Seitz~\cite{PanSeitz2012} studied labelled families of evolving $k$-tree models, generalizing simple families of increasing trees. 
An identical phase change and factorial moments of mixed Poisson type with a Gamma mixing distribution can be observed when studying the number of descendants of specific nodes in labelled families of evolving $k$-tree models.
\end{remark}

The parameter ``descendants of node $j$'' can be modelled also via urn models: we encounter classical P\'olya urns with \emph{non-standard initial values}, depending on the number of draws.
Note that Mahmoud and Smythe~\cite{MahSmy1991} used a similar approach to study the descendants of node $j$ in recursive trees, for $j$ fixed and $n \to \infty$.

In order to get an urn model description of the number of descendants of node $j$ one considers the tree evolution process generating random trees of the tree families studied. The probabilities $p(v)$ given in \eqref{pv} can then be translated into the ball replacement matrix of a two-colour urn model. Alternatively, for this task one can use combinatorial descriptions of these tree families, which we will here only figure out for the case $\dinc$, i.e., $d$-ary increasing trees; the other cases can be treated similarly. 

Consider such a $d$-ary increasing tree of size $j$: there are $1+(d-1)j$ possible attachment positions (often drawn as external nodes), where a new node can be attached. Exactly $d$ such attachment positions are contained in the subtree rooted $j$, whereas the other $(d-1)(j-1)$ are not. In the urn model description we will use balls of two colours, black and white. Each white ball will correspond to an attachment position contained in the subtree rooted $j$, whereas each black ball will correspond to an attachment position that is not contained in the subtree rooted $j$, which we call here ``remaining tree''. This already describes the initial conditions of the urn. 

Moreover, during the tree evolution process, when attaching a new node to a position in the subtree of $j$, then there appear $d-1$ new attachment positions in this subtree, whereas when attaching a new node to a position in the remaining tree, then there appear $d-1$ new attachment positions in the remaining tree. In the urn model description this simply means that when drawing a white ball one adds $d-1$ white balls and when drawing a black ball one adds $d-1$ black balls to the urn. After $n-j$ draws, which correspond to the $n-j$ attachments of nodes in the tree, the number of white balls in the urn is linearly related to the size of the subtree rooted $j$ in a tree of size $n$.
Thus, the following urn model description follows immediately.
\begin{urn}
Consider a P\'olya urn with ball replacement matrix
\[
M = \left(
\begin{matrix}
\kappa & 0\\
0 & \kappa
\end{matrix}
\right), \qquad 
\kappa =
\begin{cases}
1, &\rec,\\
1+\alpha, &\gport,\\
d-1, &\dinc,\\
\end{cases}
\]
and initial conditions
\[
W_0= 
\begin{cases}
1, \\
\alpha,\\ 
d,\\
\end{cases}
\qquad 
B_0= 
\begin{cases}
j-1, &\rec,\\
(j-1)(1+\alpha), &\gport,\\
(j-1)(d-1), & \dinc,\\
\end{cases}
\]
for $1\le j\le n$. The number $D_{n,j}$ of descendants of node $j$ in an increasing tree of size $n$ has the same distribution as the (shifted and scaled) number of white balls $W_{n-j}$ in the P\'olya urn after $n-j$ draws,
\[
D_{n,j}\law 
\begin{cases}
W_{n-j}, &\rec,\\
(W_{n-j}+1)/\kappa, &\gport,\\
(W_{n-j}-1)/\kappa, &\dinc.
\end{cases}
\]
\end{urn}

This implies that the number of white balls in the standard P\'olya urn model exhibit a phase transition according to the growth of the initial number of black balls present in the urn compared to the discrete time.

\subsection{Node-degrees in plane recursive trees.\label{MOMSEQExampleNodedegs}}
Let $X_{n,j}$ denote the random variable counting the outdegree of node $j$ in a generalized plane recursive tree of size $n$, i.e., a size-$n$ tree from the increasing tree family {\gport} defined in Section~\ref{MOMSEQExampleDesc}.
\begin{figure}[!htb]
\centering
  \includegraphics[scale=0.6]{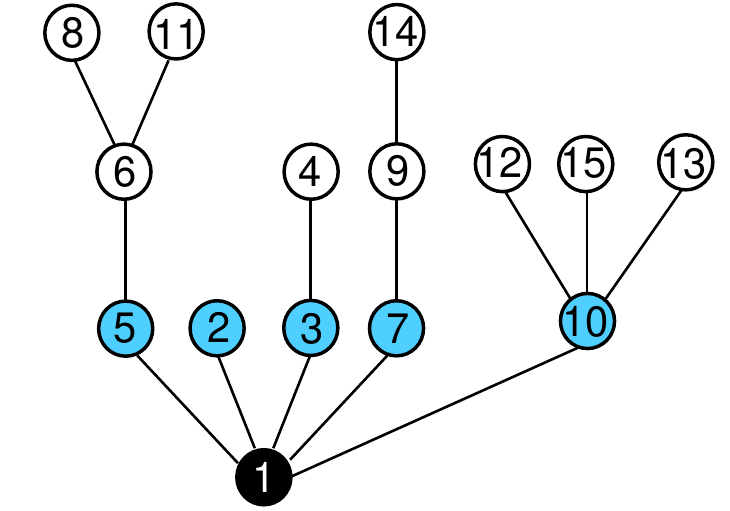} 
\caption{A size $15$ plane recursive tree, where the root node $j=1$ has outdegree five.\label{fig:PortNodeDegree}}
\end{figure}

It has been shown in~\cite{KubPan2007} using a generating functions approach that the factorial moments of the random variable $X_{n,j}$ are given by
\begin{equation*}
\E(\fallfak{X_{n,j}}{s})= \frac{\Gamma(s+\alpha)}{\Gamma(\alpha)} \sum_{k=0}^s\binom{s}{k}(-1)^k \frac{\Gamma\big(n+\frac{s-1-k}{1+\alpha}\big)
   \Gamma\big(j-\frac{1}{1+\alpha}\big)}{\Gamma\big(j+\frac{s-1-k}{1+\alpha}\big)\Gamma\big(n-\frac{1}{1+\alpha}\big)},
\end{equation*}
for $j\ge 2$, with $\alpha$ given in definition \eqref{pp_simplified} of the degree-weight generating function $\varphi(t)$. Lemma~\ref{MOMSEQMainLemma} and an application of Stirling's formula for the Gamma function~\eqref{MOMSEQStirling} leads to the following result.
\begin{coroll}
\label{MOMSEQCorollNodedeg}
The random variable $X_{n,j}$, counting the outdegree of node $j$ in a random generalized plane recursive tree of size $n$, $2 \le j\le n$, has for $n\to\infty$ and $j=j(n)\to \infty$, 
falling factorial moments of mixed Poisson type with a Gamma mixing distribution $X\law\gamma(1,\alpha)$, 
and scale parameter $\lambda_{n,j}=\Big(\frac{n}{j}\Big)^{1/(\alpha+1)}-1$:
\[
\E(\fallfak{X_{n,j}}s)= \lambda_{n,j}^s\frac{\Gamma(s+\alpha)}{\Gamma(\alpha)}(1+o(1)).
\]
\begin{itemize}
\item[(i)] for $\lambda_{n,j}\to\infty$, the random variable $\frac{X_{n,j}}{\lambda_{n,j}}$ converges in distribution, with convergence of all moments, to $X$. 

\item[(ii)] for $\lambda_{n,j}\to\rho\in(0,\infty)$, the random variable $X_{n,j}$ converges in distribution, with convergence of all moments, to a mixed Poisson distributed random variable $Y\law \MPo(\rho X)$, 
which has a negative binomial distribution. 
\end{itemize}
\end{coroll}

\begin{remark}
\label{remarkNodeDegrees}
For fixed $j$, independent of $n$, a different limit law arises (compare with Remark~\ref{MOMSEQRemarkAdditionalPhase}): the random variable $X_{n,j}/n^{1/(\alpha+1)}$ converges for $n\to\infty$ in distribution to a random variable $Z_j$ characterized by its moments or by its  density function; see~\cite{KubPan2007} and also Corollary~\ref{MOMSEQCorollRefinedNodedeg} for details.
\end{remark}

The random variable $X_{n,j}$ also allows an urn model description. To get it we consider the tree evolution process for the family {\gport}, with $p(v) = \frac{d(v)+\alpha}{(\alpha+1)i-1}$, where $p(v)$ gives the probability that the new node $i+1$ will be attached to node $v$ in a tree of size $i$, depending on the outdegree $d(v)$ of $v$. We see that the probability $p(v)$ is proportional to $d(v)+\alpha$. Thus let us think about the quantity $d(v)+\alpha$ as the affinity of node $v$ attracting a new node. The total affinity of all nodes $v_{k}$ in a tree $T$ of size $i$ is then given by $\sum_{1 \le k \le i} (d(v_{k})+\alpha) = i-1+\alpha i= (\alpha+1)i-1$ giving an interpretation of the denominator of $p(v)$. In the two-colour urn model we use white balls describing the affinity of node $j$ to attract new nodes, whereas the black balls describe the affinity of all remaining nodes in the tree to attract new nodes. When considering a tree of size $j$, it holds $d(j) = 0$ and thus that node $j$ has affinity $\alpha$, whereas all remaining nodes have total affinity $(\alpha+1)j-1-\alpha = (\alpha+1)(j-1)$ to attract a new node. This already yields the initial conditions of the urn.

Each time a new node is attached during the tree evolution process the total affinity of all nodes increases by $\alpha+1$: $\alpha$ is the affinity of the new node (which has outdegree $0$) and by attaching this node the outdegree and thus affinity of one node increases by one.
In particular, when a new node is attached to node $j$ the affinity of $j$ increases by one, whereas the total affinity of the remaining nodes increases by $\alpha$, but if a new node is attached to another node the affinity of $j$ remains unchanged. Thus in the urn model description, when drawing a white ball, one white ball and $\alpha$ black balls are added, and when drawing a black ball, $\alpha+1$ black balls are added to the urn. The following description is then immediate.

\begin{urn}\label{urn:node_degrees}
Consider a balanced triangular urn with ball replacement matrix
\[
M = \left(
\begin{matrix}
1 & \alpha\\
0 & 1+\alpha
\end{matrix}
\right),
\qquad W_0=\alpha,\quad B_0=(\alpha+1)(j-1),
\]
for $1\le j\le n$. The outdegree $X_{n,j}$ of node $j$ in a generalized plane recursive tree of size $n$ has the same distribution as the shifted number of white balls $W_{n-j}$ in the P\'olya urn after $n-j$ draws,
\[
X_{n,j}\law W_{n-j}-\alpha.
\]
\end{urn}

This implies that the number of white balls in the standard P\'olya urn model exhibit several phase transitions according to the growth of the initial number of black balls present in the urn with respect 
to the total number of draws; this will be discussed in detail in a more general setting in Section~\ref{MOMSEQtriangular}.

\subsection{Branching structures in plane recursive trees}
Let $X_{n,j,k}$ denote the random variable, which counts the number of size-$k$ branches (= subtrees of size $k$)
attached to the node labelled $j$ in a random increasing tree of size n. 
The random variables $X_{n,j,k}$ are thus related to the random variable $X_{n,j}$ studied in Section~\ref{MOMSEQExampleNodedegs}, counting the outdegree of node labelled $j$, via
\[
X_{n,j}=\sum_{k=1}^{n-j}X_{n,j,k}.
\] 

\begin{figure}[!htb]
\centering
    \includegraphics[scale=0.6]{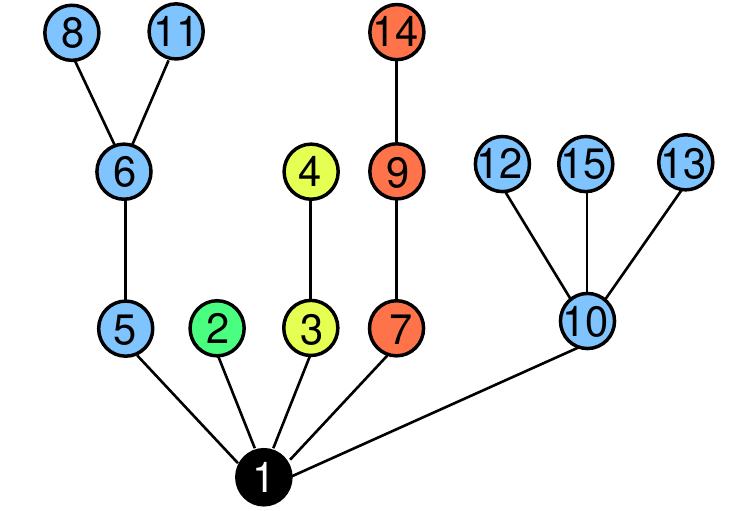} 
\caption{A size $15$ plane recursive tree, where the root node has one size-one, one size-two, one size three and two size-four branches.\label{fig:PortSubtreeSizes}}
\end{figure}

This parameter was studied in Su et al.~\cite{Su} for the particular case of the root node $j = 1$ and the instance of random recursive trees:
they derived the distribution of $X_{n,1,k}$ and a limit law for it. Furthermore they stated results for joint distributions.
The analysis was extended in~\cite{KubPan2006} to increasing tree families generated by a natural growth process (see Section~\ref{MOMSEQExampleDesc}). In particular, for the family {\gport} of generalized plane recursive trees with parameter $\alpha$ given in \eqref{pp_simplified}, the following result for the factorial moments of $X_{n,j,k}$ was obtained:
\[
\E(\fallfak{X_{n,j,k}}s)=\bigg(\frac{  \binom{k-\frac{1}{\alpha+1}}{k-1}}{(\alpha+1) k}\bigg)^s \cdot 
     \frac{\Gamma(s+\alpha)}{\Gamma(\alpha)}
     \frac{\binom{j-1-\frac{1}{\alpha+1}}{j-1}\binom{n-ks-1+\frac{s-1}{\alpha+1}}{n-j-ks}}{\binom{n-1}{j-1} \binom{n-1-\frac{1}{\alpha+1}}{n-1}}.
\]
In~\cite{KubPan2006} only the case of fixed $k$ was considered. We can easily use Lemma~\ref{MOMSEQMainLemma} and Stirling's formula for the Gamma function~\eqref{MOMSEQStirling} to extend the studies given there and to obtain the following result.
\begin{coroll}
\label{MOMSEQCorollRefinedNodedeg}
The random variable $X_{n,j,k}$, counting the number of size-$k$ branches attached to node $j$ in a random generalized plane recursive tree of size $n$, has for $j$ fixed, $n\to\infty$ and $1\le k \le n-j$, 
falling factorial moments of mixed Poisson type with mixing distribution $Z_j$ supported on $[0,\infty)$, uniquely defined by its 
moment sequence $(\mu_s)_{s\in\N}=\Big(\frac{\Gamma(s+\alpha)\Gamma(j-\frac{1}{\alpha+1})}{\Gamma(\alpha)\Gamma(j+\frac{s-1}{\alpha+1})}\Big)_{s\in\N}$, and scale parameter $\lambda_{n,j,k}=\frac{n^{\frac1{\alpha+1}}  \binom{k-\frac{1}{\alpha+1}}{k-1}}{(\alpha+1) k}$:
\[
\E(\fallfak{X_{n,j,k}}s)= \lambda_{n,j,k}^s\cdot \frac{\Gamma(s+\alpha)\Gamma(j-\frac{1}{\alpha+1})}{\Gamma(\alpha)\Gamma(j+\frac{s-1}{\alpha+1})}(1+o(1)).
\]
\begin{itemize}
\item[(i)] for $\lambda_{n,j,k}\to\infty$, the random variable $\frac{X_{n,j,k}}{\lambda_{n,j,k}}$ converges in distribution, with convergence of all moments, to $Z_j$. 

\item[(ii)] for $\lambda_{n,j,k}\to\rho\in(0,\infty)$, the random variable $X_{n,j,k}$ converges in distribution, with convergence of all moments, to a mixed Poisson distributed random variable $Y\law \MPo(\rho Z_j)$.
\end{itemize}
\end{coroll}

\begin{remark}
The results above can be generalized to growing $j=j(n)$, leading to results similar to our earlier findings for the ordinary outdegree
$X_{n,j}$. The random variable $Z_j$ is exactly the limit law of $X_{n,j}/n^{1/(\alpha+1)}$ for fixed $j$ as discussed in Remark~\ref{remarkNodeDegrees}. The density functions $f_j(x)$ of the $Z_j$, $j\in\N$, are known explicitly, see~\cite{KubPan2007}.
\end{remark}

We can interpret our findings in terms of an urn model reminiscent to the urn model for block sizes in $k$-Stirling permutations.
To do this we can extend the description given in Section~\ref{MOMSEQExampleNodedegs} leading to Urn~\ref{urn:node_degrees}.
Namely, we require a refinement of describing the affinities of the nodes in a tree to attract new nodes during the tree evolution process.

For doing that we use balls of $k+2$ different colours. Balls of colour $0$ describe the affinity of node labelled $j$ attracting a new node to become attached at $j$. Furthermore, balls of colour $i$, with $1 \le i \le k$, describe the affinity of attracting a new node of nodes that are contained in branches of size $i$ attached to node $j$, whereas balls of colour $k+1$ describe the affinity of attracting a new node of all remaining nodes, i.e., of nodes that are not contained in the subtree rooted $j$ or nodes contained in a branch of size $\ge k+1$ attached to $j$. Consider a generalized plane recursive tree of size $j$: node $j$ has affinity $\alpha$ to attract a new node, whereas the total affinity of all remaining nodes is $(\alpha+1)(j-1)$; this characterizes the initial configuration of the urn. 

The ball replacement matrix of the urn can be obtained as follows. When drawing a ball of colour $0$, i.e., when attaching a new node to $j$, the node degree of $j$ increases by one and there appears a new branch of size $1$ attached to node $j$, which means that we add one ball of colour $0$ and $\alpha$ balls of colour $1$.
When drawing a ball of colour $i$, with $1 \le i \le k$, i.e., when attaching a new node to a size-$i$ branch attached to $j$, this branch transforms into a size-$(i+1)$ branch attached to $j$, which means that we remove $(\alpha+1)i-1$ balls of colour $i$ (which corresponds to the affinity of a size-$i$ branch) and add $(\alpha+1)(i+1)-1$ balls of colour $i+1$ (which thus correspond to the affinity of a size-$(i+1)$ branch). Furthermore, when drawing a ball of colour $k+1$, i.e., when attaching a new node to node not contained in the subtree rooted $j$ or when attaching a new node to a node contained in a branch of size $\ge k+1$ attached to node $j$, this does neither affect the affinity of node $j$ nor of its branches of sizes $\le k$, which means that we add $\alpha+1$ balls of colour $k+1$. The following urn model description immediately follows.

\begin{urn}
Consider a balanced urn with balls of $k+2$ colors and let the random vector $(U_{n,0},U_{n,1},\dots,U_{n,k+1})$ count the number of balls of each color at time $n$ with $(k+2) \times (k+2)$ ball replacement matrix $M$ given by
\begin{equation*}
    M = \left(
    \begin{smallmatrix}
        1 & \alpha & 0 & \cdots & 0 & 0 & 0& 0 \\[-1ex]
        0 & -\alpha & 2(\alpha+1)-1 & \ddots & \ddots & \ddots & 0& 0 \\[-1ex]
        0 & 0 & -(2(\alpha+1)-1) & 3(\alpha+1)-1 & \ddots & \ddots & 0 & 0\\[-1ex]
        \vdots & \ddots & \ddots & \ddots & \ddots &
        \ddots & \vdots & 0\\[-1ex]
        \vdots & \ddots & \ddots & \ddots & \ddots &
        \ddots & \vdots & 0\\[-1ex]
                0 & \ddots & \ddots & \ddots & 0 & -((\alpha+1)(k-1)-1) & (\alpha+1)k-1& 0\\[-1ex]
        0 & \ddots & \ddots & \ddots & 0 & 0 & -((\alpha+1)k-1) & (\alpha+1)(k+1)-1\\
        0 & 0 & 0 & \cdots & 0 & 0 & 0& 1+\alpha
    \end{smallmatrix}
    \right).
\end{equation*}
The initial configuration of the urn (it is convenient to start here at time $0$) is given by
$(U_{0,0},U_{0,1},\dots,U_{0,k+1})=(\alpha,0,\dots,0,(\alpha+1)(j-1))$. The random variables $U_{n,i}$, with $1 \le i \le k$, described by the urn model are related to the random variables $X_{n,j,i}$, $1 \le i \le k$, which count the number of size-$i$ branches 
attached to the node labelled $j$ in a random generalized plane recursive tree of size n, as follows:
\begin{equation*}
   U_{n-j,i} = ((\alpha+1)i-1) X_{n,j,i}, \quad 1\le i\le k.
\end{equation*}
Moreover, $U_{n-j,0}$ is related to the outdegree $X_{n,j}$ via $U_{n-j,0}=X_{n,j}+\alpha$.
\end{urn}
This implies that the random variables $U_{n,i}$ occurring in the urn model undergo a phase transition according to the growth of $k$ with respect to $n$, from continuous to discrete.

\subsection{Distribution of table sizes in the Chinese restaurant process}
The Chinese restaurant process with parameters $a$ and $\theta$ is a discrete-time stochastic process, whose value at any positive integer time $n$ is one of the $B_n$ partitions of the set $[n]=\{1, 2, 3,\dots , n\}$ (see Pitman~\cite{Pitman}). The parameters $a$ and $\theta$ satisfy $0< a <1$ and $\theta >-a$. Here $B_n$ denotes the Bell number counting the number of partitions of an $n$-element set $B_0=B_1=1$, $B_2=2$, $B_3=5$, etc.\footnote{See sequence \href{http://oeis.org/A000110}{A000110} in OEIS.}
One imagines a Chinese restaurant with an infinite number of tables, and each table has an infinite number of seats. In the beginning the first customer takes place at the first table. 
At each discrete time step a new customer arrives and either joins one of the existing tables, or he takes place at the next empty table in line.
Each table corresponds to a block of a random partition. In the beginning at time $n = 1$, the trivial partition $\{ \{1\} \}$ is obtained with probability 1. Given a partition $T$ of $[n]$ with $|T|=k$ parts $t_i$, $1 \le i \le k \le n$, of sizes $|t_i|$. At time $n + 1$  the element $n + 1$ is either added to one of the existing parts $t_i\in T$ with probability
\[
\P\{n+1<_c t_i\}=\frac{|t_i|-a}{n+\theta},\quad 1\le i\le k,
\]
or added to the partition $T$ as a new singleton block with probability
\[
\P\{n+1<_c t_{|T|+1}\}=\frac{\theta+k \cdot a}{n+\theta}.
\]
This model thus assigns a probability to any particular partition $T$ of $[n]$. We are interested in the distribution of the random variable $C_{n,j}$, counting the number of parts of size $j$ in a partition of $[n]$ generated by the Chinese restaurant process. 

\smallskip

We will not directly study the Chinese restaurant process, but in order to analyse the number of tables of a certain size we study instead a variant of the growth rule for generalized plane recursive trees as introduced in Section~\ref{MOMSEQExampleDesc}. 
Combinatorially, we consider a family $\mathcal{T}_{\alpha,\beta}$ of generalized plane recursive trees, where the degree-weight generating function $\vartheta(t)=\frac{1}{(1-t)^\beta}$, $\beta>0$, associated to the root of the tree, is different to the one for non-root nodes in the tree, $\varphi(t)=\frac{1}{(1-t)^\alpha}$, $\alpha>0$. Then, the family $\mathcal{T}_{\alpha,\beta}$ is closely related to the corresponding family $\mathcal{T}$ of generalized plane recursive trees with degree-weight generating $\varphi(t)=\frac{1}{(1-t)^\alpha}$, $\alpha>0$, via the following formal recursive equations:
\begin{equation}
\begin{split}
 \mathcal{T}_{\alpha,\beta} = \bigcirc\hspace*{-0.75em}\text{\small{$1$}}\hspace*{0.3em} \times \vartheta(\mathcal{T}),
\qquad \mathcal{T} = \bigcirc\hspace*{-0.75em}\text{\small{$1$}}\hspace*{0.3em} \times \varphi(\mathcal{T}).
\end{split}
\end{equation}
The weight $w(T)$ of a tree $T \in \mathcal{T}_{\alpha,\beta}$ is then defined as
\[
w(T) := \vartheta_{d(\text{root})}\prod_{v\in T\setminus\{\text{root}\}} \varphi_{d(v)},
\]
where $d(v)$ denotes the outdegree of node $v$. Thus, the generating functions
$T_{\alpha,\beta}(z)=\sum_{n\ge 1}T_{\alpha,\beta;n}\frac{z^n}{n!}$ and $T(z)=\sum_{n\ge 1}T_n\frac{z^n}{n!}$ of the total weight of size-$n$ trees in $\mathcal{T}_{\alpha,\beta}$ and $\mathcal{T}$, respectively, satisfies the differential equations
\[
T'_{\alpha,\beta}(z)=\vartheta(T(z)),\qquad T'(z)=\varphi(T(z)).
\]

Moreover, the tree evolution process to generate a random tree of arbitrary given size in the family $\mathcal{T}$ described in Section~\ref{MOMSEQExampleDesc} can be extended in the following way to generate a random tree in the family $\mathcal{T}_{\alpha,\beta}$.
The process, evolving in discrete time, starts with the root labelled by zero.
At step $n+1$, with $n\ge 0$, the node with label $n+1$ is attached to any previous
node $v$ with outdegree $d(v)$ of the already grown tree with probabilities $p(v)$, 
which are given for non-root nodes $v$ by
\[
\P\{n+1<_c v\}=\frac{d(v)+\alpha}{\beta+(\alpha+1)n},
\]
and for the root by
\[
\P\{n+1<_c \text{root}\}=\frac{d(\text{root})+\beta}{\beta+(\alpha+1)n}.
\]
This growth process is similar to the Chinese restaurant process considered before.
Indeed, if we remove the root labelled zero, the remaining branches contain the nodes with labels given by $[n]=\{1,\dots,n\}$. 

\begin{figure}[!htb]
\centering
  \includegraphics[scale=0.6]{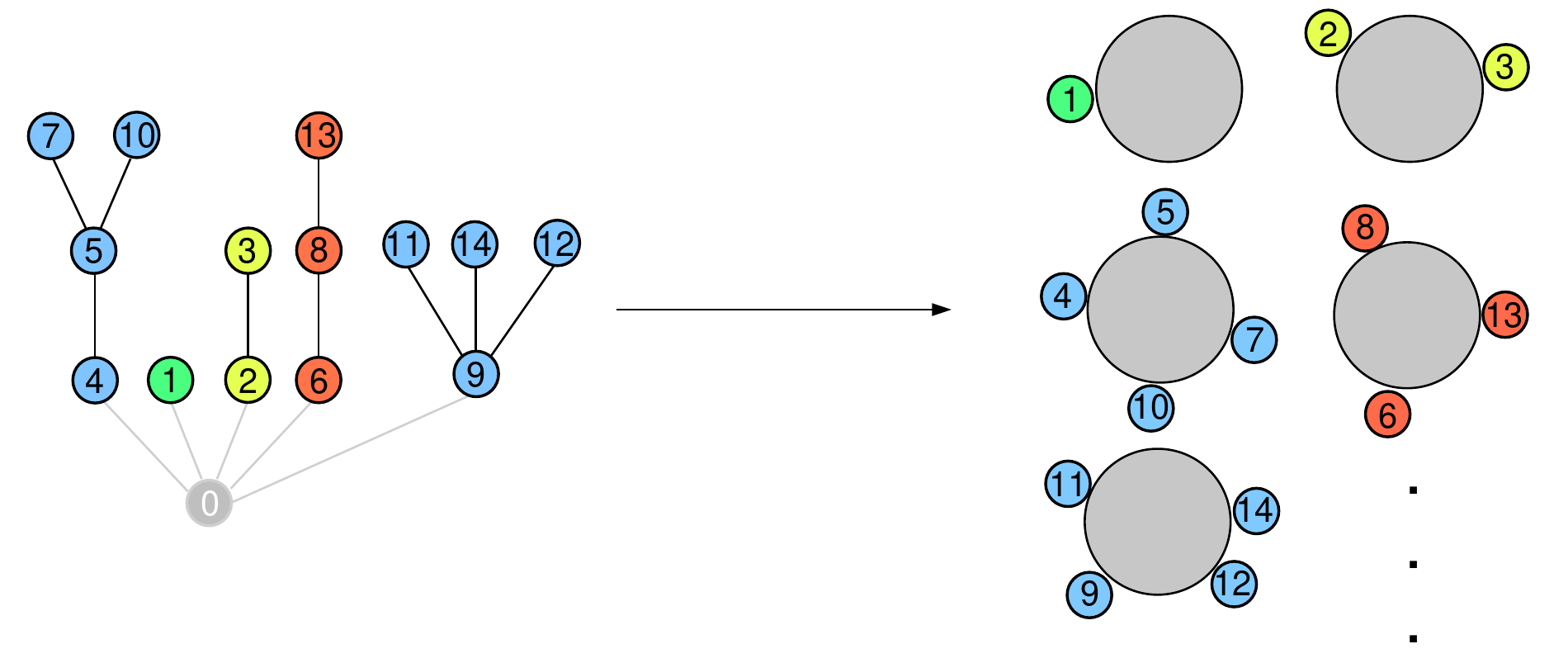} 
\caption{A size $15$ plane recursive tree, where the root node labelled zero has one size-one, one size-two, one size three, two size-four branches and the corresponding table structure in the Chinese restaurant model.\label{fig:PortSubtreeSizesb}}
\end{figure}

\begin{prop}[Chinese restaurant process and generalized plane recursive trees]
\label{ChineseProp}
A random partition of $\{1,\dots,n\}$ generated by the Chinese restaurant process with parameters $a$ and $\theta$ can be generated equivalently by the growth process of the family of generalized plane recursive trees $\mathcal{T}_{\alpha,\beta}$ when generating such a tree of size $n+1$. The parameters $a,\theta$ and $\alpha,\beta>0$, respectively, are related via
\[
a=\frac{1}{1+\alpha},\qquad \theta=\frac{\beta}{1+\alpha}.
\]
\end{prop}

\begin{remark}
\label{ChineseRem}
In above relation, $\theta$ cannot be negative, since $\beta$ is assumed to be positive. 
The above correspondence can be extended to the full range $\beta>-1$ using a different degree-weight generating function $\vartheta(t)$ 
for the root node. Assume that $-1<\beta\le 0$. Then, we cannot directly use $\vartheta(t)=(1-t)^{-\beta}=1+\beta t + \dots$ due to the negative or zero weight.
Since the root connectivity is similar to the choice $\beta\mapsto 1+\beta$ for an outdegree of the root larger than one, 
we use a shifted connectivity of the root node: 
$$\vartheta(t)=1+\int_0^{t}\frac{1}{(1-x)^{1+\beta}}dx
= 1 + \frac{1}{\beta}\Big(\frac1{(1-t)^{\beta}}-1\Big)=1+\sum_{k\ge 1}\binom{\beta+k}{k-1}\frac{t^k}k,
$$
for $-1<\beta<0$, and 
$$\vartheta(t)=1-\log(1-t)=1+\sum_{k\ge 1}\frac{t^k}k,$$
for $\beta=0$.
\end{remark}

\begin{proof}[Proof of Proposition~\ref{ChineseProp}]
Assume that a size-$n$ tree $T$ of the family $\mathcal{T}_{\alpha,\beta}$ with labels $\{0,1,\dots ,n\}$ 
has $k$ branches $t_i$, $1\le i\le k$, of sizes $|t_i|$. 
By the considerations of~\cite{PanPro2005+}, at time $n + 1$ the element $n + 1$ is either attached to one of the 
existing non-root nodes $v$ with probability
\begin{equation*}
\begin{split}
\P\{n+1<_c v\}&=\frac{d(v)+\alpha}{\beta+(\alpha+1)n},
\end{split}
\end{equation*}
or to the root of the tree with probability 
\[
\P\{n+1<_c \text{root}\}=\frac{d(\text{root})+\beta}{\beta+(\alpha+1)n}
=\frac{k+\beta}{\beta+(\alpha+1)n}=\frac{\frac{\beta}{\alpha+1} + k\cdot \frac1{\alpha+1}}{n+\frac{\beta}{\alpha+1}}.
\]
Consequently, element $n+1$ is attached to one of the branches $t_i\in T$ with probability
\begin{equation*}
\begin{split}
\P\{n+1<_c t_i\}&=\sum_{v\in t_i}\P\{n+1<_c v\}=\sum_{v\in t_i}\frac{d(v)+\alpha}{\beta+(\alpha+1)n}\\
&=\frac{|t_i|-1+|t_i|\alpha}{\beta+(\alpha+1)n}=\frac{|t_i|-\frac{1}{\alpha+1}}{n+\frac{\beta}{\alpha+1}},
\end{split}
\end{equation*}
Thus, setting $a=\frac{1}{1+\alpha}$ and $\theta=\frac{\beta}{1+\alpha}$ proves the stated result.
\end{proof}

\begin{theorem}
\label{ChineseThe}
The random variable $C_{n,j}$, counting the number of parts of size $j$ in a partition of $\{1,\dots,n\}$ generated by the Chinese restaurant process with parameters $a$ and $\theta$, is distributed as $X_{n+1,j}$, which counts the number of branches of size $j$ attached to the root of a random size-$(n+1)$ generalized plane recursive tree of the family $\mathcal{T}_{\alpha,\beta}$, with $a=\frac{1}{1+\alpha}$ and $\theta=\frac{\beta}{1+\alpha}$:
\[
C_{n,j}\law X_{n+1,j}.
\]
Assume that $\beta> 0$. Then, $X_{n+1,j}$ has falling factorial moments of mixed Poisson type with mixing distribution $Z$ supported on $[0,\infty)$, and scale parameter $\lambda_{n,j}=\frac{n^{\frac1{\alpha+1}}  \binom{j-1-\frac{1}{\alpha+1}}{j-1}}{(\alpha+1) j}$:
\[
\E(\fallfak{X_{n+1,j}}s)= \lambda_{n,j}^s\cdot \frac{\Gamma(s+\beta)\Gamma(\frac{\beta}{\alpha+1})}{\Gamma(\beta)\Gamma(\frac{\beta+s}{\alpha+1})}(1+o(1)).
\]
\begin{itemize}
\item[(i)] for $\lambda_{n,j}\to\infty$, the random variable $\frac{X_{n+1,j}}{\lambda_{n,j}}$ converges in distribution, with convergence of all moments, to $Z$. 

\item[(ii)] for $\lambda_{n,j}\to\rho\in(0,\infty)$, the random variable $X_{n+1,j}$ converges in distribution, with convergence of all moments, to a mixed Poisson distributed random variable $Y\law \MPo(\rho Z)$.
\end{itemize}
\end{theorem}

\begin{remark}
A similar result holds true for $-1<\beta \le 0$. The analysis is identical, but one has to use the adapted degree-weight generating functions stated in Remark~\ref{ChineseRem}.
\end{remark}

\begin{proof}[Proof of Theorem~\ref{ChineseThe}]
We can study $X_{n+1,j}$, which counts the number of branches of size $j$ attached to the root of a size-$(n+1)$ tree using the variable $v$ as a marker and the generating function 
$$
T_{\alpha,\beta}(z,v)=\sum_{n\ge 1}T_{\alpha,\beta;n}\E(v^{X_{n,j}})\frac{z^n}{n!}.
$$
We have
\[
T'_{\alpha,\beta}(z,v)=\vartheta\big(T(z)-\frac{T_j}{j!}z^j(1-v)\big),\qquad T'(z)=\varphi(T(z)).
\]
Solving the differential equation for $T(z)$ leads to
\[
T'_{\alpha,\beta}(z,v)=\frac1{\big(1-T(z)+\frac{T_j}{j!}z^j(1-v)\big)^\beta},\qquad T(z)=1-(1-(\alpha+1)z)^{\frac{1}{\alpha+1}}.
\]
We can access the \ith{s} moment of $X_{n+1,j}$ as follows:
\[
\E(\fallfak{X_{n+1,j}}{s})=s!\frac{[z^n w^s]T'_{\alpha,\beta}(z,v)}{\frac{T_{\alpha,\beta;n}}{n!}},
\]
where we set $w=v-1$. Consequently, an expansion of $T'_{\alpha,\beta}(z,v)$ at $w=0$ gives 
\[
\E(\fallfak{X_{n+1,j}}{s})=\frac{n!s!}{T_{\alpha,\beta;n}}\cdot\Big(\frac{T_j}{j!}\Big)^s\binom{\beta-1+s}{s} [z^{n-j s}]\frac{1}{(1-T(z))^{\beta+s}}.
\]
Since $\frac{T_{\alpha,\beta;n}}{n!}=(\alpha+1)^n\binom{\frac{\beta}{\alpha+1}+n-1}{n}$, we obtain the explicit result
\[
\E(\fallfak{X_{n+1,j}}{s})=\bigg(\frac{\binom{j-1-\frac{1}{\alpha+1}}{j-1}}{(\alpha+1)j}\bigg)^s \cdot \frac{s!}{\binom{\frac{\beta}{\alpha+1}+n-1}{n}}\cdot\binom{\beta-1+s}{s}
\binom{n -j s -1 + \frac{\beta+s}{\alpha+1}}{n-j s}.
\]
An application of Stirling's formula for the Gamma function yields then the stated result.
\end{proof}

\section{Triangular urn models\label{MOMSEQtriangular}}
During the study of node-degree in generalized plane recursive trees in Section~\ref{MOMSEQExampleNodedegs} we encountered a triangular urn model leading to factorial moments of mixed Poisson type. Here we study a more general triangular urn.
\begin{urn}
Consider a balanced triangular urn model with ball replacement matrix 
$$M = \left(\begin{matrix} \alpha & \beta \\ 0 & \gamma \end{matrix}\right),
\quad \alpha,\beta,\gamma \in \N, \qquad \gamma=\alpha+\beta\in\N.$$ 
The initial configuration of the urn consists of $w_0$ white balls and $b_0$ black balls, 
and the random variable $W_{n}$ counts the number of white balls in the urn after $n$ draws.
\end{urn}
This urn model has been studied by Puyhaubert~\cite{FlaDumPuy2006,Puy2005} who derived the probability mass function of $W_n$ and a limit law for $n\to\infty$. The results of~\cite{FlaDumPuy2006,Puy2005} were extended by Janson~\cite{Janson2006} to unbalanced triangular urn models. 
Here, using a simple closed formula for the rising factorial moments of $W_n$, we point out several phase transitions, involving amongst others moments of mixed Poisson type, for non-standard initial values $b_0=b_0(n)$, which may depend on the discrete time $n$. Due to the balanced nature of the urn 
the total number $T_n$ of balls after $n$ draws is a deterministic quantity:
\[
T_n=T_0+n\cdot \gamma,\quad n\ge 0,\qquad T_0=w_0+b_0.
\]

\smallskip

Our starting point is the analysis of the normalized number of white balls $X_n=W_n/\alpha$, such that $X_0=w_0/\alpha$.
Let $\mathcal{F}_{n}$ denote the $\sigma$-field generated by the first
$n$ steps. Moreover, denote by $\Delta_n=X_n-X_{n-1}\in\{0,1\}$ the
increment at step $n$.  
We have 
\begin{equation*}
\E(X_n \mid \mathcal{F}_{n-1})= \E(X_{n-1}+\Delta_n \mid \mathcal{F}_{n-1})
=X_{n-1} +  \E(\Delta_n \mid \mathcal{F}_{n-1}).
\end{equation*}
Since the probability that a new white ball is generated at step $n$
is proportional to the number $W_{n-1}=X_{n-1}\cdot \alpha$ of existing white balls (at step $n-1$),
we obtain further 
\begin{equation*}
\E(X_n \mid \mathcal{F}_{n-1})
=X_{n-1} +\frac{X_{n-1}\cdot \alpha}{T_{n-1}}
=\frac{T_{n-1}+\alpha}{T_{n-1}}W_{n-1},\quad n\ge 1.
\end{equation*}
Hence, let 
\[
\mathcal{X}_n= X_n\cdot\prod_{k=0}^{n-1}\frac{T_k}{T_k+\alpha}=X_n\cdot \frac{\binom{n-1+\frac{T_0}{\gamma}}n}{\binom{n-1+\frac{T_0+\alpha}{\gamma}}n}.
\]
Then
\begin{equation*}
\E( \mathcal{X}_n \mid \mathcal{F}_{n-1})
= \frac{\binom{n-1+\frac{T_0}{\gamma}}n}{\binom{n-1+\frac{T_0+\alpha}{\gamma}}n} \cdot \frac{T_{n-1}+\alpha}{T_{n-1}}W_{n-1}
=\mathcal{X}_{n-1},\quad n\ge 1.
\end{equation*}
Consequently, $\mathcal{X}_n$ is a positive martingale. By taking the unconditional expectation, this implies that the expected value 
of $X_n$ is given 
by
\[
\E(X_n)= \frac{\binom{n-1+\frac{T_0+\alpha}{\gamma}}n}{\binom{n-1+\frac{T_0}{\gamma}}n} \cdot \E(X_0)=\frac{\binom{n-1+\frac{T_0+\alpha}{\gamma}}n}{\binom{n-1+\frac{T_0}{\gamma}}n}\cdot X_0.
\]

More generally, we similarly have for any positive integer $s$
\begin{equation*}
\begin{split}
\E\bigg(\binom{X_n+s-1}{s}\biggm| \mathcal{F}_{n-1}\bigg)&= \binom{X_{n-1}+s-1}{s} + \binom{X_{n-1}+s-1}{s-1}\frac{\alpha X_{n-1}}{T_{n-1}}\\
&=\binom{X_{n-1}+s-1}{s}\frac{T_{n-1}+s\alpha}{T_{n-1}}.
\end{split}
\end{equation*}
Hence, this implies that the \ith{s} binomial moment is given by
\begin{equation*}
\begin{split}
\E\bigg(\binom{X_n+s-1}{s}\bigg)&=\frac{\binom{n-1+\frac{T_0+s\alpha}{\gamma}}n}{\binom{n-1+\frac{T_0}{\gamma}}n}\cdot \binom{X_0+s-1}{s}\\
&= \frac{\Gamma(n+\frac{w_0+b_0+s\alpha}{\gamma})\Gamma(\frac{w_0+b_0}{\gamma})\Gamma(\frac{w_0}\alpha +s )}{\Gamma(\frac{w_0+b_0+s\alpha}{\gamma})\Gamma(n+\frac{w_0+b_0}{\gamma})\Gamma(\frac{w_0}\alpha)\Gamma(s+1)}.
\end{split}
\end{equation*}

\begin{theorem}\label{thm:triangular_urn}
The \ith{s} rising factorial moment of the random variable $X_n=W_{n}/\alpha$, where $W_n$ counts the number of white ball in a balanced triangular urn 
with ball replacement matrix given by $\left(\begin{smallmatrix}\alpha & \beta \\ 0 & \gamma \end{smallmatrix}\right)$, $\alpha,\beta,\gamma\in\N$, $\gamma=\alpha+\beta$, 
is given by the exact formula 
\begin{equation*}
\begin{split}
\E\big(\auffak{X_n}s\big)&= \frac{\Gamma(n+\frac{w_0+b_0+s\alpha}{\gamma})\Gamma(\frac{w_0+b_0}{\gamma})\Gamma(\frac{w_0}\alpha +s )}{\Gamma(\frac{w_0+b_0+s\alpha}{\gamma})\Gamma(n+\frac{w_0+b_0}{\gamma})\Gamma(\frac{w_0}\alpha)},
\end{split}
\end{equation*}
where $w_0$, $b_0$ denote the initial number of white and black balls, respectively.
The factorial moments of $\hat{X}_n=X_n-\frac{w_0}\alpha$ are for $\min\{n,b_0\}\to\infty$ asymptotically of mixed Poisson type 
with a gamma mixing distribution $X\law\gamma(\frac{w_0}\alpha ,1)$, and scale parameter $\lambda_{n,b_0}=\Big(\frac{n+\frac{b_0}\gamma}{\frac{b_0}\gamma}\Big)^{\frac{\alpha}\gamma}-1$, 
\[
\E(\fallfak{\hat{X}_n}{s})=(\lambda_{n,b_0})^s\cdot \frac{\Gamma(\frac{w_0}\alpha +s )}{\Gamma(\frac{w_0}\alpha)}(1+o(1)).
\]

\begin{itemize}
\item[(i)] for $\lambda_{n,b_0}\to\infty$, the random variable $\frac{\hat{X}_{n}}{\lambda_{n,b_0}}$ converges in distribution, with convergence of all moments, to $X$. 

\item[(ii)] for $\lambda_{n,b_0}\to\rho\in(0,\infty)$, the random variable $\hat{X}_{n}$ converges in distribution, with convergence of all moments, to $Y\law\MPo(\rho X)$. 
\end{itemize}
\end{theorem}
\begin{remark}
\label{MOMSEQExampleTriangularRemark}
It is well known from the works of Puyhaubert~\cite{FlaDumPuy2006,Puy2005} and Janson~\cite{Janson2006}, 
that for fixed $b_0$ the random variable $X_n/n^{\frac{\alpha}\gamma}$ tends to a random variable $Z$, depending on the initial condition $w_0,b_0$ and also $\alpha$ and $\gamma$.
Its power moments are given by
\[
\E(Z^s)=\frac{\Gamma(\frac{w_0+b_0}{\gamma})\Gamma(\frac{w_0}\alpha +s )}{\Gamma(\frac{w_0+b_0+s\alpha}{\gamma})\Gamma(\frac{w_0}\alpha)},\quad s\ge 1;
\] 
for more details about the nature of this random variable we refer the reader to~\cite{FlaDumPuy2006,Janson2006}.
This result can easily be re-obtained using the explicit expression for the rising factorial moments of $X_n$ and the method of moments.
We obtain the gamma mixing distribution $X\law\gamma(\frac{w_0}\alpha ,1)$ from $Z=Z_{b_0}$ 
as follows:
\[
b_0^{\frac{\alpha}\gamma}Z_{b_0}\claw X,\quad \text{for }b_0\to\infty,
\]
with convergence of all moments.
\end{remark}

\begin{proof}[Proof of Theorem~\ref{thm:triangular_urn}]
Let $Y$ denote a random variable with \emph{rising factorial moments} $\E(\auffak{Y}s)=\E(Y(Y+1)\dots(Y+s-1))$ satisfying an expansion of mixed Poisson type, $\E(\auffak{Y}s)=\rho^s \cdot \mu_s$, for $s\ge 1$, with $\mu_s\ge 0$. We obtain the (falling) factorial moments using the binomial theorem for rising factorials (see~\cite{Concrete}):
\begin{equation*}
\begin{split}
\fallfak{x}{s}&=\auffak{(x-s+1)}{s}=\sum_{\ell=1}^{s}\binom{s}\ell \cdot\auffak{x}{\ell}\cdot\auffak{(-s+1)}{s-\ell}\\
&=\sum_{\ell=1}^{s}\binom{s}\ell \auffak{x}{\ell}(-1)^{s-\ell}\fallfak{(s-1)}{s-\ell},\quad s\ge 1.
\end{split}
\end{equation*}
Moreover, we can obtain the rising factorial moments of the shifted random variable $\hat{X}_n=X_n-\frac{w_0}a$
by using again the binomial theorem
\[
\auffak{(x+c)}{\ell}=\sum_{j=0}^\ell\binom\ell{j} \cdot\auffak{x}{j}\cdot\auffak{c}{\ell-j},\quad \ell\ge 0.
\]
This implies that we can express the factorial moments of $\hat{X}_n$ in terms of the rising factorial moments of $X_n$ by combining the two identities above in the following way.
\begin{equation*}
\begin{split}
\E(\fallfak{\hat{X}_n}{s})&=\sum_{\ell=1}^{s}\binom{s}\ell \E(\auffak{\hat{X}_n}{\ell})(-1)^{s-\ell}\fallfak{(s-1)}{s-\ell}\\
&=\sum_{\ell=1}^{s}\binom{s}\ell(-1)^{s-\ell}\fallfak{(s-1)}{s-\ell}\sum_{j=0}^\ell\binom\ell{j}\E\big(\auffak{X_n}\ell\big)\auffak{(-w_0/\alpha)}{\ell-j}.
\end{split}
\end{equation*}
Next we use the asymptotic expansion of the rising factorial moments of $X_n$, 
\[
\E\big(\auffak{X_n}s\big)= (\lambda_{n,b_0}+1)^{s}\frac{\Gamma(\frac{w_0}\alpha +s )}{\Gamma(\frac{w_0}\alpha)}(1+o(1)),\quad s\ge 1,
\]
where $\lambda_{n,b_0}=\left(\frac{n+\frac{b_0}\gamma}{\frac{b_0}\gamma}\right)^{\frac{\alpha}c}-1$. 
Interchanging summations, and collecting powers of $\lambda_{n,b_0}$ leads to the expansion
\begin{equation*}
\begin{split}
\E(\fallfak{\hat{X}_n}{s})&=s!\sum_{i=0}^{s}\lambda_{n,b_0}^i\bigg[\sum_{j=i}^s   \binom{j}i(-1)^j\binom{\frac{w_0}\alpha +j-1}{j}\times\\
&\qquad\times\sum_{\ell=j}^{s}\binom{\ell-1}{\ell-j} \binom{-\frac{w_0}\alpha }{s-\ell}(-1)^{\ell}\bigg](1+o(1)).
\end{split}
\end{equation*}
Next, using the hypergeometric form of the Vandermonde convolution, see~\cite[p.~212]{Concrete}, we obtain for the inner sum
\[
\sum_{\ell=j}^{s}\binom{\ell-1}{\ell-j} \binom{-\frac{w_0}\alpha }{s-\ell}(-1)^{\ell}= 
\frac{\binom{-\frac{w_0}\alpha }{s-j} \binom{j-s-\frac{w_0}\alpha }{j}}{\binom{-\frac{w_0}\alpha+j-1 }{j}}.
\]
We get further 
\begin{align*}
& s!\sum_{j=i}^s   \binom{j}i\binom{-\frac{w_0}\alpha }{s-j} \binom{j-s-\frac{w_0}\alpha }{j}(-1)^j
=\delta_{s,i}\,\cdot s!(-1)^s \binom{-\frac{w_0}\alpha}{s}\\
& \quad = \delta_{s,i}\,\cdot  \frac{\Gamma(\frac{w_0}\alpha +s )}{\Gamma(\frac{w_0}\alpha)},
\end{align*}
where $\delta_{s,i}$ denotes the Kronecker-delta function. This proves the stated result.
\end{proof}

\section{Mixed Poisson-Rayleigh laws\label{RayPoi}}

In the analysis of various combinatorial objects as, e.g., lattice paths, trees and mappings, the Rayleigh distribution occurs frequently. In this section we give several examples, where during the study of such objects, a mixed Poisson distribution with Rayleigh mixing distribution occurs in a natural way. Apart from the first example, the occurrence and proof of the mixed Poisson distribution is novel, best to our knowledge.

\subsection{The number of inversions in labelled tree families}\label{Ssec:InvTrees}
Consider a rooted labelled tree $T$, where the nodes of $T$ are labelled with distinct integers (usually of the set $\{1, 2, \dots, |T|\}$, with $|T|$ the size, i.e., the number of vertices, of $T$).
An \emph{inversion} in $T$ is a pair $(i,j)$ of vertices (we may always identify a vertex with its label), such that $i > j$ and $i$ lies on the unique
path from the root node of $T$ to $j$ (thus $i$ is an ascendant of $j$ or, equivalently, $j$ is a descendant of $i$).
Given a tree family, we introduce the r.v.\ $I_{n,j}$, which counts, for a random tree of size $n$, the number of inversions induced by the node labelled $j$, $1 \le j \le n$, i.e., it counts the number of inversions of the kind $(i,j)$, with $i > j$ an ancestor of $j$. See Figure~\ref{fig:TreeInversions} for an illustration of the quantity considered.
\begin{figure}
\begin{center}
  \includegraphics[height=2.5cm]{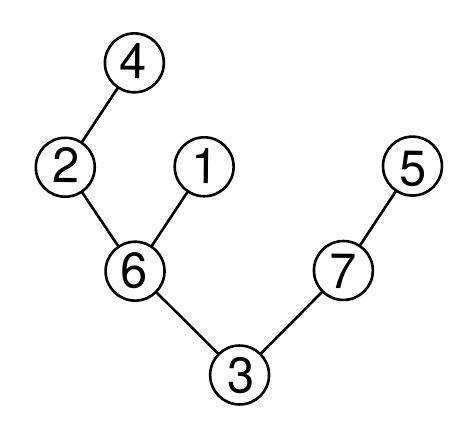}
\end{center}
\caption{A binary labelled tree of size $7$ with a total number of $6$ inversions, namely $(3,1)$, $(6,1)$, $(3,2)$, $(6,2)$, $(6,4)$, $(7,5)$. Thus two inversions each are induced by the nodes
$1$ and $2$, whereas one inversion each is induced by the nodes $4$ and $5$.\label{fig:TreeInversions}}
\end{figure}

Panholzer and Seitz~\cite{PanSeitz2011} studied the r.v.\ $I_{n,j}$ for random trees of so-called labelled simply generated tree families (see, e.g., \cite{FlaSed2009}; note that in the probabilistic literature such tree models are more commonly called Galton-Watson trees), which contain many important tree families as, e.g., ordered, unordered, binary and cyclic labelled trees as special instances.

Formally, a class $\mathcal{T}$ of labelled simply generated trees is defined in the following way:
One chooses a sequence $(\varphi_{\ell})_{\ell \geq 0}$ (the so-called \emph{degree-weight sequence}) of nonnegative real numbers with $\varphi_0 > 0$. 
Using this sequence, the \textit{weight $w(T)$} of each labelled ordered tree (i.e., each labelled rooted tree, 
in which the children of each node are ordered from left to right) is defined by 
$w(T) := \prod_{v \in T} \varphi_{d(v)}$, where by $v \in T$ we mean that $v$ is a vertex of $T$ and 
$d(v)$ denotes the number of children of $v$ (i.e., the outdegree of $v$). The family $\mathcal{T}$ associated to the degree-weight sequence $(\varphi_{\ell})_{\ell \geq 0}$ 
then consists of all trees $T$ (or all trees $T$ with $w(T) \neq 0$) together with their weights.
We let $T_n := \sum_{|T| = n} w(T)$ denote the the \emph{total weight} of all trees of size $n$ in $\mathcal{T}$;
for many important simply generated tree families, $(T_n)_{n \geq 1}$ is a sequence of natural numbers, and then the total weight $T_n$ can be interpreted simply as the 
number of trees of size $n$ in $\mathcal{T}$.

When analysing parameters in a simply generated tree family $\mathcal{T}$ it is common to assume the \emph{random tree model for weighted trees}, i.e.,
when speaking about a random tree of size $n$ one assumes that each tree $T$ in $\mathcal{T}$ of size $n$ is chosen with a probability proportional to its weight $w(T)$, i.e., is chosen with probability $\frac{w(T)}{T_{n}}$.
Under mild conditions on the degree-weight sequence $(\varphi_{\ell})_{\ell \geq 0}$ of a family $\mathcal{T}$ of labelled simply generated trees and assuming the random tree model, in \cite{PanSeitz2011} the following asymptotic formula for the factorial moments of $I_{n,j}$ has been obtained:
\[
\E(\fallfak{I_{n,j}}{s})=\Gamma\big(\frac{s}{2}+1\big)\Big(\frac{2}{\kappa}\Big)^{\frac{s}{2}}\frac{\fallfak{(n-j)}{s}}{n^{\frac{s}{2}}} \cdot \big(1+o(1)\big),
\]
where the constant $\kappa$ depends on the particular tree family, i.e., on the degree-weight sequence, and is given in \cite{PanSeitz2011}. Consequently, 
an application of Lemma~\ref{MOMSEQMainLemma} and taking into account Example~\ref{ExampleRayleigh} yields the following result, which adds to the results of \cite{PanSeitz2011} the characterization of the limiting distribution as a mixed Poisson distribution.
\begin{coroll}
The random variable $I_{n,j}$, which counts the number of inversions induced by node $j$ in a random labelled simply generated tree of size $n$ has for $n\to\infty$ and arbitrary $1 \le j=j(n) \le n$ asymptotically factorial moments of mixed Poisson type with a Rayleigh mixing distribution $X$ and scale parameter $\lambda_{n,j}=\sqrt{\frac{1}{\kappa}}\frac{n-j}{\sqrt{n}}$ (with constant $\kappa$ given in \cite{PanSeitz2011}),
\[
\E(\fallfak{I_{n,j}}{s})=\lambda_{n,j}^{s} \, 2^{\frac{s}{2}} \, \Gamma\big(\frac{s}{2}+1\big) \cdot \big(1+o(1)\big).
\]
\begin{itemize}
\item[(i)] for $\lambda_{n,j}\to\infty$, the random variable $\frac{I_{n,j}}{\lambda_{n,j}}$ converges in distribution, with convergence of all moments, to $X$. 

\item[(ii)] for $\lambda_{n,j}\to\rho\in(0,\infty)$, the random variable $I_{n,j}$ converges in distribution, with convergence of all moments, to a mixed Poisson distributed random variable $Y\law \MPo(\rho X)$.
\end{itemize}
Moreover, the random variable $Y\law \MPo(\rho X)$ converges for $\rho\to\infty$, after scaling, to its mixing distribution $X$: $\frac{Y}{\rho}\claw X$, with convergence of all moments.
\end{coroll}
\begin{remark}
Considering $\lambda_{n,j}$, the critical phase occurs at $j =  n - \Theta(\sqrt{n})$: $\E(I_{n,j}) \to \infty$, for $n-j \gg \sqrt{n}$, whereas $\E(I_{n,j}) \to 0$, for $n-j \ll \sqrt{n}$.
\end{remark}

\subsection{Record-subtrees in Cayley-trees\label{Ssec:Records}}
Given a rooted labelled tree $T$, a \emph{min-record} (or simply \emph{record}, for short) is a node $x \in T$, which has the smallest label amongst all nodes on the (unique) path from the root-node of $T$ to $x$. Let us assume that $\{r_{1}, \dots, r_{k}\}$ is the set of records of $T$; then this set naturally induces a decomposition of the tree $T$ into what is called here \emph{record-subtrees} $\{S_{1}, \dots, S_{k}\}$: $S_{i}$, $1 \le i \le k$, is the largest subtree rooted at the record $r_{i}$ not containing any of the remaining records $r_{1}, \dots, r_{i-1}, r_{i+1}, \dots, r_{k}$. In other words, a record-subtree $S$ is a maximal subtree (i.e., it is not properly contained in another such subtree) of $T$ with the property that the root-node of $S$ has the smallest label amongst all nodes of $S$.
See Figure~\ref{fig:RecordSubtrees} for an illustration of these quantities.
\begin{figure}
\begin{center}
  \includegraphics[height=2.5cm]{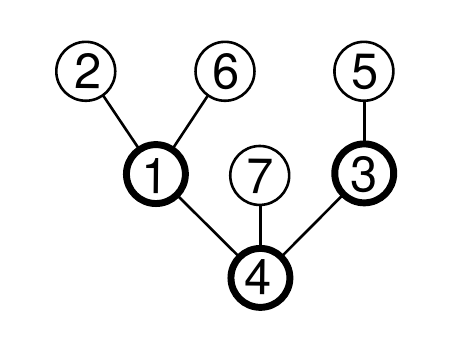} \quad \raisebox{1.1cm}{\Large{$\Rightarrow$}} \quad \includegraphics[height=2.5cm]{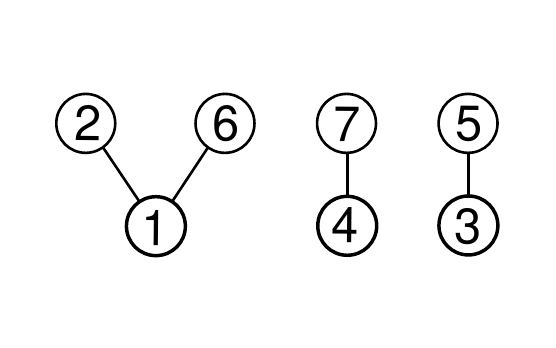}
\end{center}
\caption{A labelled tree of size $7$ with three min-records: nodes $1$, $3$, and $4$. This decomposes the original tree into three record subtrees, one of them is of size $3$ and two of them are of size $2$.\label{fig:RecordSubtrees}}
\end{figure}

In the following we will study the occurrence of record-subtrees of a given size for one of the most natural random tree models, namely random rooted labelled unordered trees, often called random Cayley-trees. A Cayley-tree is a rooted tree $T$, where the nodes of $T$ are labelled with distinct integers of $\{1, 2, \dots, |T|\}$ and where the children of any node $x \in T$ are not equipped with any left-to-right ordering (i.e., we may think that each node in $T$ has a possibly empty set of children). 
Combinatorially, the family $\mathcal{T}$ of Cayley-trees can be described formally via the \textsc{Set} construction as
\begin{equation}
\label{MOMSEQCayley}
\mathcal{T} = \bigcirc \ast \text{\sc{Set}}(\mathcal{T}).
\end{equation}
Note that Cayley-trees are a particular family of labelled simply generated trees as described in Section~\ref{Ssec:InvTrees}, where the degree-weight sequence $(\varphi_{\ell})_{\ell \ge 0}$ is given by $\varphi_{\ell} = \frac{1}{\ell!}$. It is well-known that there are exactly $T_{n} = n^{n-1}$ different Cayley-trees of size $n$ (see, e.g., \cite{FlaSed2009,Stanley}) and in the random tree model, which we will always assume here, each of these trees may occur with the same probability when considering a size-$n$ tree.

The r.v.\ $R_{n}$, counting the number of records in a random size-$n$ Galton-Watson tree (i.e., in a simply generated tree), has been studied by Janson~\cite{Janson2006b} showing (after a suitable scaling by $\frac{1}{\sqrt{n}}$) a Rayleigh limiting distribution result; in particular, for Cayley-trees it holds $\frac{R_{n}}{\sqrt{n}} \claw \text{Rayleigh}(1)$. Here we introduce the r.v.\ $R_{n,j}$, which counts the number of record-subtrees of size $j$ in a random Cayley-tree of size $n$.
Of course, the random variables $R_{n,j}$, $1 \le j \le n$, are a refinement of $R_{n}$ and are related by the identity
\begin{equation*}
  R_{n} = \sum_{j=1}^{n} R_{n,j}.
\end{equation*}
As has been pointed out already in \cite{Janson2006b}, records in trees are closely related to a certain node removal procedure for trees. Starting with a tree $T$ one chooses a node $x \in T$ at random and cuts off the subtree $T''$ of $T$ rooted at $x$, and iterates this cutting procedure with the remaining subtree $T'$ until only the empty subtree remains. The r.v.\ $C_{n}^{[v]}$ counting the number of (vertex) cuts required to cut-down the whole tree by this cutting procedure when starting with a random Cayley-tree of size $n$ is then distributed as $R_{n}$, i.e., $R_{n} \law C_{n}^{[v]}$. 
We can extend this relation by considering the r.v.\ $C_{n,j}^{[v]}$ counting the number of subtrees of size $j$, which are cut-off during the (vertex) cutting procedure when starting with a random Cayley-tree of size $n$, where it holds $R_{n,j} \law C_{n,j}^{[v]}$. This can be seen easily by means of coupling arguments given in \cite{Janson2006b}: consider the node-removal procedure, where, starting with a tree $T$, in each step the node with smallest label amongst all nodes in the remaining tree is selected and together with all its descendants detached from the tree. Then it holds that node $x$ is a min-record in the tree $T$ if and only if node $x$ is selected as a vertex cut during this node-removal procedure and in this case the record-subtree (and thus its size) rooted at $x$ corresponds to the subtree (with its respective size), which is removed in this cut.

We will show that $R_{n,j}$ and thus also $C_{n,j}^{[v]}$ has factorial moments of mixed Poisson type yielding the following theorem.
\begin{theorem}\label{the:Records}
The random variable $R_{n,j}$, counting the number of record-subtrees of size $j$ in a random Cayley-tree of size $n$, has, for $n\to\infty$ and arbitrary $1 \le j=j(n) \le n$, 
asymptotically factorial moments of mixed Poisson type with a Rayleigh mixing distribution $X$ and scale parameter $\lambda_{n,j} = \frac{\sqrt{n} j^{j-1}}{j! e^{j}}$:
\begin{equation*}
\E(\fallfak{R_{n,j}}s) = \lambda_{n,j}^{s} \, 2^{\frac{s}{2}} \, \Gamma\big(\frac{s}{2}+1\big) \cdot \big(1+o(1)\big), \enspace \text{for $j=o(n)$}, \quad
\E(\fallfak{R_{n,j}}s) = \mathcal{O}\big(\lambda_{n,j}^{s}\big).
\end{equation*}
\begin{itemize}
\item[(i)] for $\lambda_{n,j}\to\infty$, the random variable $\frac{R_{n,j}}{\lambda_{n,j}}$ converges in distribution, with convergence of all moments, to $X$. 

\item[(ii)] for $\lambda_{n,j}\to\rho\in[0,\infty)$, the random variable $R_{n,j}$ converges in distribution, with convergence of all moments, to a mixed Poisson distributed random variable $Y\law \MPo(\rho X)$. 
\end{itemize}
Moreover, the random variable $Y\law \MPo(\rho X)$ converges, for $\rho\to\infty$, after scaling, to its mixing distribution $X$: $\frac{Y}{\rho}\claw X$, with convergence of all moments.
\end{theorem}
\begin{remark}
Stirling's formula for the Gamma function gives $\lambda_{n,j} = \frac{\sqrt{n} j^{j-1}}{j! e^{j}} \sim \sqrt{\frac{n}{2 \pi j^{3}}}$, for $j,n \to \infty$, thus the critical phase occurs at $j = \Theta(n^{\frac{1}{3}})$: $\E(R_{n,j}) \to \infty$, for $j \ll n^{\frac{1}{3}}$, whereas $\E(R_{n,j}) \to 0$, for $j \gg n^{\frac{1}{3}}$. In all succeeding examples in this section the critical phase behaviour also occurs at $j=\Theta(n^{\frac{1}{3}})$.
\end{remark}
\begin{proof}[Proof of Theorem~\ref{the:Records}]
We consider the description of the problem via the node removal procedure. This immediately yields the following stochastic recurrence for the r.v.\ $R_{n,j} \law C_{n,j}^{[v]}$:
\begin{equation}\label{eqn:Record_rec}
  R_{n,j} \law R_{K_{n},j} + \boldsymbol{1}_{\{n-K_{n}=j\}}, \quad \text{for $1 \le j \le n$},
\end{equation}
with $R_{n,j} = 0$, for $0 \le n < j$, and where the r.v.\ $K_{n}$ measures the size of the subtree remaining after selecting a random node and removing the subtree rooted at it from a randomly selected size-$n$ Cayley-tree.

In the following we will compute the splitting probabilities $p_{n,k} := \mathbb{P}\{K_{n} = k\}$, with $0 \le k \le n-1$, and by doing this we also show that the recurrence \eqref{eqn:Record_rec} is indeed valid, i.e., that the subtree $T'$ (let us assume of size $k \ge 1$) remaining after removing the subtree containing the selected node $x$ of a random size-$n$ Cayley-tree $T$ is (after an order-preserving relabelling of the nodes) again a random Cayley-tree of size $k$ (the so-called \emph{random preservation property} holds).

This can be done by simple combinatorial reasoning. Consider a pair $(T,x)$ of a size-$n$ Cayley-tree $T$ and a node $x \in T$. When detaching the subtree rooted at $x$ from $T$, we obtain a pair $(T',T'')$
of subtrees with $T''$ containing $x$ and $T'$ the possibly empty remaining subtree. Of course, $T'$ is the empty subtree exactly if $x$ is the root-node of $T$ and thus there are exactly $T_{n}$ pairs $(T,x)$ yielding $|T'| = k = 0$. Let us now assume that $1 \le |T'| = k \le n-1$. After an order-preserving relabelling of $T'$ and $T''$ with labels $\{1, \dots, k\}$ and $\{1, \dots, n-k\}$, respectively, both subtrees are Cayley-trees of size $k$ and $n-k$, respectively. Consider now a particular pair $(\tilde{T}',\tilde{T}'')$ of Cayley-trees of size $|\tilde{T}'|=k$ and $|\tilde{T}''|=n-k$, respectively, and let us count the number of pairs $(T,x)$, with $T$ a size-$n$ Cayley-tree and $x \in T$, yielding the pair $(\tilde{T}',\tilde{T}'')$ of subtrees after cutting. By constructing such pairs $(T,x)$, one obtains that there are exactly $k \binom{n}{k}$ possibilities ($k$ possible ways of attaching $\tilde{T}''$ to a node in $\tilde{T}'$ and $\binom{n}{k}$ possibilities of distributing the labels $\{1, \dots, n\}$ order-preserving to the subtrees), independent of the chosen pair of trees; thus the random preservation property holds.

Moreover, one obtains that there are exactly $k \binom{n}{k} T_{k} T_{n-k}$ pairs $(T,x)$ splitting after a cut into a pair $(T',T'')$ of subtrees with respective sizes $k$ and $n-k$, for $1 \le k \le n-1$.
Of course, in total there are $n T_{n}$ pairs $(T,x)$ and thus we get the following result for the splitting probabilities $p_{n,k}$:
\begin{equation*}
  p_{n,k} = \mathbb{P}\{K_{n}=k\} = \begin{cases} \frac{k \tilde{T}_{k} \tilde{T}_{n-k}}{n \tilde{T}_{n}}, & \quad 1 \le k \le n-1, \\ \frac{1}{n}, & \quad k=0, \end{cases}
\end{equation*}
where we use throughout this section the abbreviation
\begin{equation}
\label{eqn:Tntilde}
\tilde{T}_{n} := \frac{T_{n}}{n!} = \frac{n^{n-1}}{n!}.
\end{equation}

In order to treat the stochastic recurrence \eqref{eqn:Record_rec} and to compute the asymptotic behaviour of the (factorial) moments of $R_{n,j}$ we find it appropriate to introduce the generating function
\begin{equation*}
  F(z,v) := F_{j}(z,v) = \sum_{n \ge 1} \tilde{T}_{n} \E(v^{R_{n,j}}) z^{n} = \sum_{n \ge 1} \sum_{m \ge 0} \tilde{T}_{n} \mathbb{P}\{R_{n,j}=m\} z^{n} v^{m}.
\end{equation*}
Starting with \eqref{eqn:Record_rec}, straightforward computations (which are omitted here) yield then the following differential equation
\begin{equation*}
  F_{z}(z,v) = T(z) F_{z}(z,v) + \frac{T(z)}{z} + \tilde{T}_{j} z^{j-1} (v-1) + \tilde{T}_{j} z^{j} F_{z}(z,v) (v-1),
\end{equation*}
where the so-called tree function (the exponential generating function of the number of size-$n$ Cayley-trees) appears:
\begin{equation}
\label{eqn:Tree-function}
T(z) = \sum_{n \ge 1} \tilde{T}_{n} z^{n} = \sum_{n \ge 1} \frac{n^{n-1}}{n!} z^{n}.
\end{equation}
Simple manipulations and using the well-known functional equation of the tree function (which is thus closely related to the Lambert-W function),
\begin{equation*}
  T(z) = z e^{T(z)},
\end{equation*}
give then the following explicit formula for the derivative of $F(z,v)$ w.r.t.\ $z$:
\begin{equation}\label{eqn:Fzv_formula}
  F_{z}(z,v) = \frac{e^{T(z)} + (v-1) \tilde{T}_{j} z^{j-1}}{1-T(z)-(v-1)\tilde{T}_{j}z^{j}}.
\end{equation}

To get the (factorial) moments of $R_{n,j}$ we use the substitution $w := v-1$ and introduce $\tilde{F}(z,w) := F(z,w+1)$. 
Extracting coefficients $[w^{s}]$, $s \ge 1$, from
\begin{equation*}
  \tilde{F}_{z}(z,w) = \frac{e^{T(z)} + w \tilde{T}_{j} z^{j-1}}{1-T(z)-w\tilde{T}_{j}z^{j}} = \frac{1}{z} \Big(\frac{1}{1-T(z)-w\tilde{T}_{j}z^{j}}-1\Big),
\end{equation*}
easily gives
\begin{equation*}
  [w^{s}] \tilde{F}_{z}(z,w) = \frac{\tilde{T}_{j}^{s} z^{js-1}}{(1-T(z))^{s+1}}.
\end{equation*}
We further obtain
\begin{equation}\label{eqn:coeff_Fzw}
  [z^{n}w^{s}] \tilde{F}(z,w) = \frac{1}{n} [z^{n-1} w^{s}] \tilde{F}_{z}(z,w) = \frac{\tilde{T}_{j}^{s}}{n} [z^{n-js}] \frac{1}{(1-T(z))^{s+1}}.
\end{equation}

It is not difficult to extract coefficients from \eqref{eqn:coeff_Fzw} and stating an explicit formula for $[z^{n}w^{s}] \tilde{F}(z,w)$ and the factorial moments of $R_{n,j}$; however, for asymptotic considerations it is easier to use well-known analytic properties of the tree function $T(z)$ and deduce from it the desired asymptotic growth behaviour of $\E(R_{n,j}^{\underline{s}})$.
Namely, we use standard applications of so-called singularity analysis, see~\cite{FlaSed2009}, to transfer the local behaviour of a generating function in a complex neighbourhood of the dominant singularity (i.e., the singularity of smallest modulus; we are here only concerned with functions with a unique dominant singularity) to the asymptotic behaviour of its coefficients.
It holds (see, e.g., \cite{FlaSed2009}) that the tree function $T(z)$ has a unique dominant singularity (a branch point) at $z=e^{-1}$, where the function evaluates to $T(e^{-1}) = 1$ and where it admits the following local expansion:
\begin{equation}\label{eqn:Tz_expansion}
  T(z) = 1 - \sqrt{2} \sqrt{1-ez} + \mathcal{O}(1-ez).
\end{equation}
Thus the function $(1-T(z))^{-s-1}$ also has a unique dominant singularity at $z=e^{-1}$ with the following local bound:
\begin{equation*}
  \frac{1}{(1-T(z))^{s+1}} = \mathcal{O}\Big(\frac{1}{(1-ez)^{\frac{s+1}{2}}}\Big).
\end{equation*}
Singularity analysis yields then
\begin{equation*}
  [z^{n}] \frac{1}{(1-T(z))^{s+1}} = \mathcal{O}\Big(e^{n} n^{\frac{s-1}{2}}\Big).
\end{equation*}
Therefore, \eqref{eqn:coeff_Fzw} yields
\begin{equation*}
  [z^{n}w^{s}] \tilde{F}(z,w) = \mathcal{O}\Big(\big(\frac{\tilde{T}_{j}}{e^{j}}\big)^{s} e^{n} n^{\frac{s-3}{2}}\Big).
\end{equation*}
This, together with Stirling's formula for the Gamma function~\eqref{MOMSEQStirling}, shows the following bound for the $s$-th moments of $R_{n,j}$, which holds uniformly for all $1 \le j  \le n$:
\begin{equation}\label{eqn:Rnj_bound}
  \E(R_{n,j}^{\underline{s}}) = \frac{s!}{\tilde{T}_{n}} [z^{n} w^{s}] \tilde{F}(z,w) = \mathcal{O}\Big(\big(\frac{\tilde{T}_{j}}{e^{j}}\big)^{s} \frac{e^{n} n^{\frac{s-3}{2}}}{\tilde{T}_{n}}\Big)
  = \mathcal{O}\Big(\big(\frac{\sqrt{n} \, \tilde{T}_{j}}{e^{j}}\big)^{s}\Big).
\end{equation}

To get the mixed Poisson behaviour for $j=o(n)$ we use the refined expansion
\begin{equation*}
  \frac{1}{(1-T(z))^{s+1}} = \frac{1}{2^{\frac{s+1}{2}} (1-ez)^{\frac{s+1}{2}}} \cdot \Big(1+\mathcal{O}\big(\frac{1}{\sqrt{1-ez}}\big)\Big),
\end{equation*}
locally around $z=e^{-1}$, which can be obtained from~\eqref{eqn:Tz_expansion}. Singularity analysis gives then the expansion
\begin{equation}\label{eqn:znTz}
  [z^{n}] \frac{1}{(1-T(z))^{s+1}} = \frac{e^{n} n^{\frac{s-1}{2}}}{2^{\frac{s+1}{2}} \big(\frac{s-1}{2}\big)!} \cdot \Big(1+\mathcal{O}\big(\frac{1}{\sqrt{n}}\big)\Big).
\end{equation}
Thus we get for $j=o(n)$ the stated behaviour of the the $s$-th factorial moments of $R_{n,j}$:
\begin{align}
  & \E(R_{n,j}^{\underline{s}}) = \frac{s!}{\tilde{T}_{n}} [z^{n}w^{s}] \tilde{F}(z,w) = \frac{s! \tilde{T}_{j}^{s} e^{n-js} (n-js)^{\frac{s-1}{2}}}{\tilde{T}_{n} n 2^{\frac{s+1}{2}} \big(\frac{s-1}{2}\big)!} \cdot \Big(1+\mathcal{O}\big(\frac{1}{\sqrt{n-js}}\big)\Big) \notag \\
  & \quad = \Big(\frac{\sqrt{n} \, \tilde{T}_{j}}{e^{j}}\Big)^{s} \frac{s! \sqrt{\pi}}{2^{\frac{s}{2}} \big(\frac{s-1}{2}\big)!} \cdot \big(1+o(1)\big)
  = \Big(\frac{\sqrt{n} \, \tilde{T}_{j}}{e^{j}}\Big)^{s} \, 2^{\frac{s}{2}} \, \Gamma\big(\frac{s}{2}+1\big) \cdot \big(1+o(1)\big), \label{eqn:Rnj_moments}
\end{align}
where we used in the final step the duplication formula of the factorials:
\begin{equation}\label{eqn:duplication}
  \big(\frac{s-1}{2}\big)! = \frac{\sqrt{\pi} \, (s-1)!}{2^{s-1} \big(\frac{s}{2}-1\big)!}.
\end{equation}
The mixed Poisson limit law with Rayleigh mixing distribution as stated in Theorem~\ref{the:Records} follows then from \eqref{eqn:Rnj_bound} and \eqref{eqn:Rnj_moments} by applying Lemma~\ref{MOMSEQMainLemma}.
\end{proof}

\subsection{Edge-cutting in Cayley-trees\label{Ssec:Cutting}}
The following prominent edge-cutting procedure for trees is closely related to the node removal procedure considered in Section~\ref{Ssec:Records}. Starting with a tree $T$ one chooses an edge $e \in T$ and removes it from $T$. After that $T$ decomposes into two subtrees $T'$ and $T''$, where we assume that $T'$ contains the original root of $T$. We discard the subtree $T''$ and continue the edge-cutting procedure with $T'$ until we have isolated the root-node of the original tree $T$. This cutting-down procedure has been introduced in \cite{MeiMoo1970}, where the number of random cuts $C_{n}$ to isolate the root-node of a random Cayley-tree of size $n$, where in each cutting-step an edge from the remaining tree is chosen uniformly at random, has been studied yielding asymptotic formul{\ae} for the first two moments of $C_{n}$. The Rayleigh limiting distribution of $C_{n}$ for Cayley-trees and other families of simply generated trees has been obtained in \cite{Pan2003,Pan2006} and in a more general setting by Janson~\cite{Janson2006b}; in particular, for Cayley-trees one obtains $\frac{C_{n}}{\sqrt{n}} \claw \text{Rayleigh}(1)$. Moreover, in \cite{Janson2006b} it was shown in general that for Galton-Watson tree families (thus containing Cayley-trees) the random variables $C_{n}$ and $C_{n}^{[v]}$ (as introduced in Section~\ref{Ssec:Records}) for the edge-cutting procedure and the node-removal procedure, respectively, have the same limiting distribution behaviour. A number of works have analysed the edge-cutting procedure and related processes using the connection between Cayley-trees and the so-called continuum random tree,
in particular see the work of Addagio-Berry, Broutin and Holmgren~\cite{ABBH2013}, and the recent works of Bertoin~\cite{Bertoin2012,Bertoin2012+}.

In this section we consider a refinement of the r.v.\ $C_{n}$ for Cayley-trees, namely we study the behaviour of the r.v.\ $C_{n,j}$, counting the number of subtrees of size $j$ cut-off during the random edge-cutting procedure when starting with a random size-$n$ Cayley-tree until the root-node is isolated. Of course, it holds
\begin{equation*}
  C_{n} = \sum_{j=1}^{n-1} C_{n,j}.
\end{equation*}
Before continuing we want to remark that an alternative description of the problem can be given via edge-records in edge-labelled trees: given a size-$n$ tree $T$ we first distribute the labels $\{1, \dots, |T|-1\}$ randomly to the edges of $T$. An edge-record in $T$ is then an edge $e=(x,y)$, where $y$ is a child of $x$, with smallest label amongst all edges on the path from the root-node of $T$ to $y$. Analogous to Section~\ref{Ssec:Records} (and stated already in \cite{Janson2006b}) one gets that the r.v.\ $R_{n}^{[e]}$ counting the number of edge-records in a random size-$n$ Cayley-tree is distributed as $C_{n}$, i.e., $R_{n}^{[e]} \law C_{n}$. Moreover, the edge-records $e_{1}, \dots, e_{k}$ of an edge-labelled tree $T$ naturally decompose $T$ into the root-node and $k$ \emph{edge-record subtrees} $S_{1}, \dots, S_{k}$, obtained from $T$ by removing the root-node of $T$ and all edges $e_{1}, \dots, e_{k}$. Again we can introduce the r.v.\ $R_{n,j}^{[e]}$, which counts the number of edge-record subtrees of size $j$ in a random edge-labelled Cayley-tree of size $n$. It is then immediate to see that $R_{n,j}^{[e]} \law C_{n,j}$.

In Figure~\ref{fig:CutDown} we illustrate the edge-cutting procedure for a particular tree.
\begin{figure}
\begin{center}
  \includegraphics[height=2.5cm]{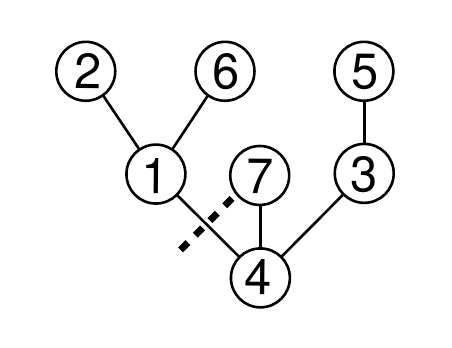} \quad \raisebox{0.4cm}{\Large{$\Rightarrow$}} \quad \includegraphics[height=2.5cm]{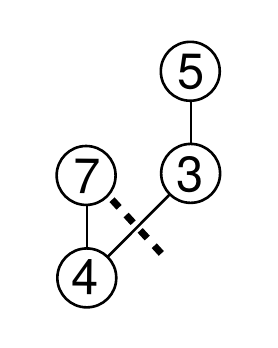} \quad \raisebox{0.4cm}{\Large{$\Rightarrow$}} \quad \includegraphics[height=2.5cm]{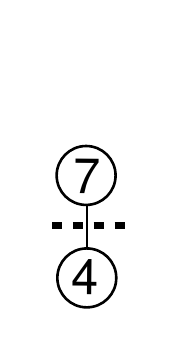} \quad \raisebox{0.4cm}{\Large{$\Rightarrow$}} \quad \includegraphics[height=2.5cm]{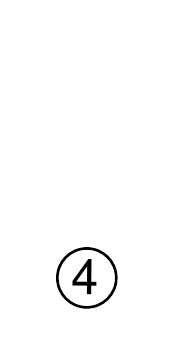}
\end{center}
\caption{Cutting-down a tree of size $7$ with three edge-cuts. The respective sizes of the subtrees cut-off during this procedure are $3$, $2$, and $1$.\label{fig:CutDown}}
\end{figure}

In the following theorem we state that $C_{n,j}$ (and thus also $R_{n,j}^{[e]}$) has factorial moments of mixed Poisson type with a Rayleigh mixing distribution. The method of proof is analogous to the one presented in Section~\ref{Ssec:Records}, but due to the less explicit nature of the formul{\ae} occurring, the proof steps are more technical and a bit lengthy.
\begin{theorem}\label{the:Edgecut}
The random variable $C_{n,j}$, counting the number of subtrees of size $j$, which are cut-off during the edge-cutting procedure starting with a random Cayley-tree of size $n$, has, for $n\to\infty$ and arbitrary $1 \le j=j(n) \le n-1$, 
asymptotically factorial moments of mixed Poisson type with a Rayleigh mixing distribution $X$ and scale parameter $\lambda_{n,j} = \frac{\sqrt{n} j^{j-1}}{j! e^{j}}$:
\begin{equation*}
\E(\fallfak{C_{n,j}}s) = \lambda_{n,j}^{s} \, 2^{\frac{s}{2}} \, \Gamma\big(\frac{s}{2}+1\big) \cdot \big(1+o(1)\big), \enspace \text{for $j=o(n)$}, \quad
\E(\fallfak{C_{n,j}}s) = \mathcal{O}\big(\lambda_{n,j}^{s}\big).
\end{equation*}
\begin{itemize}
\item[(i)] for $\lambda_{n,j}\to\infty$, the random variable $\frac{C_{n,j}}{\lambda_{n,j}}$ converges in distribution, with convergence of all moments, to $X$. 

\item[(ii)] for $\lambda_{n,j}\to\rho\in[0,\infty)$, the random variable $C_{n,j}$ converges in distribution, with convergence of all moments, to a mixed Poisson distributed random variable $Y\law \MPo(\rho X)$. 
\end{itemize}
\end{theorem}
\begin{remark}
According to Theorems~\ref{the:Records} and \ref{the:Edgecut} the r.v.\ $C_{n,j}$ and $C_{n,j}^{[v]}$ (and thus also $R_{n,j}^{[e]}$) and $R_{n,j}$, for the edge- and vertex-versions of the cutting procedures as considered in Section~\ref{Ssec:Records}-\ref{Ssec:Cutting}, have the same limiting distribution behaviour. Janson~\cite{Janson2006b} was able to bound the difference between the random variables $R_{n}$ and $R_{n}^{[e]}$ (i.e., between the number of node- and edge-records) in a suitable metric and thus to show directly the same limiting behaviour of these r.v. It would be interesting to extend his proof technique to the refined quantities studied here.
\end{remark}
\begin{proof}[Proof of Theorem~\ref{the:Edgecut}]
Decomposing a tree according to the first cut of the edge-cutting procedure immediately yields the following stochastic recurrence for the r.v.\ $C_{n,j}$:
\begin{equation}\label{eqn:EdgeCutting_rec}
  C_{n,j} \law C_{K_{n}^{[e]},j} + \boldsymbol{1}_{\{n-K_{n}^{[e]}=j\}}, \quad \text{for $1 \le j < n$},
\end{equation}
with $C_{n,j} = 0$, for $1 \le n \le j$, and where the r.v.\ $K_{n}^{[e]}$ measures the size of the subtree containing the root of the original tree after cutting a random edge from a randomly selected size-$n$ Cayley-tree.
It is well-known~\cite{MeiMoo1970} (and can be shown completely analogous to the computations in the proof of Theorem~\ref{the:Records}) that the \emph{random preservation property} of Cayley-trees also holds for the edge-cutting procedure (thus implying correctness of \eqref{eqn:EdgeCutting_rec}) and that the splitting probabilities $p_{n,k}^{[e]}$ (i.e., the distribution of $K_{n}^{[e]}$) are given as follows
(we recall the definition $\tilde{T}_{n} = \frac{n^{n-1}}{n!}$ given in \eqref{eqn:Tntilde}):
\begin{equation*}
  p_{n,k}^{[e]} := \mathbb{P}\{K_{n}^{[e]}=k\} = \frac{k \tilde{T}_{k} \tilde{T}_{n-k}}{(n-1) \tilde{T}_{n}}, \quad \text{for} \enspace 1 \le k \le n-1 \enspace \text{and} \enspace n \ge 2.
\end{equation*}

Again, in order to treat the stochastic recurrence \eqref{eqn:EdgeCutting_rec} we introduce suitable generating functions:
\begin{equation*}
  G_{j}(z,v) := \sum_{n \ge 1} \tilde{T}_{n} \E(v^{C_{n,j}}) z^{n} = \sum_{n \ge 1} \sum_{m \ge 0} \tilde{T}_{n} \mathbb{P}\{C_{n,j}=m\} z^{n} v^{m}.
\end{equation*}
Straightforward computations give then the differential equation
\begin{equation*}
  z \frac{\partial}{\partial z}G_{j}(z,v) - G_{j}(z,v) = z T(z) \frac{\partial}{\partial z}G_{j}(z,v) + (v-1) \tilde{T}_{j} z^{j+1}  \frac{\partial}{\partial z}G_{j}(z,v),
\end{equation*}
where again the tree function $T(z) = \sum_{n \ge 1} \tilde{T}_{n} z^{n}$ as defined in \eqref{eqn:Tree-function} appears.
Solving this linear differential equation yields the following solution satisfying the initial condition $G_{j}(0,v) = 0$:
\begin{equation}\label{eqn:Gjzvsol}
  G_{j}(z,v) = T(z) \cdot \exp\Big(\int_{0}^{z} \frac{(v-1) \tilde{T}_{j} t^{j-1}}{(1-T(t)) (1-T(t) - (v-1) \tilde{T}_{j} t^{j})} \, dt\Big).
\end{equation}
To study the moments of $C_{n,j}$ we apply in \eqref{eqn:Gjzvsol} the substitution $w:=v-1$ and introduce the function $\tilde{G}_{j}(z,w) := G_{j}(z,w+1)$, which yields the following expansion w.r.t.\ $w$:
\begin{equation*}
  \tilde{G}_{j}(z,w) = T(z) \cdot \exp\Big(\sum_{\ell \ge 1} w^{\ell} \alpha_{j,\ell}(z) \Big)
  = T(z) \cdot \Big(1+\sum_{r \ge 1} \frac{1}{r!} \Big(\sum_{\ell \ge 1} w^{\ell} \alpha_{j,\ell}(z)\Big)^{r}\Big),
\end{equation*}
where we use the abbreviation
\begin{equation*}
  \alpha_{j,\ell}(z) := \tilde{T}_{j}^{\ell} \cdot \int_{0}^{z} \frac{t^{j \ell-1}}{(1-T(t))^{\ell+1}} dt.
\end{equation*}
This leads to the following explicit formula for the generating function of the \mbox{$s$-th} integer moments of $C_{n,j}$, which will be the starting point of the asymptotic considerations:
\begin{align}
  G_{j}^{[s]}(z) & := [w^{s}] \tilde{G}_{j}(z,w) = \sum_{n \ge 1} \frac{\tilde{T_{n}} \E\big(C_{n,j}^{\underline{s}}\big)}{s!} z^{n} \notag \\
  & = T(z) \cdot \sum_{r=1}^{s} \frac{1}{r!} \sum_{\begin{smallmatrix}\ell_{1}+\cdots+\ell_{r}=s, \\ \ell_{q} \ge 1, 1 \le q \le r\end{smallmatrix}} \prod_{q=1}^{r} \alpha_{j,\ell_{q}}(z), \quad \text{for $s \ge 1$}.\label{eqn:Gjszexp}
\end{align}

The following bounds on the growth of the coefficients of the functions appearing, which all can be obtained in a straightforward way by applying standard techniques as singularity analysis or approximating sums by integrals (we omit here some of the details), will play a key r{\^o}le in the asymptotic evaluation of the moments.
First, it holds, for $\ell$ arbitrary but fixed and uniformly for $1 \le j \le n$ and $n \to \infty$:
\begin{equation*}
  [z^{n}] \alpha_{j,\ell}(z) = \frac{\tilde{T}_{j}^{\ell}}{n} [z^{n-j\ell}] \frac{1}{(1-T(z))^{\ell+1}} = \mathcal{O}\Big(\frac{\tilde{T}_{j}^{\ell} e^{n} n^{\frac{\ell-3}{2}}}{e^{j \ell}}\Big).
\end{equation*}
This implies
\begin{equation}\label{eqn:alphajellqprod}
  [z^{n}] \left.\prod_{q=1}^{r} \alpha_{j,\ell_{q}}(z)\right|_{\begin{smallmatrix}\ell_{1}+\cdots+\ell_{r}=s,\\\ell_{q} \ge 1, 1 \le q \le r\end{smallmatrix}}
  = \mathcal{O}\bigg(\frac{\tilde{T}_{j}^{s} e^{n}}{e^{j s}} \cdot \sum_{\begin{smallmatrix}k_{1}+\cdots+k_{r}=n,\\ k_{q}\ge1, 1 \le q \le r\end{smallmatrix}} 
  \prod_{q=1}^{r} k_{q}^{\frac{\ell_{q}-3}{2}}\bigg).
\end{equation}
The sum occurring in the latter bound \eqref{eqn:alphajellqprod} can itself be bounded as follows:
\begin{equation*}
  \sum_{\begin{smallmatrix}k_{1}+\dots+k_{r}=n,\\ k_{q} \ge 1, 1 \le q \le r\end{smallmatrix}} \prod_{q=1}^{r} k_{q}^{\frac{\ell_{q}-3}{2}} = 
  \begin{cases}
    \mathcal{O}\Big(n^{\frac{s}{2} - \frac{r}{2} -1} \big(\log n\big)^{Q}\Big), & \enspace \text{for $s > r$},\\
     \mathcal{O}\Big(\frac{(\log n)^{s-1}}{n}\Big), & \enspace \text{for $s=r$},
  \end{cases}
\end{equation*}
with $Q = |\{q: \ell_{q}=1\}|$, thus yielding
\begin{equation*}
  [z^{n}] \left.\prod_{q=1}^{r} \alpha_{j,\ell_{q}}(z)\right|_{\begin{smallmatrix}\ell_{1}+\cdots+\ell_{r}=s,\\\ell_{q} \ge 1, 1 \le q \le r\end{smallmatrix}}
  = \begin{cases}
    \mathcal{O}\Big(\frac{\tilde{T}_{j}^{s} e^{n} n^{\frac{s}{2} - \frac{r}{2} -1} \big(\log n\big)^{Q}}{e^{js}}\Big), & \enspace \text{for $s > r$}, \\
    \mathcal{O}\Big(\frac{\tilde{T}_{j}^{s} e^{n} (\log n)^{s-1}}{e^{js} n}\Big), & \enspace \text{for $s=r$}.
  \end{cases}
\end{equation*}

Therefore, when considering the coefficients of the expression 
\begin{equation*}
  \sum_{r=1}^{s} \frac{1}{r!} \sum_{\begin{smallmatrix}\ell_{1}+\cdots+\ell_{r}=s, \\ \ell_{q} \ge 1, 1 \le q \le r\end{smallmatrix}} \prod_{q=1}^{r} \alpha_{j,\ell_{q}}(z),
\end{equation*}
the summand $\alpha_{j,s}(z)$ gives the main contribution and implies the following bound, which holds uniformly for $1 \le j \le n$ (with $s$ arbitrary but fixed and $n \to \infty$):
\begin{equation}\label{eqn:sumprodalphabound}
  [z^{n}] \sum_{r=1}^{s} \frac{1}{r!} \sum_{\begin{smallmatrix}\ell_{1}+\cdots+\ell_{r}=s, \\ \ell_{q} \ge 1, 1 \le q \le r\end{smallmatrix}} \prod_{q=1}^{r} \alpha_{j,\ell_{q}}(z)
  = \mathcal{O}\bigg(\frac{\tilde{T}_{j}^{s} e^{n} n^{\frac{s-3}{2}}}{e^{js}}\bigg).
\end{equation}

Now we are in a position to derive the stated bound on the $s$-th integer moments of $C_{n,j}$. First, by using \eqref{eqn:Gjszexp} and \eqref{eqn:sumprodalphabound} we get for $s \ge 1$:
\begin{align*}
  \E\big(C_{n,j}^{\underline{s}}\big) & = \frac{s!}{\tilde{T}_{n}} [z^{n}] G_{j}^{[s]}(z) = \frac{s!}{\tilde{T}_{n}} [z^{n}] T(z) \cdot \sum_{r=1}^{s} \frac{1}{r!} \sum_{\begin{smallmatrix}\ell_{1}+\cdots+\ell_{r}=s, \\ \ell_{q} \ge 1, 1 \le q \le r\end{smallmatrix}} \prod_{q=1}^{r} \alpha_{j,\ell_{q}}(z)\\
  & = \mathcal{O}\Big(\frac{s!}{\tilde{T}_{n}} \sum_{k=1}^{n-1} \tilde{T}_{k} \frac{\tilde{T}_{j}^{s} e^{n-k} (n-k)^{\frac{s-3}{2}}}{e^{js}}\Big)
  = \mathcal{O}\Big(\frac{\tilde{T}_{j}^{s} n^{\frac{3}{2}}}{e^{js}} \sum_{k=1}^{n-1} \frac{(n-k)^{\frac{s-3}{2}}}{k^{\frac{3}{2}}}\Big).
\end{align*}
Splitting the remaining sum easily gives
\begin{align}
  & \sum_{k=1}^{n-1} \frac{(n-k)^{\frac{s-3}{2}}}{k^{\frac{3}{2}}} = \sum_{k=1}^{\lfloor\frac{n}{2}\rfloor} \frac{(n-k)^{\frac{s-3}{2}}}{k^{\frac{3}{2}}} 
  + \sum_{k=\lfloor\frac{n}{2}\rfloor+1}^{n-1} \frac{(n-k)^{\frac{s-3}{2}}}{k^{\frac{3}{2}}} \notag\\
  & \quad = \mathcal{O}\big(n^{\frac{s-3}{2}}\big) + \begin{cases} \mathcal{O}\big(n^{\frac{s}{2}-2}\big), & s \ge 2, \\ \mathcal{O}\big(\frac{\log n}{n^{\frac{3}{2}}}\big), & s=1 \end{cases}
  = \mathcal{O}\big(n^{\frac{s-3}{2}}\big), \label{eqn:sumkn-k}
\end{align}
thus showing the bound (which holds uniformly for $1 \le j \le n$)
\begin{equation}\label{eqn:Cnj_bound}
  \E\big(C_{n,j}^{\underline{s}}\big) = \mathcal{O}\Big(\Big(\frac{\sqrt{n} \, \tilde{T}_{j}}{e^{j}}\Big)^{s}\Big).
\end{equation}

To give the refined asymptotic expansion of $\E\big(C_{n,j}^{\underline{s}}\big)$, yielding factorial moments of mixed Poisson type, one has to spot and evaluate the main contribution of the coefficients of \eqref{eqn:Gjszexp} in more detail and to bound the remaining contributions.
In order to do this we split 
\begin{equation*}
  \sum_{r=1}^{s} \frac{1}{r!} \sum_{\begin{smallmatrix}\ell_{1}+\cdots+\ell_{r}=s, \\ \ell_{q} \ge 1, 1 \le q \le r\end{smallmatrix}} \prod_{q=1}^{r} \alpha_{j,\ell_{q}}(z)
  = \alpha_{j,s}(z) + \underbrace{\sum_{r=2}^{s} \frac{1}{r!} \sum_{\begin{smallmatrix}\ell_{1}+\cdots+\ell_{r}=s, \\ \ell_{q} \ge 1, 1 \le q \le r\end{smallmatrix}} \prod_{q=1}^{r} \alpha_{j,\ell_{q}}(z)}_{=: Q_{j,s}(z)},
\end{equation*}
such that
\begin{equation}\label{eqn:momCnjs}
  \E\big(C_{n,j}^{\underline{s}}\big) = \frac{s!}{\tilde{T}_{n}} [z^{n}] T(z) \cdot \big(\alpha_{j,s}(z) + Q_{j,s}(z)\big).
\end{equation}
Completely analogous to the previous computations one shows for $s \ge 2$ (of course, $Q_{j,1}(z)=0$):
\begin{equation*}
  [z^{n}] Q_{j,s}(z) = \mathcal{O}\Big(\frac{\tilde{T}_{j}^{s} e^{n} n^{\frac{s}{2}-2} \log n}{e^{js}}\Big),
\end{equation*}
and furthermore (uniformly for $1 \le j \le n$) the following bound for the contribution of the remainder:
\begin{equation}\label{eqn:CnjsRemBound}
  \frac{s!}{\tilde{T}_{n}} [z^{n}] T(z) \cdot Q_{j,s}(z) = \mathcal{O}\Big(\Big(\frac{\sqrt{n} \, \tilde{T}_{j}}{e^{j}}\Big)^{s} \cdot \frac{\log n}{\sqrt{n}}\Big).
\end{equation}

Now we consider the term in \eqref{eqn:Cnj_bound} yielding the main contribution, where we assume from now on that $j=o(n)$. We get
\begin{equation}\label{eqn:Cnjmain}
  [z^{n}] T(z) \alpha_{j,s}(z) = \sum_{k=1}^{n-sj} \tilde{T}_{k} \tilde{T}_{j}^{s} \frac{1}{n-k} [z^{n-k-sj}] \frac{1}{(1-T(z))^{s+1}}.
\end{equation}
We split the summation interval of \eqref{eqn:Cnjmain} at $k=\lfloor\frac{n}{2}\rfloor$ and consider the contributions separately.
Additionally, we only require the already computed asymptotic bounds \eqref{eqn:znTz} and \eqref{eqn:sumkn-k}.
The first part yields
\begin{align*}
  & \sum_{k=1}^{\lfloor\frac{n}{2}\rfloor} \tilde{T}_{k} \tilde{T}_{j}^{s} \frac{1}{n-k} [z^{n-k-sj}] \frac{1}{(1-T(z))^{s+1}} \\
  & \quad = \tilde{T}_{j}^{s} \sum_{k=1}^{\lfloor\frac{n}{2}\rfloor} \tilde{T}_{k} \frac{1}{n-k} \frac{e^{n-k-sj} (n-k-sj)^{\frac{s-1}{2}}}{2^{\frac{s+1}{2}} \big(\frac{s-1}{2}\big)!} \cdot \Big(1+\mathcal{O}\Big(\frac{1}{\sqrt{n-k-sj}}\Big)\Big)\\
  & \quad = \frac{\tilde{T}_{j}^{s} e^{n} n^{\frac{s-3}{2}}}{e^{js} \, 2^{\frac{s+1}{2}} \big(\frac{s-1}{2}\big)!} \sum_{k=1}^{\lfloor\frac{n}{2}\rfloor} \frac{\tilde{T}_{k}}{e^{k}} \cdot \Big(1 + \mathcal{O}\big(\frac{k}{n}\big) + \mathcal{O}\big(\frac{j}{n}\big)
 + \mathcal{O}\big(\frac{1}{\sqrt{n}}\big)\Big).
\end{align*}
The main contribution comes from
\begin{equation*}
  \sum_{k=1}^{\lfloor\frac{n}{2}\rfloor} \frac{\tilde{T}_{k}}{e^{k}} = \sum_{k=1}^{\infty} \frac{\tilde{T}_{k}}{e^{k}} - \sum_{k=\lfloor\frac{n}{2}\rfloor+1}^{\infty} \frac{\tilde{T}_{k}}{e^{k}} =
  T(e^{-1}) + \mathcal{O}\Big(\sum_{k=\lfloor\frac{n}{2}\rfloor+1}^{\infty} \frac{1}{k^{\frac{3}{2}}}\Big)
  = 1 + \mathcal{O}\big(\frac{1}{\sqrt{n}}\big),
\end{equation*}
whereas
\begin{equation*}
  \mathcal{O}\Big(\sum_{k=1}^{\lfloor\frac{n}{2}\rfloor} \frac{\tilde{T}_{k}}{e^{k}} \frac{k}{n}\Big)
  = \mathcal{O}\Big(\frac{1}{n} \sum_{k=1}^{\lfloor\frac{n}{2}\rfloor}\frac{1}{\sqrt{k}}\Big)
  = \mathcal{O}\big(\frac{1}{\sqrt{n}}\big).
\end{equation*}
Thus
\begin{equation}\label{eqn:mainfirstpart}
  \sum_{k=1}^{\lfloor\frac{n}{2}\rfloor} \tilde{T}_{k} \tilde{T}_{j}^{s} \frac{1}{n-k} 
  [z^{n-k-sj}] \frac{1}{(1-T(z))^{s+1}}
  = \frac{\tilde{T}_{j}^{s} e^{n} n^{\frac{s-3}{2}}}{e^{js} \, 2^{\frac{s+1}{2}} \big(\frac{s-1}{2}\big)!} \cdot \Big(1+\mathcal{O}\big(\frac{j}{n}\big)+\mathcal{O}\big(\frac{1}{\sqrt{n}}\big)\Big).
\end{equation}
The second part yields
\begin{align}
  & \sum_{k=\lfloor\frac{n}{2}\rfloor+1}^{n-sj} \tilde{T}_{k} \tilde{T}_{j}^{s} \frac{1}{n-k} [z^{n-k-sj}] \frac{1}{(1-T(z))^{s+1}}\notag\\
  & \quad = \mathcal{O}\Big(\tilde{T}_{j}^{s} \sum_{k=\lfloor\frac{n}{2}\rfloor+1}^{n-sj} \frac{e^{k} e^{n-k-sj} (n-k-sj)^{\frac{s-1}{2}}}{k^{\frac{3}{2}} (n-k)}\Big)
  = \mathcal{O}\Big(\frac{\tilde{T}_{j}^{s} e^{n}}{e^{js} \, n^{\frac{3}{2}}} \sum_{k=sj}^{\lceil\frac{n}{2}\rceil-1} \frac{(k-sj)^{\frac{s-1}{2}}}{k}\Big)\notag\\
  & \quad = \mathcal{O}\Big(\frac{\tilde{T}_{j}^{s} e^{n}}{e^{js} \, n^{\frac{3}{2}}} \sum_{k=sj}^{\lceil\frac{n}{2}\rceil-1} k^{\frac{s-3}{2}} \Big) 
  = \begin{cases} \mathcal{O}\Big(\frac{\tilde{T}_{j}^{s} e^{n} n^{\frac{s}{2}-2}}{e^{js}}\Big), & \enspace \text{for $s \ge 2$},\\ \mathcal{O}\Big(\frac{\tilde{T}_{j} e^{n} \log n}{e^{j} n^{\frac{3}{2}}}\Big), & 
  \enspace \text{for $s=1$}.
  \end{cases}\label{eqn:mainsecondpart}
\end{align}
Starting with \eqref{eqn:Cnjmain} and combining the contributions \eqref{eqn:mainfirstpart} and \eqref{eqn:mainsecondpart} finally yields
\begin{equation}\label{eqn:mainasympt}
  [z^{n}] T(z) \alpha_{j,s}(z) = \Big(\frac{\sqrt{n} \, \tilde{T}_{j}}{e^{j}}\Big)^{s} \frac{e^{n}}{2^{\frac{s+1}{2}} \big(\frac{s-1}{2}\big)! \, n^{\frac{3}{2}}} \cdot
  \Big(1+\mathcal{O}\big(\frac{j}{n}\big) + \mathcal{O}\Big(\frac{(\log n)^{\delta_{s,1}}}{\sqrt{n}}\Big)\Big),
\end{equation}
with $\delta_{s,i}$ the Kronecker-delta function. Together with Stirling's formula for the Gamma function we thus obtain from \eqref{eqn:momCnjs}, \eqref{eqn:CnjsRemBound} and \eqref{eqn:mainasympt}
\begin{equation*}
  \E\big(C_{n,j}^{\underline{s}}\big) = \frac{s!}{\tilde{T}_{n}} [z^{n}] G_{j}^{[s]}(z)
  = \frac{\sqrt{2\pi} s!}{2^{\frac{s+1}{2}} \big(\frac{s-1}{2}\big)!} \cdot \Big(\frac{\sqrt{n} \, \tilde{T}_{j}}{e^{j}}\Big)^{s} \cdot \Big(1+\mathcal{O}\big(\frac{j}{n}\big)+\mathcal{O}\Big(\frac{\log n}{\sqrt{n}}\Big)\Big).
\end{equation*}
Finally, applying the duplication formula of the factorials \eqref{eqn:duplication}, we get the stated result for the asymptotic expansion of the factorial moments:
\begin{equation}\label{eqn:Cnj_moments}
    \E\big(C_{n,j}^{\underline{s}}\big) = \Big(\frac{\sqrt{n} \tilde{T}_{j}}{e^{j}}\Big)^{s} \, 2^{\frac{s}{2}} \, \Gamma\big(\frac{s}{2}+1\big) \cdot \big(1+o(1)\big), \quad \text{for $j=o(n)$}.
\end{equation}
The mixed Poisson limit law with Rayleigh mixing distribution as stated in Theorem~\ref{the:Edgecut} follows then from \eqref{eqn:Cnj_bound} and \eqref{eqn:Cnj_moments} by applying Lemma~\ref{MOMSEQMainLemma}.
\end{proof}

\subsection{Parking functions and growth of the initial cluster\label{Ssec:Parking}}
Parking functions are objects introduced by Konheim and Weiss \cite{KonWei1966}, which are of interest in combinatorics (e.g., due to connections to various other combinatorial structures as forests, acyclic functions or hyperplane arrangements), see, e.g., \cite{Stanley}, and computer science (e.g., due to close relations to hashing variants), see, e.g., \cite{Knuth1998}.
A vivid description of parking functions is as follows: consider a one-way street with $n$ parking spaces numbered from $1$ to $n$ and a sequence of $n$ drivers (we will here exclusively deal with the case that the number of parking spaces is equal to the number of drivers) with preferred parking spaces $s_{1}, s_{2}, \dots, s_{n}$. The drivers arrive sequentially and driver $k$ tries to park at its preferred parking space $s_{k}$. If it is free he parks, otherwise he moves further in the allowed direction (thus examining parking spaces $s_{k}+1, s_{k}+2, \dots$) until he finds a free parking space, where he parks; if there is no such parking space he leaves the street without parking. A parking function is then a sequence $(s_{1}, \dots, s_{n}) \in \{1, \dots, n\}^{n}$ such that all drivers are able to park.
It has been shown already in \cite{KonWei1966} that there are exactly $P_{n} = (n+1)^{n-1}$ parking functions of size $n$.
Figure~\ref{fig:ParkingFunction} gives an example of a parking function.
\begin{figure}
\begin{center}
  \includegraphics[height=2.5cm]{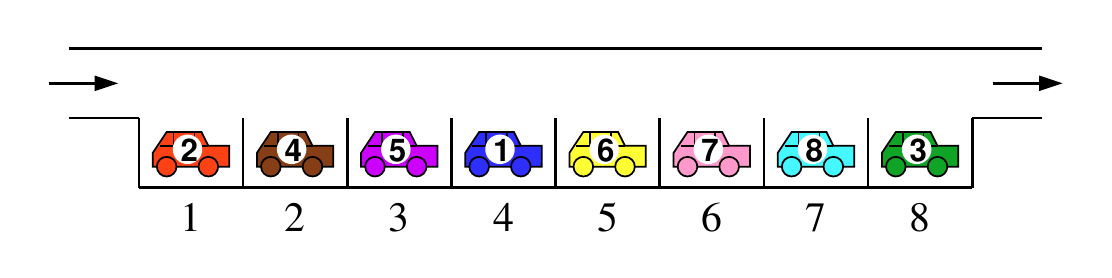}
\end{center}
\caption{The parking function $(4,1,8,1,3,4,3,1)$ with the respective parking positions when carrying out the parking procedure.\label{fig:ParkingFunction}}
\end{figure}

We note that there are many alternative ways of defining parking functions $(s_{1}, \dots, s_{n}) \in \{1, \dots, n\}^{n}$, e.g., via the characterization $|\{j: s_{j} \le k\}| \ge k$, for all $1 \le k \le n$; however, in what follows the description given above is more intuitive and seems to be advantageous. Namely, we may start with a parking function $(s_{1}, \dots, s_{n})$ and consider the filling of the parking spaces during the \emph{parking procedure}, where at the beginning (step $0$) we have an empty street, and where in step $k$ the $k$-th driver arrives and successfully parks, until (after step $n$) eventually all parking spaces are occupied. When carrying out the parking procedure, at each time step we may define the initial cluster as the maximum sequence of consecutive occupied parking spaces starting with parking space $1$; if parking space $1$ is empty we say that the initial cluster is empty. The size of the initial cluster is then simply the number of consecutive occupied parking spaces containing parking space $1$. In this section we are interested in the growth of the initial cluster during the parking procedure starting with a random parking function: let the r.v.\ $X_{n}$ denote the number of increments of the initial cluster and the refinement $X_{n,j}$ measure the number of increments of amount $j$ of the initial cluster during the parking procedure of a random parking function of size $n$; of course $X_{n} = \sum_{j=1}^{n-1} X_{n,j}$. It turns out that the r.v.\ $X_{n}$ and $X_{n,j}$ are closely related to $C_{n}$ and $C_{n,j}$, respectively, studied in Section~\ref{Ssec:Cutting} during the analysis of the edge-cutting procedure of Cayley-trees.
Figure~\ref{fig:ParkingProcedure} illustrates the parking procedure and the growth of the initial cluster.
\begin{figure}
\begin{flushleft}
  \includegraphics[height=1.3cm]{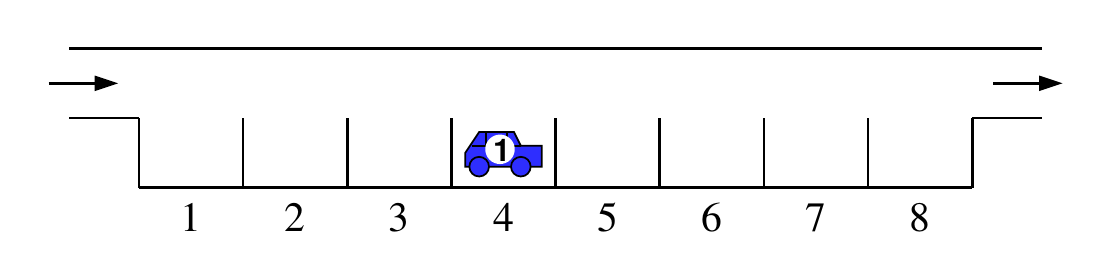} \enspace \text{\raisebox{0.6cm}{\Large{$\Rightarrow$}}} \enspace  
  \includegraphics[height=1.3cm]{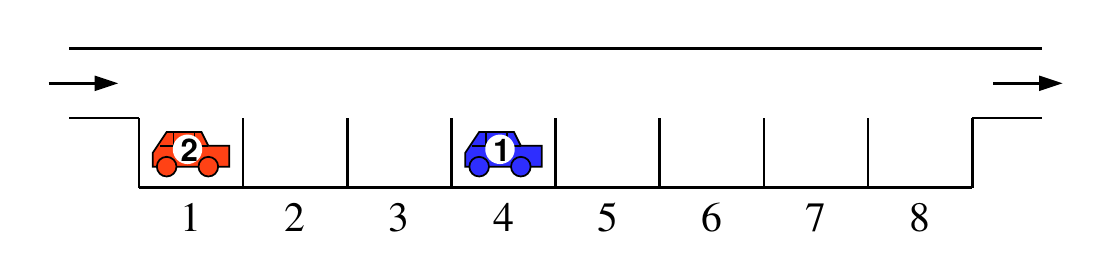} \enspace \text{\raisebox{0.6cm}{\Large{$\Rightarrow$}}} \enspace
  \includegraphics[height=1.3cm]{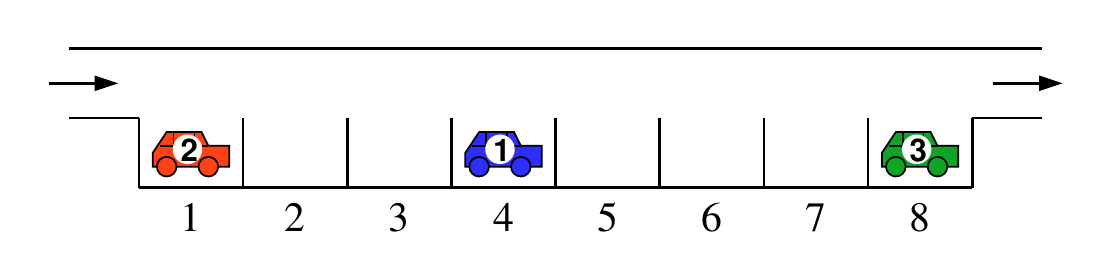} \enspace \text{\raisebox{0.6cm}{\Large{$\Rightarrow$}}} \enspace
  \includegraphics[height=1.3cm]{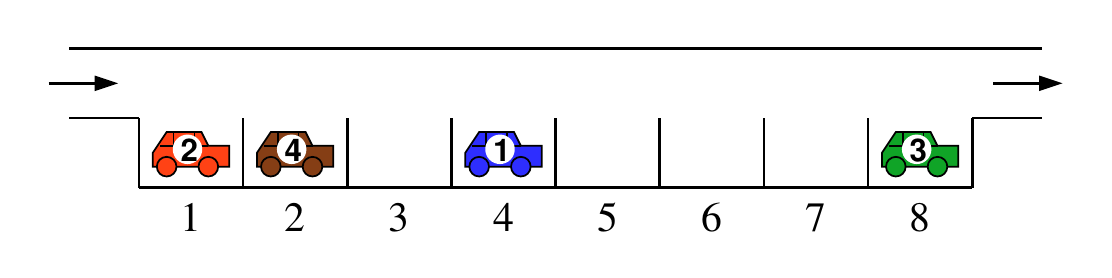} \enspace \text{\raisebox{0.6cm}{\Large{$\Rightarrow$}}} \enspace
  \includegraphics[height=1.3cm]{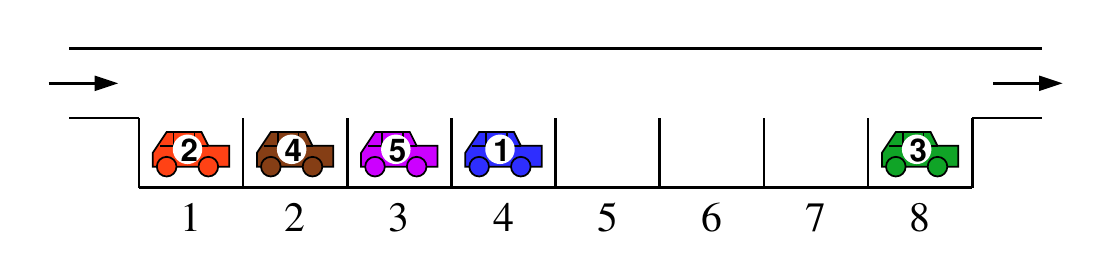} \enspace \text{\raisebox{0.6cm}{\Large{$\Rightarrow$}}} \enspace
  \includegraphics[height=1.3cm]{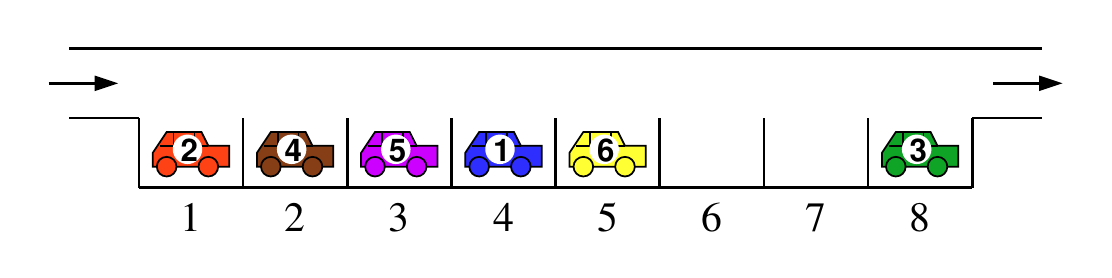} \enspace \text{\raisebox{0.6cm}{\Large{$\Rightarrow$}}} \enspace
  \includegraphics[height=1.3cm]{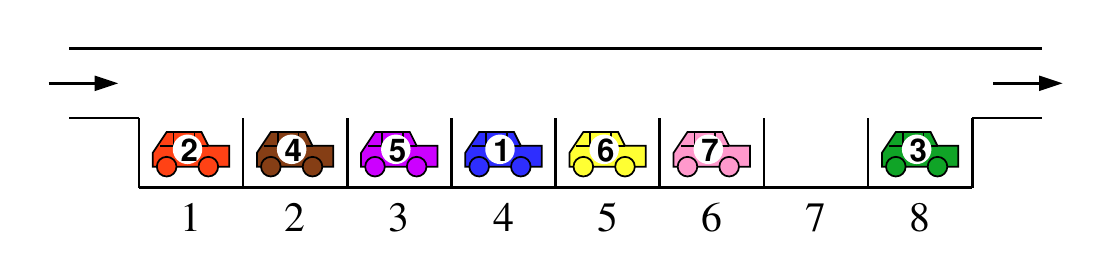} \enspace \text{\raisebox{0.6cm}{\Large{$\Rightarrow$}}} \enspace
  \includegraphics[height=1.3cm]{stell14pic8a.pdf}
\end{flushleft}
\caption{The parking procedure for the parking function $(4,1,8,1,3,4,3,1)$. There are $6$ increments of the initial cluster, $4$ increments of amount one, and $2$ increments of amount two.\label{fig:ParkingProcedure}}
\end{figure}

It is well-known \cite{Stanley} that the number $P_{n}$ of parking functions of size $n$ coincides with the number of rooted labelled forests, i.e., forests of rooted labelled unordered trees, of size $n$. There are various known bijections between these objects; however, it seems that the following bijective relation between the growth of the initial cluster during the parking procedure and records in forests of rooted labelled trees has not been observed earlier. We mention that a similar bijection has been used by Chassaing and Louchard in \cite{ChaLou2002}, where they related the growth of clusters in parking functions with the additive coalescent model for particle coagulation.
Analogous to the definition in Section~\ref{Ssec:Records} a max-record in a forest $F$ of labelled trees is a node $x \in F$, which has the largest label amongst all nodes on the unique path from the root-node of the tree component containing $x$ to $x$.
\begin{prop}\label{Prop:bijparkingseq}
There is a bijection, which maps parking functions of size $n$ to forests of rooted labelled unordered trees of size $n$, such that the number of increments of the initial cluster during the parking procedure of a parking function corresponds to the number of max-records in the forest. Moreover, the number of increments of amount $j$ of the initial cluster during the parking procedure corresponds to the number of max-record subtrees of size $j$ in the forests.
\end{prop}
\begin{proof}
Given a parking function $(s_{1}, \dots, s_{n}) \in \{1, \dots, n\}^{n}$ of size $n$ we describe the mapping, i.e., the construction of the corresponding rooted labelled forest $F$ of size $n$, in an iterative way, which reflects the parking procedure of the $n$ drivers. In order to describe the construction we assume that after step $k$ the first $k$ drivers are parked; then the ``parking street'' consists of a set of clusters of parking spaces (a cluster is here a maximal sequence of consecutive occupied parking spaces) separated by empty parking spaces. In the construction the $k$-th driver of the parking function will correspond to the node labelled $k$ in the forest and after step $k$ we will obtain a rooted labelled forest $F^{(k)}$ of size $k$ (with nodes labelled by $\{1, \dots, k\}$). Moreover, the forest $F^{(k)}$ has the property, that each cluster of parking spaces occurring in the parking procedure after step $k$ corresponds to a subset of rooted labelled trees in $F^{(k)}$.
It follows the description of the construction of the forest $F := F^{(n)}$:
\begin{description}
\item[Step $0$] We start with the empty forest $F^{(0)} = \emptyset$.
\item[Step $k$] According to the parking of driver $k$ we distinguish between two cases.
\begin{itemize}
\item Driver $k$ parks at his preferred parking space $s := s_{k}$: let us consider the parking procedure after Step~$(k-1)$ and the cluster of parking spaces starting with parking space $s+1$ (i.e., the cluster of parking spaces right to the parking space of driver $k$); if parking space $s+1$ is empty the cluster is $\emptyset$. According to the construction this cluster corresponds to a subset $G$ of trees of the forest $F^{(k-1)}$. Let $T$ be the tree rooted at the new node labelled $k$ with $G$ its subtrees. Then the forest $F^{(k)}$ is defined by $F^{(k)} := (F^{(k-1)} \setminus G) \cup T$.
\item Driver $k$ cannot park at his preferred parking space $s_{k}$, since it is occupied by the $\ell$-th driver ($\ell < k$), but parks at the first empty space $s > s_{k}$: let us consider the parking procedure after Step~$(k-1)$ and the cluster of parking spaces starting with parking space $s+1$ (i.e., the cluster of parking spaces right to the parking space of driver $k$); if parking space $s+1$ is empty the cluster is $\emptyset$. According to the construction this cluster corresponds to a subset $G$ of trees of the forest $F^{(k-1)}$. Furthermore, let $T'$ be the tree of the forest $F^{(k-1)}$ containing label $\ell$ (by construction $T' \not\in G$). Then, construct the rooted tree $T$ by letting $G$ be the subtrees of the (new) node labelled $k$ and attaching node $k$ to node $\ell \in T'$. Then the forest $F^{(k)}$ is defined by $F^{(k)} := (F^{(k-1)} \setminus(G \cup T')) \cup T$.
\end{itemize}
\end{description}
Figure~\ref{fig:ParkingBijectionScheme} illustrates both cases of the bijection.
\begin{figure}
\begin{center}
  \includegraphics[width=5.0cm]{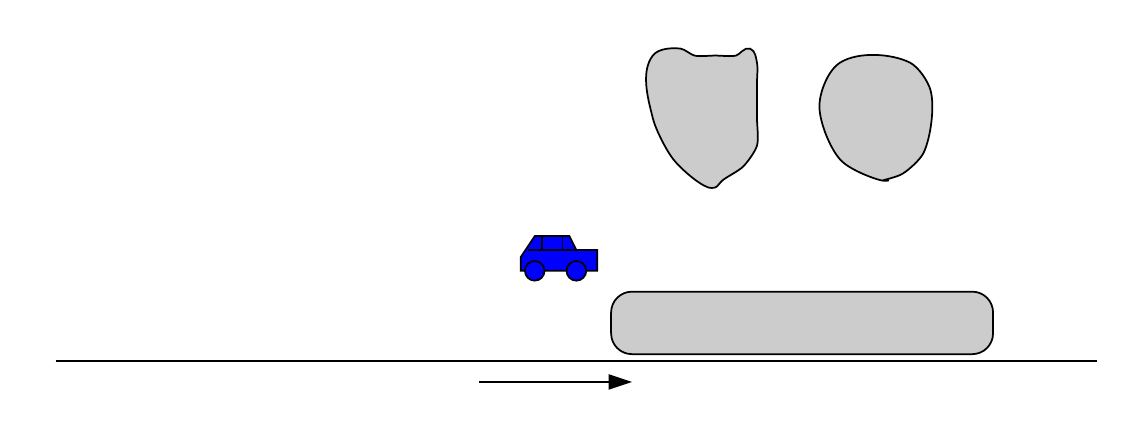} \enspace \text{\raisebox{10mm}{\Large{$\Rightarrow$}}} \enspace \includegraphics[width=5.0cm]{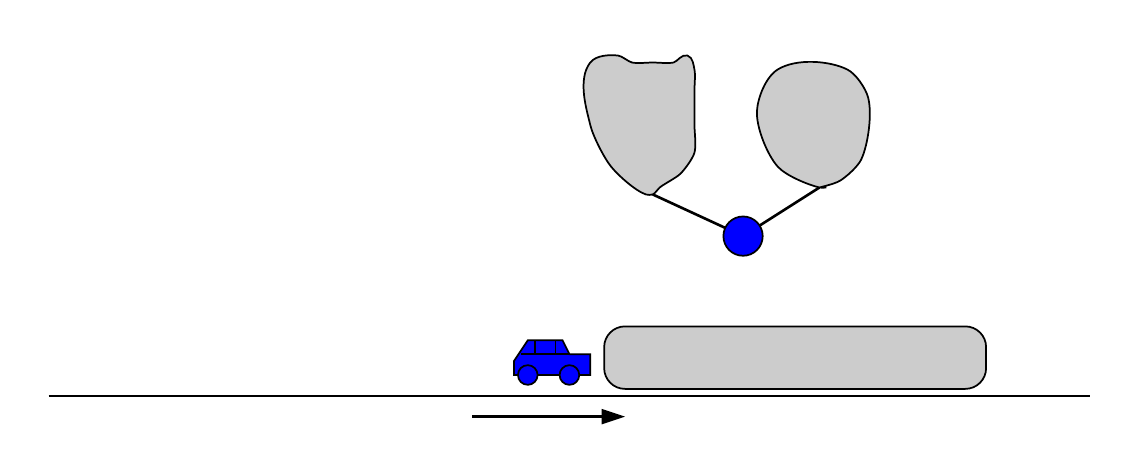}\\[5mm]
  \includegraphics[width=5.0cm]{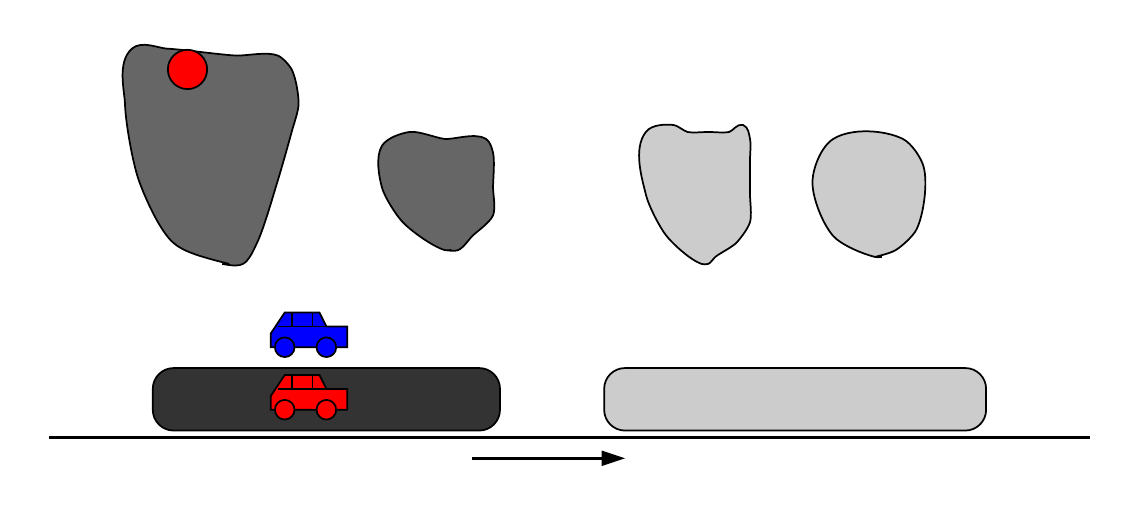} \enspace \text{\raisebox{10mm}{\Large{$\Rightarrow$}}} \enspace \includegraphics[width=5.0cm]{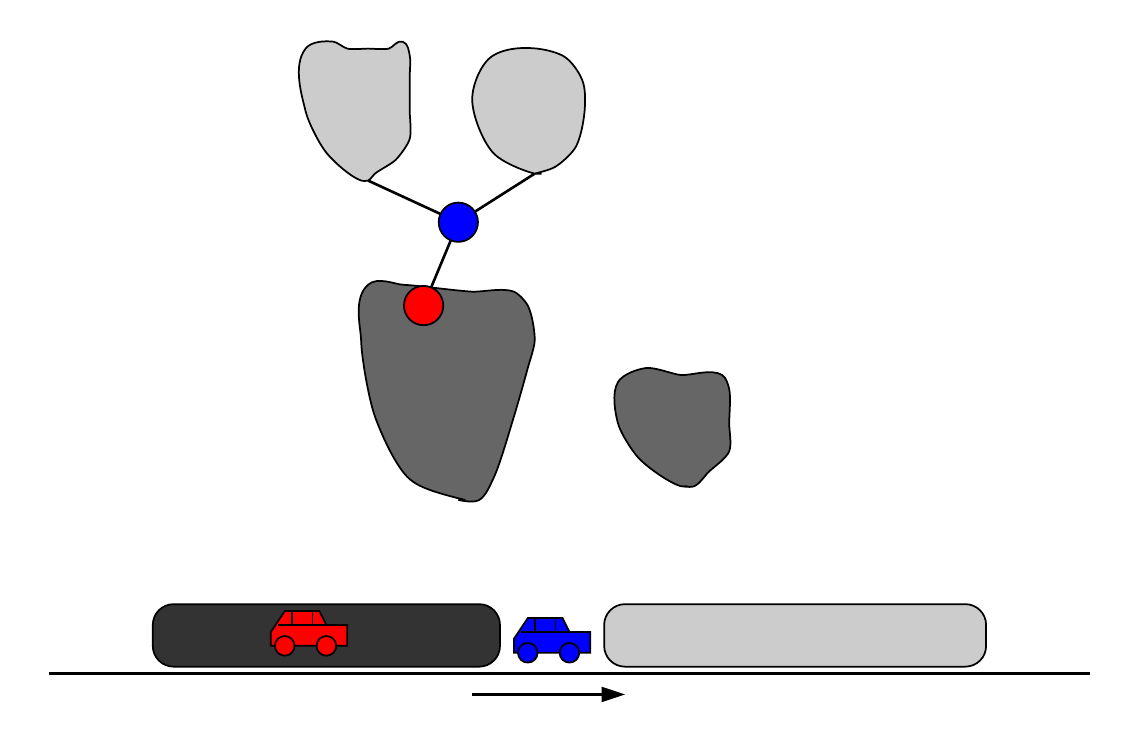}
\end{center}

\vspace*{-2mm}

\caption{Constructing a forest of labelled rooted trees from a parking function during the parking procedure illustrating the cases, where the preferred parking space of a driver is free (first picture) or occupied (second picture).\label{fig:ParkingBijectionScheme}}
\end{figure}

It follows from the construction that node $k$ is a max-record in the forest $F$ iff the first $k-1$ drivers occupy all parking spaces left to the parking space $s$, where the $k$-th driver has parked, which is equivalent to the event that Step~$k$ was an increment of the initial cluster. Moreover, in this case the subtree rooted at node $k$ in $F^{(k)}$ corresponds to a record subtree in the forest $F$, whose size corresponds then to the amount of increment of the initial cluster.
It is not difficult to see that this construction is indeed a bijection from the set of parking functions of size $n$ to the set of rooted labelled forests of size $n$, but we omit here to state the inverse mapping (which could be formulated also in an iterative way).
\end{proof}
In Figure~\ref{fig:ParkingBijection} we give a parking function and the corresponding forest of labelled rooted trees under this bijection.
\begin{figure}
\begin{center}
  \includegraphics[height=2.0cm]{stell14pic8a.pdf}\\
  \hspace*{0.1cm} \Large{$\Updownarrow$} \hspace*{3.5cm} \Large{$\Updownarrow$} \hspace*{1.8cm}\\
  \includegraphics[height=2.8cm]{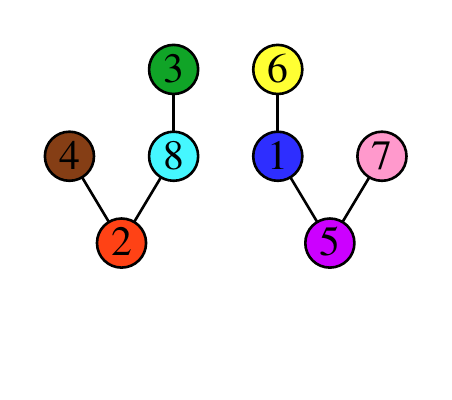} \enspace \raisebox{1.1cm}{\Large{$\Leftrightarrow$}} \includegraphics[height=2.8cm]{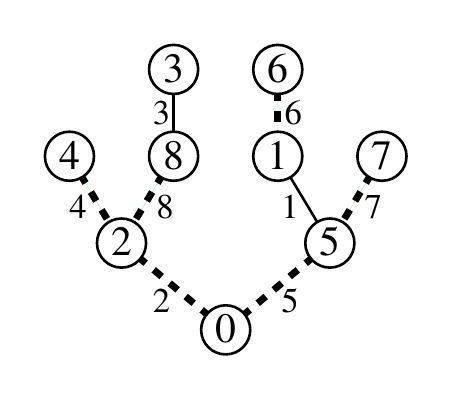}
\end{center}

\vspace*{-5mm}

\caption{The parking function $(4,1,8,1,3,4,3,1)$ and the forest of labelled rooted trees obtained by the mapping described in the proof of Proposition~\ref{Prop:bijparkingseq} as well as the edge-labelled rooted tree described in the proof of Theorem~\ref{the:ParkingSeq}.
The increments of the initial cluster of the parking function correspond to the max-records in the respective forest as well as to the cuts in the respective edge-labelled rooted tree (visualized as dotted lines).\label{fig:ParkingBijection}}
\end{figure}

Proposition~\ref{Prop:bijparkingseq} yields thus a coupling between records in forests of rooted labelled unordered trees and increments in parking functions. This coupling can be extended easily to one between increments in parking functions and edge-cuts to isolate the root-node in Cayley-trees.
\begin{theorem}\label{the:ParkingSeq}
The random variable $X_{n}$, counting the number of increments of the initial cluster in a random parking function of size $n$, is equally distributed as the random variable $C_{n}$, counting the number of edge-cuts to isolate the root-node in a random Cayley-tree of size $n$, i.e., $X_{n} \law C_{n}$. Moreover, the number of increments of amount $j$ in a random parking function of size $n$ is equally distributed as the number of subtrees of size $j$ cut-off during the edge-cutting procedure when starting with a random size-$n$ Cayley-tree, i.e., $X_{n,j} \law C_{n,j}$, $1 \le j \le n-1$.

After suitable normalization $X_{n}$ is asymptotically, for $n \to \infty$, Rayleigh distributed with parameter $1$:
\begin{equation*}
  \frac{X_{n}}{\sqrt{n}} \claw \text{Rayleigh}(1).
\end{equation*}
$X_{n,j}$ has, for $n\to\infty$ and arbitrary $1 \le j=j(n) \le n-1$, 
asymptotically factorial moments of mixed Poisson type with a Rayleigh mixing distribution $X$ and scale parameter $\lambda_{n,j} = \frac{\sqrt{n} j^{j-1}}{j! e^{j}}$:
\begin{equation*}
\E(\fallfak{X_{n,j}}s) = \lambda_{n,j}^{s} \, 2^{\frac{s}{2}} \, \Gamma\big(\frac{s}{2}+1\big) \cdot \big(1+o(1)\big), \enspace \text{for $j=o(n)$}, \quad
\E(\fallfak{X_{n,j}}s) = \mathcal{O}\big(\lambda_{n,j}^{s}\big).
\end{equation*}
\begin{itemize}
\item[(i)] for $\lambda_{n,j}\to\infty$, the random variable $\frac{X_{n,j}}{\lambda_{n,j}}$ converges in distribution, with convergence of all moments, to $X$. 

\item[(ii)] for $\lambda_{n,j}\to\rho\in[0,\infty)$, the random variable $X_{n,j}$ converges in distribution, with convergence of all moments, to a mixed Poisson distributed random variable $Y\law \MPo(\rho X)$. 
\end{itemize}
\end{theorem}
\begin{proof}
Starting with a labelled forest $F$ of size $n$ we get a rooted labelled tree $T$ of size $n+1$ by attaching the trees in $F$ as subtrees of the root-node labelled $0$. Next we label the edges of $T$ by labels $\{1, \dots, n\}$, where each edge $e=(x,y)$, with $x$ the parent of $y$, gets the label of the child $y$. When applying the edge-cutting procedure to the edge-labelled tree $T$ in a way that at each step the edge with largest label is chosen and cut-off, each max-record of the original forest $F$ corresponds to a cut in $T$ and furthermore a record-subtree of size $j$ in $F$ corresponds in $T$ to a cut-off of a branch of size $j$. Together with Proposition~\ref{Prop:bijparkingseq} this yields
\begin{equation*}
  X_{n} \law C_{n} \quad \text{and} \quad X_{n,j} \law C_{n,j}, \enspace 1 \le j \le n-1.
\end{equation*}
The limiting distribution results for $X_{n}$ and $X_{n,j}$ follow thus from the corresponding results for $C_{n}$ and $C_{n,j}$ given in \cite{Janson2006b,Pan2006} and Theorem~\ref{the:Edgecut}, respectively.
\end{proof}

\subsection{Zero contacts in bridges\label{Ssec:Dyck}}
We consider directed lattice paths from left to right starting at $(0,0)$ and ending at $(2n,0)$. 
At each horizontal unit step we can either go one unit up (step $(1,1)$) or down (step $(1,-1)$). 
Such lattice paths are called bridges of length $2n$ starting and ending at height zero, 
and the steps are stemming from so-called Dyck paths. 
Of course, such lattice paths are in bijection with lattice paths on a square grid, starting at $(0,0)$ and ending at $(n,n)$, with allowed steps $(1,0)$ (right) and $(0,1)$ (up).
Apparently, there are $B_{n} = \binom{2n}{n}$ such lattice paths and thus bridges of length $2n$.

Flajolet and Sedgewick~\cite[Example~IX.40, page~707]{FlaSed2009} considered the random variable $X_{n}$, counting the number of visits of the $x$-axis in a random bridge of size $2n$,
i.e., the number of $k$, $1 \le k \le n$, with $(2k,0)$ contained in the bridge, by selecting one of the $B_{n}$ bridges of length $2n$ uniformly at random.
By using a combinatorial decomposition of bridges (the so-called arch decomposition), they have shown that $X_{n}$ follows asymptotically a Rayleigh distribution, i.e., $\frac{X_{n}}{\sqrt{n}} \claw \text{Rayleigh}(\sqrt{2})$.

We consider here a refinement of the r.v.\ $X_{n}$ by introducing the r.v.\ $X_{n,j}$ counting the number of $j$-visits of the $x$-axis, where a $j$-visit is simply a visit after an excursion of length $2j$, i.e., a return to the $x$-axis after exactly $2j$, $j \ge 1$, steps.
Of course, $X_{n} = \sum_{j=1}^{n} X_{n,j}$. 
Figure~\ref{fig:ZeroContacts} illustrates these quantities.
\begin{figure}
\begin{center}
  \includegraphics[height=2.5cm]{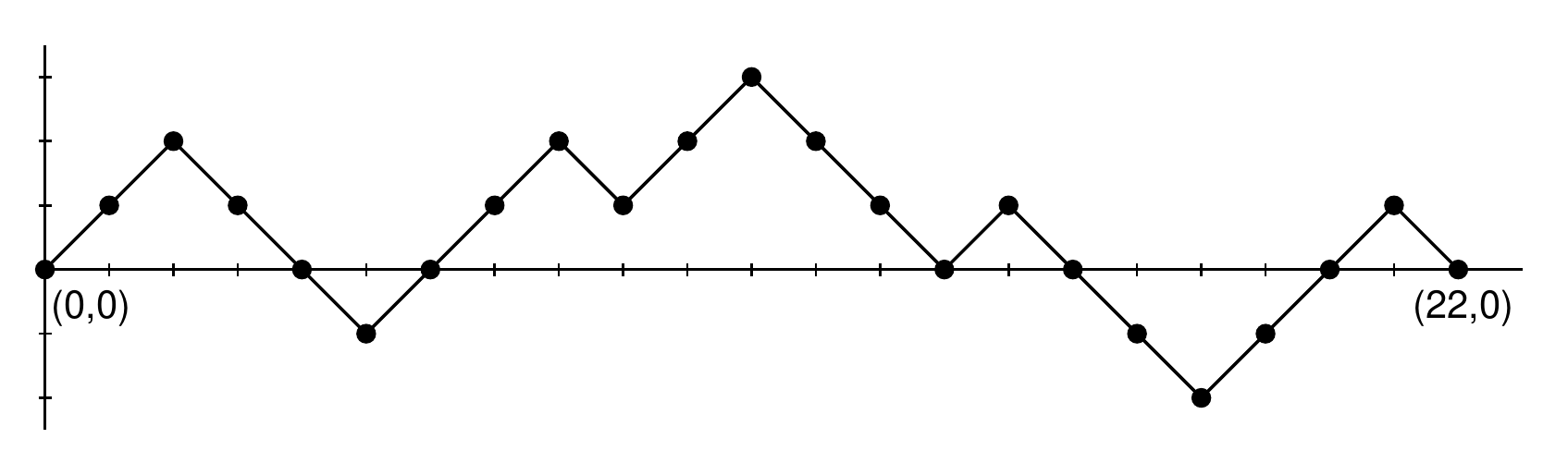}
\end{center}
\caption{A bridge of length $22$ with $6$ visits of the $x$-axis: three $1$-visits, two $2$-visits, and one $4$-visit.\label{fig:ZeroContacts}}
\end{figure}

In order to examine the limiting behaviour of $X_{n,j}$ we start with a combinatorial description of the problem using the before mentioned arch decomposition.
Let $\mathcal{B}$ be the combinatorial family of bridges of length $\ge 0$ and $\mathcal{D}$ be the family of \emph{positive} Dyck path excursions of length $\ge 2$, i.e., Dyck paths of positive length starting and ending on the $x$-axis, where all points in between are above the $x$-axis. Analogous, let $\overline{\mathcal{D}}$ be the family of \emph{negative} Dyck path excursions of length $\ge 2$, i.e., Dyck paths of positive length starting and ending on the $x$-axis, where all points in between are below the $x$-axis.
Then, $\mathcal{B}$ consists of a sequence of positive and negative Dyck path excursions, i.e., it can be described combinatorially by the \textsc{Seq} construction:
\begin{equation}\label{eqn:B_family}
  \mathcal{B} = \text{\textsc{Seq}}(\mathcal{D} \: \dot{\cup} \: \overline{\mathcal{D}}).
\end{equation}
Furthermore, the family $\mathcal{D}$ can be described formally as
\begin{equation}\label{eqn:D_family}
  \mathcal{D} = \nearrow\raisebox{2mm}{\text{\textsc{Seq}$(\mathcal{D})$}}\searrow.
\end{equation}
Let $D_{n}$ be the number of positive Dyck path excursions of length $2n$ and
$D(z) := \sum_{n \ge 1} D_{n} z^{n}$ its generating function; of course, due to symmetry, they coincide with the corresponding quantities for negative Dyck path excursions. Furthermore, let $B(z) := \sum_{n \ge 0} B_{n} z^{n}$ be the generating function of the number of bridges $B_{n}$ of length $2n$. Equations~\eqref{eqn:B_family}-\eqref{eqn:D_family} immediately yield the following equations for the generating functions,
\begin{equation*}
  B(z) = \frac{1}{1-2D(z)} \quad \text{and} \quad D(z) = \frac{z}{1-D(z)},
\end{equation*}
with solutions
\begin{equation*}
  D(z) = \frac{1-\sqrt{1-4z}}{2} \quad \text{and} \quad B(z) = \frac{1}{\sqrt{1-4z}}.
\end{equation*}
Of course, by extracting coefficients we reobtain $B_{n} = \binom{2n}{n}$, whereas the sequence $D_{n} = \frac{1}{n} \binom{2(n-1)}{n-1}$ is enumerated by the (shifted) Catalan numbers.
But more interestingly, the above combinatorial description~\eqref{eqn:B_family}-\eqref{eqn:D_family} can be extended easily to enumerate a suitably introduced bivariate generating function of the distribution of $X_{n,j}$:
\begin{equation*}
  B_{j}(z,v) := \sum_{n \ge 0} \sum_{m \ge 0} B_{n} \mathbb{P}\{X_{n,j} = m\} z^{n} v^{m}.
\end{equation*}
We get then for $j \ge 1$ the solution
\begin{equation}
  B_{j}(z,v) = \frac{1}{1-2\big(D(z) + (v-1) D_{j} z^{j}\big)}
  = \frac{1}{\sqrt{1-4z} - 2 (v-1) D_{j} z^{j}}.
\end{equation}
In order to obtain the factorial moments of $X_{n,j}$ we set $w:=v-1$ and introduce the function
$\tilde{B}_{j}(z,w) := B_{j}(z,w+1)$, thus yielding
\begin{equation*}
  \tilde{B}_{j}(z,w) = \frac{1}{\sqrt{1-4z} - 2 w D_{j} z^{j}}.
\end{equation*}
Extracting coefficients by using standard singularity analysis easily gives
\begin{align*}
  \mathbb{E}(X_{n,j}^{\underline{s}}) & = 
  \frac{s!}{B_{n}} [z^{n} w^{s}] \tilde{B}_{j}(z,w) = 
  \frac{s!}{B_{n}} [z^{n}] \frac{2^{s} D_{j}^{s} z^{js}}{(1-4z)^{\frac{s+1}{2}}}\\
  & = \frac{s! \, 2^{s} D_{j}^{s}}{B_{n}} [z^{n-js}] \frac{1}{(1-4z)^{\frac{s+1}{2}}}
  \sim \frac{s! \, 2^{s} D_{j}^{s} 4^{n-js} (n-js)^{\frac{s-1}{2}}}{B_{n} \, \Gamma(\frac{s+1}{2})}.
\end{align*}
If $j=o(n)$ we further get
\begin{equation*}
  \mathbb{E}(X_{n,j}^{\underline{s}}) \sim \frac{s! \, 2^{s} D_{j}^{s} 4^{n-js} n^{\frac{s-1}{2}}}{B_{n} \, \Gamma(\frac{s+1}{2})},
\end{equation*}
and, together with $B_{n} \sim \frac{4^{n}}{\sqrt{\pi} \sqrt{n}}$ and the duplication formula of the factorials \eqref{eqn:duplication}, this gives
\begin{align*}
  \mathbb{E}(X_{n,j}^{\underline{s}}) & \sim \frac{s! \, 2^{s} D_{j}^{s} 4^{-js} \sqrt{\pi} \, n^{\frac{s}{2}}}{\Gamma(\frac{s+1}{2})}
  = 2^{s} D_{j}^{s} (4^{-j})^{s} n^{\frac{s}{2}} 2^{s} \Gamma\big(\frac{s}{2}+1\big)\\
  & = \Big(\frac{2 \sqrt{2} D_{j} \sqrt{n}}{4^{j}}\Big)^{s} \, 2^{\frac{s}{2}} \, \Gamma\big(\frac{s}{2}+1\big).
\end{align*}
Thus, by an application of Lemma~\ref{MOMSEQMainLemma}, we get the following characterization of the limit law of $X_{n,j}$.
\begin{theorem}
The random variable $X_{n,j}$, counting the number of $j$-visits of the $x$-axis in a random bridge of length $2n$, has, for $n \to \infty$ and arbitrary $1 \le j=j(n) \le n$ with $j=o(n)$, asymptotically factorial moments of mixed Poisson type with a Rayleigh mixing distribution $X$ and scale parameter $\lambda_{n,j} = \frac{2 \sqrt{2} \binom{2(j-1)}{j-1} \sqrt{n}}{j \, 4^{j}}$: 
\[
\E(\fallfak{X_{n,j}}s)= (\lambda_{n,j})^{s} \, 2^{\frac{s}2} \, \Gamma\big(\frac{s}{2}+1\big) \cdot \big(1+o(1)\big).
\]
\begin{itemize}
\item[(i)] for $\lambda_{n,j}\to\infty$, the random variable $\frac{X_{n,j}}{\lambda_{n,j}}$ converges in distribution, with convergence of all moments, to $X$. 

\item[(ii)] for $\lambda_{n,j}\to\rho\in(0,\infty)$, the random variable $X_{n,j}$ converges in distribution, with convergence of all moments, to $Y\law\MPo(\rho X)$. 
\end{itemize}
Moreover, the random variable $Y\law \MPo(\rho X)$ converges, for $\rho\to\infty$, after scaling, to its mixing distribution $X$: $\frac{Y}{\rho}\claw X$, with convergence of all moments.
\end{theorem}
Of course, the result above can be readily adapted to obtain joint distributions for the number of $j$-visits and the total number of visits of the $x$-axis as considered by Flajolet and Sedgewick~\cite{FlaSed2009}; see also Section~\ref{Compositions}. 
Our results can be extended to other bridges with different families of steps (see~\cite{BandFla2002}). One can also study this parameter for modified excursions, as considered in~\cite{BandWallner2014}, and obtain similar results.

\subsection{Cyclic points and trees in graphs of random mappings\label{Ssec:CyclePoints}}
We call a function $f: [n] \to [n]$ from the finite set $[n] := \{1, 2, \dots, n\}$ into itself an \emph{$n$-mapping} (or an \emph{$n$-mapping function}); let us denote by $\mathcal{F}_{n}$ the set of $n$-mappings. When selecting one of the $n^{n}$ $n$-mappings at random (i.e., if we assume that each of the $n^{n}$ $n$-mappings can occur equally likely) one speaks about a random $n$-mapping.
There exists a vast literature (see, e.g., \cite{DS97,DS95,FlaOdl1990} and references therein) devoted to reveal the typical behaviour of important quantities (as, e.g., the number of components, the number of cyclic nodes, etc.) of random $n$-mappings and the corresponding mapping graphs, respectively.

The \emph{mapping graph}, i.e., the \emph{functional digraph}, of an $n$-mapping $f \in \mathcal{F}_{n}$ is the directed graph $G_{f} = (V,E)$ with set of vertices $V = [n]$ and set of directed edges $E=\{(i,f(i)), i \in [n]\}$.
The structure of the mapping graph $G_{f}$ of an arbitrary mapping function $f$ is well known \cite{DS97,FlaOdl1990}: the weakly connected components of $G_{f}$ are cycles of \emph{rooted labelled trees}, i.e., \emph{Cayley-trees}, which means that each connected component consists of rooted labelled trees (with edges oriented towards the root nodes) whose root nodes are connected by directed edges such that they are forming a cycle.

This description allows to interpret a mapping $f$ as a set of cycles of labelled trees. 
Hence, in order to describe the family $\mathcal{F}=\bigcup_{n\ge 0}\mathcal{F}_n$ of all mappings, we can apply the combinatorial constructions \textsc{Set} and \textsc{Cycle} to the family of 
rooted labelled trees $\mathcal{T}$, as introduced and discussed in Section~\ref{Ssec:Records}.
This yields the following combinatorial description of the family of mappings $\mathcal{F}$:
\[
\mathcal{F}=\text{\sc Set}(\text{\sc Cycle}(\mathcal{T})).
\]
Hence, the exponential generating function 
\begin{equation*}
  F(z) := \sum_{n\ge 0}n^n\frac{z^n}{n!} = \sum_{n \ge 0} \tilde{F}_{n} z^{n},
\end{equation*}
with $\tilde{F}_{n} := \frac{n^{n}}{n!}$, of the number of $n$-mappings satisfies
\begin{equation}\label{eqn:Fz_Def}
F(z)=\exp\Big(\log\big(\frac{1}{1-T(z)}\big)\Big)=\frac{1}{1-T(z)},
\end{equation}
where the tree function $T(z)$ is defined in \eqref{eqn:Tree-function}.
Equation~\eqref{eqn:Fz_Def} suggests the alternative combinatorial description of mappings as sequences of rooted labelled trees:
$\mathcal{F}=\text{\textsc{Seq}}(\mathcal{T})$. This can be justified easily by using an analogue of the canonical cycle representation of permutations: for each cycle of trees order the cycle by starting with the tree having the largest root-label amongst all these trees (let us call it the cycle leader) and then rank the different cycles in descending order of the cycle leaders.

The random variable $X_{n}$ counting the number of cyclic points in a random $n$-mapping $f$ (i.e., elements $k \in [n]$, such that there exists a $\ell > 0$ with $k = f^{\ell}(k)$), which of course coincides with the number of rooted labelled trees in the decomposition of the mapping graph $G_{f}$ given above, has been analysed by Drmota and Soria~\cite{DS97}. They have shown a Rayleigh limit law:
$\frac{X_{n}}{\sqrt{n}} \claw \text{Rayleigh}(1)$.
We are considering here a refinement of $X_{n}$, namely, we introduce the random variables $X_{n,j}$, counting the number of trees of size $j$ occurring in the decomposition of the mapping graph $G_{f}$ of a random $n$-mapping $f \in \mathcal{F}_{n}$; of course, it holds $X_{n} = \sum_{j=1}^{n} X_{n,j}$.
The quantities considered are visualized in Figure~\ref{fig:Mappings}.
\begin{figure}
\begin{center}
  \includegraphics[height=3.5cm]{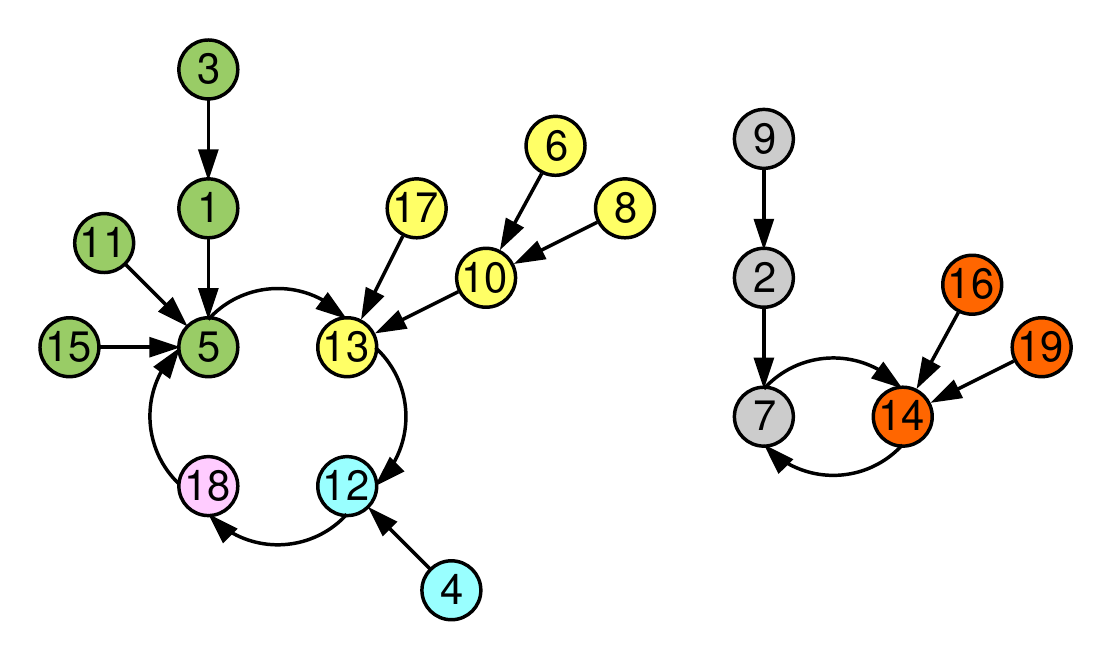}
\end{center}
\caption{The mapping graph $G_{f}$ of a $19$-mapping function $f$ has $6$ cyclic points, i.e., decomposes into $6$ trees, one tree of size $1$, one tree of size $2$, two trees of size $3$ and two trees of size $5$.\label{fig:Mappings}}
\end{figure}

When introducing a suitable generating function of the distribution of $X_{n,j}$ via $F_{j}(z,v) := \sum_{n \ge 0} \sum_{m \ge 0} \tilde{F}_{n} \mathbb{P}\{X_{n,j}=m\} z^{n} v^{m}$, the combinatorial decomposition of the family of mappings $\mathcal{F}$ as given above immediately yields an explicit formula for $F_{j}(z,v)$, $j \ge 1$:
\begin{equation*}
  F_{j}(z,v) = \frac{1}{1-\big( T(z) - \tilde{T}_{j} z^{j} + v \tilde{T}_{j} z^{j}\big)} = \frac{1}{1-T(z)-(v-1)\tilde{T}_{j}z^{j}}.
\end{equation*}
Introducing $w:=v-1$ and $\tilde{F}_{j}(z,w) := F(z,w+1)$ gives
\begin{equation*}
  \tilde{F}_{j}(z,w) = \frac{1}{1-T(z)-w\tilde{T}_{j} z^{j}},
\end{equation*}
from which the factorial moments of $X_{n,j}$ can be obtained easily by extracting coefficients.
This can be done completely analogous to the computations in Section~\ref{Ssec:Records} leading for $j=o(n)$ to the following result:
\begin{equation*}
  \mathbb{E}(X_{n,j}^{\underline{s}}) = \frac{s!}{\tilde{F}_{n}} [z^{n} w^{s}] \tilde{F}_{j}(z,w) \sim \Big(\frac{\tilde{T}_{j} \sqrt{n}}{e^{j}}\Big)^{s} \, 2^{\frac{s}{2}} \, \Gamma\big(\frac{s}{2}+1\big),
\end{equation*}
which, together with Lemma~\ref{MOMSEQMainLemma}, shows the theorem stated below.
\begin{theorem}
The random variable $X_{n,j}$, counting the number of trees of size $j$ occurring in the decomposition of the mapping graph of a random $n$-mapping, has, for $n \to \infty$ and arbitrary $1 \le j=j(n) \le n$ with $j=o(n)$, asymptotically factorial moments of mixed Poisson type with a Rayleigh mixing distribution $X$ and scale parameter $\lambda_{n,j} = \frac{j^{j-1} \sqrt{n}}{j! e^{j}}$: 
\[
\E(\fallfak{X_{n,j}}s)= (\lambda_{n,j})^{s} \, 2^{\frac{s}{2}} \, \Gamma\big(\frac{s}{2}+1\big) \cdot \big(1+o(1)\big).
\]
\begin{itemize}
\item[(i)] for $\lambda_{n,j}\to\infty$, the random variable $\frac{X_{n,j}}{\lambda_{n,j}}$ converges in distribution, with convergence of all moments, to $X$. 

\item[(ii)] for $\lambda_{n,j}\to\rho\in(0,\infty)$, the random variable $X_{n,j}$ converges in distribution, with convergence of all moments, to $Y\law\MPo(\rho X)$. 
\end{itemize}
Moreover, the random variable $Y\law \MPo(\rho X)$ converges, for $\rho\to\infty$, after scaling, to its mixing distribution $X$: $\frac{Y}{\rho}\claw X$, with convergence of all moments.
\end{theorem}

\section{Multivariate mixed Poisson distributions\label{SecMMPD}}
Definition~\ref{MOMSEQdef1} readily extends to multivariate distributions, compare with~\cite{Ferrari}.
\begin{defi}
\label{MOMSEQdef2}
Let $(X_1,\dots,X_m)$ denote a random vector with non-negative components and cumulative distribution function $\boldsymbol{\Lambda}(.)$ and $\rho_1,\dots \rho_m>0$ scale parameters. 
The discrete random vector $(Y_1,\dots,Y_m)$ with joint probability mass function given by  
\[
\P\{Y_1=\ell_1,\dots Y_m=\ell_m\}=\frac{\rho_1^{\ell_1}\dots\rho_m^{\ell_m}}{\ell_1!\dots\ell_m!}
\int_{(\R^+)^{m}}X_1^{\ell_1}\dots X_m^{\ell_m}e^{-\sum_{j=1}^{m}\rho_jX_j}d\boldsymbol{\Lambda},
\]%
$\ell_1,\dots,\ell_m\ge 0$, has a multivariate mixed Poisson distribution with mixing distribution $(X_1,\dots,X_m)$ and scale parameters $\rho_1,\dots,\rho_m$.
\end{defi}

Relation~\eqref{MOMSEQ-eqn2} for the moments of $X$ and $Y$ given there extends to the multivariate case in the following way:
\begin{equation}
\label{MOMSEQ-eqn3}
\E\big(\fallfak{Y_1}{s_1}\dots \fallfak{Y_m}{s_m}\big)=\rho_1^{s_1}\dots \rho_m^{s_m}\mu_{s_1,\dots,s_m},
\end{equation}
where $\mu_{s_1,\dots,s_m}=\E(X_1^{s_1}\cdots X_m^{s_m})$, for $s_1,\dots,s_m\ge 0$; this can readily be seen by a direct computation.

Similar to Proposition~\ref{MOMSEQthe1} we obtain the following result, if the distribution of the random vector $(X_1,\dots,X_m)$ is uniquely determined by the sequence of its (mixed) moments.

\begin{prop}
Let $\mathbf{X}=(X_1,\dots,X_m)$ denote a random vector determined by its sequence of mixed moments $(\mu_{s_1,\dots,s_m})_{s_1,\dots,s_m\in\N_0}$, 
assuming that the moment generating function $\psi(\mathbf{z})=\E(e^{\mathbf{z}\mathbf{X}})$ of $\mathbf{X}$ 
exists in a neighbourhood of $\boldsymbol{0}$ including $-\boldsymbol{\rho}$. Then, the random vector $\mathbf{Y}=(Y_1,\dots,Y_m)$ with mixed factorial moments given by~\eqref{MOMSEQ-eqn3} has a multivariate mixed Poisson distribution with mixing distribution $\mathbf{X}$ and scale parameters $\rho_1,\dots \rho_m>0$. The moment generating function
$\varphi(\mathbf{z})=\E(e^{\mathbf{z}\mathbf{Y}})$ is given by the Stirling transform of $\psi(\mathbf{z})$,
\[
\varphi(\mathbf{z})=\big(\psi(\mathbf{z})\big)\Big|_{\displaystyle{z_1=\rho_1(e^{z_1}-1),\dots,z_m=\rho_m(e^{z_m}-1)}},
\]
and the probability mass function of $\mathbf{Y}$ satisfies
\[
\P\{Y_1=\ell_1,\dots, Y_m=\ell_m\}= \sum_{\ell_1\ge j_1,\dots,\ell_m\ge j_m}\mu_{j_1,\dots,j_m}\prod_{i=1}^{m}
\binom{j_i}{\ell_i}(-1)^{j_i-\ell_i}\frac{\rho_i^{j_i}}{j_i!},
\]
for $\ell_1,\dots,\ell_m\ge 0$.
\end{prop}

\begin{proof}
We proceed similarly to the proof of Proposition~\ref{MOMSEQthe1}. Using 
\[
Y_1^{s_1}\dots Y_{m}^{s_m}=\sum_{\ell_1=0}^{s_1}\dots\sum_{\ell_m=0}^{s_m}\bigg(\prod_{i=1}^{m}\fallfak{Y_i}{s_i}\bigg)
\]
and~\eqref{MOMSEQWilf}, the moment generating function of $\mathbf{Y}$ is readily computed:
\begin{align*}
\varphi(\mathbf{z}) & = \sum_{s_1,\dots,s_m\ge 0}\E(Y_1^{s_1}\dots Y_{m}^{s_m})\frac{z_1^{s_1}\dots z_m^{s_m}}{s_1!\dots s_m!}\\
& = \sum_{\ell_1,\dots,\ell_m\ge 0}\mu_{\ell_1,\dots,\ell_m}\prod_{i=1}^{m}\Big(\rho_i(e^{z_i}-1)\Big)^{\ell_i}.
\end{align*}
Moreover, it coincides with the Stirling transform of $\psi(\mathbf{z})$. This proves that the random vector has a multivariate mixed Poisson law, 
since $\varphi(\mathbf{z})$ is analytic in a neighbourhood of $\boldsymbol{0}$. Moreover, the probability mass function is obtained according to
\[
\P\{Y_1=\ell_1,\dots, Y_m=\ell_m\}= \frac{\rho_1^{\ell_1}\dots\rho_m^{\ell_m}}{\ell_1!\dots\ell_m!}\cdot\bigg(\frac{\partial^{\sum_{k=1}^{m}\ell_k}}{\partial z_1^{\ell_1}\dots \partial z_m^{\ell_m}}\varphi(\mathbf{z})\bigg)\bigg|_{\mathbf{z}=-\boldsymbol{\rho}}.
\]
\end{proof}

Moreover, we note that the basic limit theorem for the univariate case given in Lemma~\ref{MOMSEQMainLemma} can be readily extended to limit laws for random vectors. 

\section{Outlook and extensions}
\subsection{Mixed Poisson distributions in Analytic Combinatorics - compositions\label{Compositions}}
Looking at the examples in Sections~\ref{Ssec:Records} and ~\ref{Ssec:Dyck} leading 
to mixed Poisson laws with Rayleigh mixing distribution, and the example in Section~\ref{MOMSEQExampleBlocks} from~\cite{PanKuCPC}, it is desirable to find a unifying combinatorial scheme leading to mixed Poisson distributions. The generating functions appearing in the examples mentioned aforehand indicate how to do so: we
can use the critical compositions (see Flajolet and Sedgewick~\cite[Proposition IX.24, page~712]{FlaSed2009}, based on the pioneering works of Soria et al.~\cite{BFSS01,DS97,DS95,FS90,FS91,FS93}.
Assumed that generating functions $G(z)$ and $H(z)$ are the counting series of certain combinatorial 
families. We are interested in compositions of generating functions of the form $G(H(z))$. Combinatorially, this amounts to a substitution between structures of the form $\mathcal{F}=\mathcal{G}\circ\mathcal{H}$, We measure the size of the so-called core $X_n$, and additionally taking into account the contribution of parts of size $j$ to the core:
\[
F(z,u,v)=G\big(u(H(z)-(v-1)H_j z^j)\big).
\] 
Here the variable $u$ marks as usual the total size of the so-called core $X_n$, 
\[
\P\{X_{n}=m\}=\frac{[z^nu^k]F(z,u,1)}{[z^n] G(H(z))},
\]
and the new variable $v$ marks the contribution of parts of size $j$ measured by the random variable $X_{n,j}$ to the core,
\[
\P\{X_{n,j}=m\}=\frac{[z^n v^j]F(z,1,v)}{[z^n] G(H(z))},
\]
such that $X_n=\sum_{j=1}^n X_{n,j}$. Using the semi-large power theorem stated in~\cite[Theorem IX.16, page~709]{FlaSed2009}, one can study this $j$-part core $X_{n,j}$. More generally, it is desirable to study the joint distributions $(X_n;X_{n,j_1},\dots,X_{n,j_k})$ via
\[
F(z,u,\mathbf{v})=G\big(u(H(z)-\sum_{\ell=1}^{k}(v_\ell-1)H_{j_\ell}z^{j_\ell})\big).
\] 
We will report on our findings on this refined analysis of compositions elsewhere~\cite{KuCo}.

\subsection{Extensions and open problems}
Several of the results presented in this work can be extended to a multivariate analysis involving a random vector $\mathbf{X}_{n,\mathbf{j}}=(X_{n,j_1},\dots,X_{n,j_r})$ by its mixed moments.
It should be possible, at least for some of the examples presented, to use mixed Poisson approximation techniques to derive distances, i.e., the total variation distance, between the random
variables of interest and the corresponding mixed Poisson distributions, refining the results stated in this work.

\smallskip

A natural question is to ask for a direct explanation of the critical growth rates of the second parameter, say $j$, when the limit laws of the random variables $X_{n,j}$ change from the law of the mixing distribution $X$ to the mixed Poisson distributions $\MPo(\rho X)$. 

\section*{Acknowledgements}
The authors thank Svante Janson for very helpful remarks on mixed Poisson distributions, in particular regarding notation and the correct approach to Lemma~\ref{MOMSEQLemmaUniqueness}. Furthermore, the authors warmly thank Walter Kuba for several motivating discussions and for providing a concise argument used in the series related part of Lemma~\ref{MOMSEQLemmaUniqueness}.


\begin{thebibliography}{00}

\bibitem{ABBH2013}
L.~Addario-Berry, N.~Broutin, C.~Holmgren, Cutting down trees with a Markov chainsaw. \emph{The Annals of Applied Probability} 24, 2297--2339, 2014.

\bibitem{BFSS01} 
C.~Banderier, P.~Flajolet, G.~Schaeffer, M.~Soria, Random  maps, coalescing saddles, singularity analysis and Airy phenomena. \emph{Random Structures \& Algorithms} 19, 194--246, 2001.

\bibitem{BandFla2002}
C.~Banderier and P.~Flajolet, Basic analytic combinatorics of directed lattice paths. \emph{Theoretical
Computer Science} 281, 37--80, 2002.

\bibitem{BandWallner2014}
C.~Banderier and M.~Wallner, \emph{Some reflections on directed lattice paths}, \emph{DMTCS Proceedings Series}, Proceedings of AofA 2014, to appear, 2014.

\bibitem{Bertoin2012}
J.~Bertoin, Fires on trees. \emph{Annales de l'Institut Henri Poincar\'{e} Probabilit\'{e}s et Statistiques} 48, no.~4, 909--921, 2012.

\bibitem{Bertoin2012b}
J.~Bertoin, Sizes of the largest clusters for supercritical percolation on random recursive trees. \emph{Random Structures \& Algorithms} 44, 29--44, 2014.

\bibitem{Bertoin2014a}
J.~Bertoin, The cut-tree of large recursive trees. \emph{Annales de l'Institut Henri Poincar\'{e}}, to appear. 

\bibitem{Bertoin2014b}
E.~Baur and J.~Bertoin, Cutting edges at random in large recursive trees. Preprint, available at \url{http://hal.archives-ouvertes.fr/docs/01/00/31/51/PDF/cutting-edges.pdf}. 

\bibitem{Bertoin2012+}
J.~Bertoin and G.~Miermont, The cut-tree of large Galton-Watson trees and the Brownian CRT. \emph{Ann. Appl. Probab.} 23, 1469--1493, 2013. 

\bibitem{bai2002}
Z.-D.~Bai, F.~Hu and L.-X. Zhang, Gaussian approximation theorems for urn models and their 
applications, \emph{Annals of Applied Probability} 12, 1149--1173, 2002.

\bibitem{BerFlaSal1992}
F.~Bergeron, P.~Flajolet and B.~Salvy, Varieties of Increasing Trees,
\emph{Lecture Notes in Computer Science} 581, 24--48, 1992.

\bibitem{Sloane}
M.~Bernstein and N.~J.~A.~Sloane, Some canonical sequences of integers. \emph{Linear Algebra and its Applications} 226/228, 57--72, 1995.

\bibitem{Brenti1989}
F.~Brenti, Unimodal, log-concave, and P\'olya frequency sequences in
combinatorics,
\emph{Memoirs Amer. Math. Soc.} 81, no.~413, 1989.

\bibitem{Brenti1998}
F.~Brenti, Hilbert polynomials in combinatorics,
\emph{J. Algebraic Combinatorics} 7, 127--156, 1998.

\bibitem{Carle}
T.~Carleman, Sur le probl\'eme des moments, \emph{Acad. Sci. Paris} 174, 1680--1682, 1922.

\bibitem{ChaLou2002}
P.~Chassaing and G.~Louchard, Phase transition for parking blocks, Brownian excursion and coalescence, 
\emph{Random Structures \& Algorithms} 21, 76--119, 2002.

\bibitem{Curt1942}
J.~H.~Curtiss, A note on the theory of moment generating functions,
\emph{Annals of Math. Stat.} 13, 430--433, 1942.

\bibitem{DS97}
M.~Drmota and M.~Soria, Images and preimages in random mappings, \emph{SIAM Journal on Discrete Mathematics} 10, 246--269, 1997.

\bibitem{DS95}
M.~Drmota and M.~Soria. Marking in combinatorial constructions: Generating functions and limiting distributions, \emph{Theoretical Computer Science} 144, 67--99, 1995.

\bibitem{Ferrari}
A.~Ferrari, G.~Letacy and J.-Y.~Tourneret, Multivariate mixed Poisson distributions, \emph{EUSIPCO-04}, Vienna, Austria, 2004.

\bibitem{FlaDumPuy2006}
P.~Flajolet, P.~Dumas and V.~Puyhaubert, Some exactly solvable
models of urn process theory, 
\emph{Discrete Mathematics and
Theoretical Computer Science}, vol. AG, 59--118, 2006, in
``Proceedings of Fourth Colloquium on Mathematics and Computer
Science'', P.~Chassaing Editor.

\bibitem{FlaGabPek2005}
P.~Flajolet, J.~Gabarr{\'{o}} and H.~Pekari, Analytic urns,
\emph{Annals of Probability} 33, 1200--1233, 2005.

\bibitem{FlaOdl1990}
P.~Flajolet and A.~Odlyzko, Random mapping statistics.
in: \emph{Advances in cryptology -- EUROCRYPT '89}, 
Lecture Notes in Computer Science 434, 329--354, Springer, Berlin, 1990.

\bibitem{FlaSed2009}
P. Flajolet and R. Sedgewick,
\emph{Analytic Combinatorics}.
Cambridge Univ. Press, Cambridge, UK, 2009.

\bibitem{FS90} 
P.~Flajolet and M.~Soria, Gaussian limiting distributions for the number of components in combinatorial structures, \emph{J. Combinatorial Theory, Series A} 53, 165--182, 1990.

\bibitem{FS91} 
P.~Flajolet and M.~Soria, The cycle construction, \emph{SIAM Journal on Discrete Mathematics} 4, 58--60, 1991.

\bibitem{FS93} 
P.~Flajolet and M.~Soria, General combinatorial schemas: Gaussian limit distributions and  exponential tails, \emph{Discrete Mathematics} 114, 159--180, 1993.

\bibitem{VitFla1990}
P.~Flajolet and J.~Vitter, {Average case analysis of algorithms and data structures},
in \emph{Handbook of Theoretical Computer Science}, 431--524,
Elsevier, Amsterdam, 1990. 

\bibitem{GessStan1978}
I.~Gessel and R.~P.~Stanley, Stirling polynomials,
\emph{Journal of Combinatorial Theory, Series A} 24, 24--33, 1978.

\bibitem{Concrete}
R.~L.~Graham, D.~E.~Knuth, and O.~Patashnik, \emph{Concrete Mathematics}, Second Edition. Reading, Massachusetts: Addison-Wesley, 1994.

\bibitem{Gruebi2005}
R.~Gr\"ubel and N.~Stefanoski, Mixed Poisson approximation of node depth distributions in random binary search trees. \emph{The Annals of Applied Probability} 15, 279--297, 2005. 

\bibitem{Gut2002}
A.~Gut, On the moment problem, \emph{Bernoulli} 8, 407--421, 2002.

\bibitem{Hwang2002}
H.-K.~Hwang and R.~Neininger, Phase change of limit laws in the quicksort recurrence under varying toll functions.
\emph{SIAM Journal on Computing} 31, 1687--1722, 2002.

\bibitem{Janson2004}
S.~Janson, Functional limit theorems for multitype branching
processes and generalized P{\'{o}}lya urns, 
\emph{Stochastic processes and applications} 110, 177--245, 2004.

\bibitem{Jan2004} 
S.~Janson, Random records and cuttings in complete binary trees. In: \emph{Mathematics and Computer Science III, Algorithms, Trees, Combinatorics and Probabilities}, M.~Drmota, P.~Flajolet, D.~Gardy, B.~Gittenberger (eds.), 241--253, Birkh{\"a}user, Basel, 2004.

\bibitem{Janson2006b}
S.~Janson, Random cutting and records in deterministic and random trees. \emph{Random Structures \& Algorithms} 29, 139--179, 2006.


\bibitem{Janson2006}
S.~Janson, Limit theorems for triangular urn schemes, 
\emph{Probability Theory and Related Fields 134}, 417--452, 2005. 

\bibitem{Janson2008}
S.~Janson, Plane recursive trees, Stirling permutations and an urn model,
Proceedings, Fifth Colloquium on Mathematics and Computer Science (Blaubeuren, 2008) , \emph{DMTCS Proceedings AI}, 541--548, 2008.

\bibitem{Janson2010}
S.~Janson, Moments of Gamma type and the Brownian supremum process area. \emph{Probability Surveys}, 7, 1--52, 2010.

\bibitem{Janson2011}
S.~Janson, M.~Kuba and A.~Panholzer. Generalized Stirling permutations, families of increasing trees and urn models.
\emph{Journal of Combinatorial Theory, Series A} 118, 94--114, 2011

\bibitem{Johnson1992}
N.~L.~Johnson, S.~Kotz and A.~W.~Kemp, \emph{Univariate Discrete Distributions}, 2.~Edition, New York, John Wiley, 1992.

\bibitem{Karlis}
D.~Karlis and E.~Xekalaki, Mixed Poisson Distributions, \emph{International Statistical Review} 73, 35--58, 2005.

\bibitem{Knuth1998}
D.~E.~Knuth, \emph{The Art of Computer Programming, Volume 3: Sorting and Searching}, second edition, Addison-Wesley, Reading, 1998.

\bibitem{Koganov}
L.~ M.~Koganov, Universal bijection between Gessel-Stanley permutations and diagrams of connections of corresponding ranks. 
\emph{Uspekhi Mat. Nauk} 51, no. 2(308), 165--166, 1996; English translation in \emph{Russian
Math. Surveys} 51, no. 2, 333--335, 1996.

\bibitem{KonWei1966}
A.~G.~Konheim and B.~Weiss, An occupancy discipline and applications, \emph{SIAM Journal on Applied Mathematics} 14, 1266--1274, 1966.

\bibitem{KuCo}
M.~Kuba, A note on Critical Compositions and Mixed Poisson distributions, manuscript.

\bibitem{Desc-KubPan2005}
M.~Kuba and A.~Panholzer, Descendants in increasing trees, \emph{Electronic Journal of Combinatorics} 13 (1), Paper 8, 2006.

\bibitem{KubPan2006}
M.~Kuba and A.~Panholzer, Analysis of label-based parameters in increasing trees.
\emph{Discrete Mathematics and Theoretical Computer Science}, Proceedings AG, 321--330, 2006.

\bibitem{KubPan2007}
M.~Kuba and A.~Panholzer, On the degree distribution of the nodes in increasing trees,
\emph{Journal of Combinatorial Theory, Series~A} 114, 597--618, 2007.

\bibitem{KuPan2008}
M.~Kuba and A.~Panholzer, Isolating nodes in recursive trees. \emph{Aequationes Mathematicae} 76, 258--280, 2008.

\bibitem{PanKuCPC}
M.~Kuba and A.~Panholzer, Analysis of statistics for generalized Stirling permutations, 
\emph{Combinatorics, Probability and Computing} 20, 875--910, 2011.

\bibitem{PanKuAofA2012}
M.~Kuba and A.~Panholzer, On death processes and urn models. \emph{Discrete Mathematics and Theoretical Computer Science},
in: ``Proceedings of the 23rd International Meeting on Probabilistic, Combinatorial and Asymptotic Methods for the Analysis of Algorithms (AofA 2012)'', Proceedings AQ, 29--42, 2012.

\bibitem{PanKuAdvances}
M.~Kuba and A.~Panholzer, Limiting distributions for a class of diminishing urn models,
\emph{Advances in Applied Probability} 44, 1--31, 2012.

\bibitem{KuPan2014}
M.~Kuba and A.~Panholzer, Isolating nodes in recursive trees, \emph{Online Journal of Analytic Combinatorics} 9, 2014, 26~pages.

\bibitem{Loe1977}
M.~Lo{\`{e}}ve, \emph{Probability Theory I}, 4th Edition, Springer-Verlag, New York, 1977.

\bibitem{MahSmy1991}
H.~Mahmoud and R.~Smythe, On the distribution of leaves in rooted subtrees of recursive trees,
\emph{Annals of Applied Probability} 1, 406--418, 1991. 

\bibitem{MeiMoo1970}
A.~Meir and J.~W.~Moon, Cutting down random trees. \emph{Journal
of the Australian Mathematical Society} 11, 313--324, 1970.

\bibitem{MeiMoo1974}
A.~Meir and J.~W.~Moon, Cutting down recursive trees.
\emph{Mathematical Biosciences} 21, 173--181, 1974.

\bibitem{Masse}
J.~C.~Mass\'e, and R.~ Theodorescu, Neyman type A distribution revisited. \emph{Statistica Neerlandica}, Volume 59, Issue 2, 206-213, 2005.

\bibitem{Neyman}
J.~Neyman, On a new class of ``contagious'' distributions applicable in entomology and
bacteriology, \emph{Annals of Mathematical Statistics} 10, 35-57, 1939.

\bibitem{Pan2003}
A.~Panholzer, Non-crossing trees revisited: cutting down and
spanning subtrees. In: \emph{Discrete Random Walks}, C.~Banderier
and C.~Krattenthaler (eds.), \textit{Discrete Mathematics and
Theoretical Computer Science}, Proceedings AC, 265--276, 2003.

\bibitem{Pan2006}
A.~Panholzer, Cutting down very simple trees. \emph{Quaestiones Mathematicae} 29, 211--228, 2006.

\bibitem{PanPro2005+}
A.~Panholzer and H.~Prodinger, 
The level of nodes in increasing trees revisited,
\emph{Random Structures \& Algorithms} 31, 203--226, 2007.

\bibitem{PanSeitz2011}
A.~Panholzer and G.~Seitz, Limiting distributions for the number of inversions in labelled tree families, \emph{Annals of Combinatorics} 16, 847--870, 2012.

\bibitem{PanSeitz2012}
A.~Panholzer and G.~Seitz, Ancestors and descendants in evolving k-tree models, \emph{Random Structures \& Algorithms} 44, 465--489, 2014.

\bibitem{Park1994b}
S.~K.~Park, The $r$-multipermutations,
\emph{Journal of Combinatorial Theory, Series A} 67, 44--71, 1994.

\bibitem{Park1994a}
S.~K.~Park, Inverse descents of $r$-multipermutations, \emph{Discrete Mathematics} 132, 215--229, 1994.

\bibitem{Park1994c}
S.~K.~Park, $P$-partitions and $q$-Stirling numbers,
\emph{Journal of Combinatorial Theory, Series A} 68, 33--52, 1994.

\bibitem{Pitman}
J.~Pitman, Exchangeable and partially exchangeable random
partitions, \emph{Probab. Theory Relat.~Fields} 102, 145--158, 1995.

\bibitem{Pou2008}
N.~Pouyanne, An algebraic approach to P\'olya processes.
\emph{Annales de l'Institut Henri Poincar\'e} 44, 293--323, 2008.

\bibitem{Pivault2011}
N.~ Privault, Generalized Bell polynomials and the combinatorics of Poisson central moments. \emph{Electronic Journal of Combinatorics} 18, Paper 54, 2011.

\bibitem{Puy2005}
V. Puyhaubert, \emph{Mod\'eles d'urnes et ph\'enom\'enes de seuils en combinatoire analytique}, Ph.D.\ thesis, \'Ecole Polytechnique, Palaiseau, 2005.

\bibitem{Stanley}
R.~Stanley, \emph{Enumerative Combinatorics, Volume I \& II}, Cambridge University Press, 1997 \& 1999.

\bibitem{Su}
C.~Su, Q.~Feng and Z.~Hu, Uniform recursive trees: Branching structure and simple random downward
walk, \emph{Journal of Mathematical Analysis and Applications} 315, 225--243, 2006.

\bibitem{Weisstein}
Eric W.~Weisstein, "Stirling Transform." From MathWorld--A Wolfram Web Resource. \url{http://mathworld.wolfram.com/StirlingTransform.html}

\bibitem{Wilf}
H.~S.~Wilf, \emph{Generatingfunctionology}, Second Edition, Academic Press, 1992.

\bibitem{Willmot}
G.~Willmot, Mixed compound Poisson distributions, \emph{ASTIN Bulletin} 16, 59--80, 1986.
\end{thebibliography}
\end{document}